\documentclass[10pt]{amsart}

\usepackage{xcolor}
\definecolor{winered}{rgb}{0.8,0,0}
\definecolor{deepblue}{rgb}{0,0,0.8}
\usepackage[colorlinks=true]{hyperref}
\hypersetup{linkcolor=winered, citecolor=deepblue}
\usepackage{amsmath}
\usepackage{amsthm}
\usepackage{amssymb}
\usepackage{mathrsfs}
\usepackage{tikz-cd}
\newtheorem{thm}{Theorem}[subsection]
\newtheorem{prop}[thm]{Proposition}
\newtheorem{cor}[thm]{Corollary}

\newtheorem{lem}[thm]{Lemma}
\newtheorem*{theorem.3.4.3}{Theorem 3.4.3}
\newtheorem*{theorem.5.2.6}{Theorem 5.2.6}
\newtheorem*{prop1}{Proposition 1}
\theoremstyle{definition}
\newtheorem{df}[thm]{Definition}
\newtheorem{rmk}[thm]{Remark}
\newtheorem{exm}[thm]{Example}

\newtheorem{const}[thm]{Construction}

\theoremstyle{remark}
\numberwithin{equation}{section}

\newcommand{\A}{\mathbb{A}}

\newcommand{\C}{\mathbb{C}}

\newcommand{\E}{\mathbb{E}}

\newcommand{\G}{\mathbb{G}}

\newcommand{\N}{\mathbb{N}}
\renewcommand{\P}{\mathbb{P}}

\newcommand{\R}{\mathbb{R}}

\newcommand{\Sphere}{\mathbb{S}}

\newcommand{\Z}{\mathbb{Z}}

\newcommand{\cA}{\mathcal{A}}

\newcommand{\cC}{\mathcal{C}}
\newcommand{\cD}{\mathcal{D}}
\newcommand{\cE}{\mathcal{E}}
\newcommand{\cF}{\mathcal{F}}
\newcommand{\cG}{\mathcal{G}}

\newcommand{\cV}{\mathcal{V}}

\newcommand{\rB}{\mathrm{B}}

\newcommand{\sX}{\mathscr{X}}

\newcommand{\et}{{\acute{e}t}}
\newcommand{\Et}{\mathbf{\acute{E}t}}
\newcommand{\isoEtAff}{\mathbf{iso\acute{E}tAff}}

\DeclareMathOperator{\Hom}{Hom}

\DeclareMathOperator{\Tor}{Tor}
\DeclareMathOperator{\Spec}{Spec}

\newcommand{\colim}{\mathop{\mathrm{colim}}}

\newcommand{\id}{\mathrm{id}}
\newcommand{\ul}{\underline}

\newcommand{\cofib}{\mathrm{cofib}}
\newcommand{\fib}{\mathrm{fib}}

\newcommand{\Set}{\mathbf{Set}}
\newcommand{\Shv}{\mathbf{Shv}}
\newcommand{\Psh}{\mathbf{Psh}}

\newcommand{\Sp}{\mathrm{Sp}}
\newcommand{\SpO}{\mathrm{Sp}^O}

\newcommand{\CAlg}{\mathrm{CAlg}}
\newcommand{\Comm}{\mathbf{Comm}}
\newcommand{\CommG}{\mathbf{Comm}_G}
\newcommand{\NAlg}{\mathrm{NAlg}}

\newcommand{\Mod}{\mathrm{Mod}}

\newcommand{\Sch}{\mathbf{Sch}}
\newcommand{\Aff}{\mathbf{Aff}}

\newcommand{\Fun}{\mathrm{Fun}}
\newcommand{\Map}{\mathrm{Map}}

\newcommand{\Nerve}{\mathrm{N}}

\newcommand{\infPsh}{\mathcal{P}\mathrm{sh}}
\newcommand{\infShv}{\mathcal{S}\mathrm{hv}}
\newcommand{\Span}{\mathrm{Span}}
\newcommand{\Fin}{\mathrm{Fin}}
\newcommand{\FinGpd}{\mathrm{FinGpd}}

\newcommand{\PSigma}{\mathcal{P}_{\Sigma}}
\newcommand{\Mack}{\mathrm{Mack}}
\newcommand{\Green}{\mathrm{Green}}
\newcommand{\Ho}{\mathrm{Ho}}
\newcommand{\Fold}{\mathrm{Fold}}
\newcommand{\res}{\mathrm{res}}
\newcommand{\tran}{\mathrm{tran}}

\newcommand{\tfib}{\mathrm{tfib}}

\newcommand{\pt}{\mathrm{pt}}

\newcommand{\EM}{\mathrm{H}}

\newcommand{\infH}{\mathcal{H}}
\newcommand{\infHpt}{\mathcal{H}_{\bullet}}
\newcommand{\infSH}{\mathcal{SH}}

\newcommand{\Bdi}{\mathrm{B}^{\mathrm{di}}}
\newcommand{\Bsigma}{\mathrm{B}^\sigma}

\newcommand{\Ncy}{\mathrm{N}^\mathrm{cy}}

\newcommand{\Ndi}{\mathrm{N}^\mathrm{di}}

\newcommand{\Nsigma}{\mathrm{N}^\sigma}

\newcommand{\THH}{\mathrm{THH}}
\newcommand{\THR}{\mathrm{THR}}
\newcommand{\THO}{\mathrm{THO}}
\newcommand{\TC}{\mathrm{TC}}
\newcommand{\TCR}{\mathrm{TCR}}

\newcommand{\sd}{\mathrm{sd}}

\newcommand{\Ztwo}{{\mathbb{Z}/2}}
\newcommand{\Cat}{\mathrm{Cat}}
\newcommand*{\KR}{\mathbf{KR}}
\newcommand*{\KO}{\mathbf{KO}}
\newcommand*{\SHG}{\mathcal{SH}^{\Ztwo}(k)}

\begin{document}
\author{Jens Hornbostel}
\address{Bergische Universit{\"a}t Wuppertal,
Fakult{\"a}t f\"ur Mathematik und Naturwissenschaften
\\
Gau{\ss}strasse 20, 42119 Wuppertal, Germany}
\email{hornbostel@math.uni-wuppertal.de}
\author{Doosung Park}
\address{Bergische Universit{\"a}t Wuppertal,
Fakult{\"a}t f\"ur Mathematik und Naturwissenschaften
\\
Gau{\ss}strasse 20, 42119 Wuppertal, Germany}
\email{dpark@uni-wuppertal.de}

\title{Real topological Hochschild homology of schemes}
\subjclass{Primary 19D55; Secondary 11E70, 16E40, 55P91}
\keywords{real topological Hochschild homology, equivariant homotopy theory, isovariant \'etale descent}
\begin{abstract}
We prove that real topological Hochschild homology $\THR$ for schemes with involution satisfies base change and descent for the $\Ztwo$-isovariant \'etale topology. As an application, we provide computations for the projective line (with and without involution) and the higher dimensional projective spaces.
\end{abstract}
\maketitle

\section{Introduction}

Hochschild and cyclic homology and their refinements $\THH$ and $\TC$ have been extensively studied for many decades, both for their own sake and for their deep connections with algebraic $K$-theory via traces, see e.g.\ \cite{BHM93}. Standard textbook references include \cite{Lo} and \cite{DGM13}. More recently, Bhatt-Morrow-Scholze \cite{BMS19} have introduced a filtration on $\THH$ and $\TC$ that is strongly related to integral $p$-adic Hodge theory. Hahn-Raksit-Wilson \cite{HRW} have provided an alternative construction of the filtration that applies to commutative ring spectra satisfying certain assumptions.

\medskip

Both topological and algebraic $K$-theory have real or hermitian refinements. This is also true for $\THH$ and $\TC$, but the serious study of their real variants $\THR$ and $\TCR$ has just started, see e.g.\ \cite{DO19}, \cite{DMPR21}, and \cite{QS}. These theories apply to very different branches in algebra and geometry: commutative rings, group rings, schemes, ring spectra, ... all with or without a non-trivial involution. Concerning the $\Ztwo$-equivariant cyclotomic trace, we understand that there is work in progress by Harpaz, Nikolaus, and Shah studying this map in the very general setting of stable Poincar\'e $\infty$-categories. (When studying traces, beware of the difference between rings, which is what most algebraic $K$-theorists and algebraic geometers look at, and ring spectra, which are the input of $\THH$ and $\TC$.)

\medskip 

The current article contributes to a better understanding of $\THR$ in algebraic geometry, although some of our results also apply to other settings. Our first main result is the base change result in Theorem
\ref{etale.5} for the isovariant \'etale topology. This then is one of the main ingredients in the proof of the following isovariant \'etale descent theorem for $\THR$.

\begin{theorem.3.4.3}
The presheaf
\[
\THR\in \infPsh(\Aff_{\Ztwo},\Sp_{\Ztwo})
\]
satisfies isovariant \'etale descent, where $\Aff_{\Ztwo}$ denotes the category of affine schemes with involutions, and $\Sp_{\Ztwo}$ denotes the $\infty$-category of $\Ztwo$-spectra.
\end{theorem.3.4.3}

In the non-equivariant case, similar results have been established for $\THH$ and $\TC$ by Weibel-Geller \cite{GW91} and Geisser-Hesselholt \cite{GH}.

\medskip

This descent theorem implies in particular that $\THR$ satisfies the equivariant Zariski-Mayer-Vietoris
property for affine schemes with involutions. This can be used to extend $\THR$ to non-affine schemes with involutions, see 
Definition \ref{etale.34}, in a way that is compatible with existing definitions of $\THH$, see Proposition \ref{etale.35}.

\medskip

Using this Zariski-Mayer-Vietoris theorem and various explicit computations of $\THR$ of products of monoid rings and maps between them, we are able to compute $\THR(X)$ for $X=\P^1,\P^{\sigma}$ and more generally $\P^n$ with trivial involution, see Theorems \ref{period.3}, \ref{period.1} and \ref{proj.2}.
Here $\P^{\sigma}$ denotes the projective line with the involution switching the homogeneous coordinates.
We recall that Blumberg and Mandell \cite{BM} compute $\THH(\P^n)$. For $\P^n$ with the trivial involution, we obtain the following:

\begin{theorem.5.2.6}
For any separated scheme with involution $X$ and integer $n\geq 0$, there is an equivalence of $\Ztwo$-spectra
\[
\THR(X\times \P^{n})
\simeq
\left\{
\begin{array}{ll}
\THR(X)
\oplus
\bigoplus_{j=1}^{\lfloor n/2 \rfloor} i_*\THH(X) &
\text{if $n$ is even,}
\\
\THR(X)
\oplus
\bigoplus_{j=1}^{\lfloor n/2 \rfloor} i_*\THH(X) \oplus \Sigma^{n(\sigma-1)}\THR(X) &
\text{if $n$ is odd.}
\end{array}
\right.
\]
\end{theorem.5.2.6}

For a $\Ztwo$-spectra $E$, $i_*E:=E \oplus E$ with the obvious involution, see \eqref{etale.8.1}.
We refer to Remarks \ref{krthr}, \ref{krpn}, and \ref{thrpn} for a comparison with hermitian and real algebraic $K$-theory.
\medskip

Recall that unlike algebraic and hermitian $K$-theory, $\THH$ and $\THR$ are not $\A^1$-invariant even on regular schemes. On the other hand, $\THH$ and $\TC$ do extend to log schemes, and using descent, trivial $\P^1$-bundle formula, and computations on the log schemes $(\P^n,\P^{n-1})$, can be shown to be representable in the log-variant of the $\P^1$-stable motivic Morel-Voevodsky homotopy category. This is done in the very recent joint work \cite{BPO2} of the second author with Federico Binda and Paul Arne {\O}stv{\ae}r.
We refer to \cite{BPOCras} for an overview of their work.
The results of this article are expected to provide most of the necessary ingredients for showing
stable representability of $\THR$ or at least its fixed points $\THO$ in the corresponding equivariant log-homotopy categories. Furthermore, the work of Quigley-Shah \cite{QS} allows us to extend many results of this article including Theorems \ref{etale.7} and \ref{proj.2} to $\TCR$. The second author hopes to carry out motivic representability of $\THR$ and $\TCR$ in forthcoming work.

\medskip

We conclude with a short overview of the different sections. Section \ref{sec2} discusses some generalities about commutative ring spectra with involutions and their $\THR$.
This includes the study of several functors between $\Ztwo$-equivariant and non-equivariant categories, including the norm and equivariant notion of flatness.
The necessary background on $G$-stable equivariant homotopy theory for finite groups $G$, both ${\infty}$- and model categorical, is provided in Appendix \ref{secA}. 
Section \ref{sec3} reviews and extends several definitions for Grothendieck topologies on schemes with involutions, notably Thomason's isovariant \'etale topology. Base change and descent for $\THR$ are established in subsections \ref{sec3.2} and \ref{sec3.4}.
Note that the proof of \'etale base change in subsection \ref{sec3.2} heavily relies on various results about Green functors established in subsections \ref{seca.4} and \ref{seca.5}.
Section \ref{sec4} recollects and extends material on real and dihedral nerves, which is crucial when computing $\THR$ of spherical monoid rings and of monoid rings over Eilenberg MacLane spectra, and where the monoids may have involutions. These monoid ring computations together with isovariant Zariski descent are then used in the computations for projective spaces in section \ref{sec5}.

\medskip

\textit{Acknowledgements.} This research was conducted in the framework of the DFG-funded research training group GRK 2240: {\em Algebro-
Geometric Methods in Algebra, Arithmetic and Topology}. We are pleased to thank the referee for a very precise and helpful report.

\medskip

\textbf{Please also note the addendum at the end of this article.}

\section{Real topological Hochschild homology of rings with involutions}
\label{sec2}

\begin{rmk}\label{krthr}
$\THR$ is to real algebraic $K$-theory $\KR$ what $\THH$ is to algebraic $K$-theory.
Hence we recall some recent results on {\em real algebraic $K$-theory} $\KR$. This is a $\Ztwo$-equivariant motivic spectrum constructed in \cite{HuKO} and \cite{Ca}. We recover hermitian $K$-theory  when restricting $\KR$ to schemes with trivial $\Ztwo$-action. (This is one incarnation of the philosophy ``the fixed points of $\KR$ are hermitian $K$-theory'' in the world of presheaves on $\Ztwo$-schemes). Forgetting the action, we recover Voevodsky's algebraic $K$-theory spectrum $KGL$. As observed in \cite[section 7.3]{Xi}, see also \cite{Ca}, Schlichting's techniques generalize to show that $\KR$ is representable in $\SHG$. 
In particular, $\KR$ satisfies equivariant Nisnevich descent. Here and below, following \cite{Ca} and \cite{HuKO}, we consider $\KR$ as a motivic spectrum with respect to the circle $T^{\rho} \simeq \P^1 \wedge \P^{\sigma}$, where $\rho$ is the regular representation of $\Ztwo$, and $\P^{\sigma}$ denotes $\P^1$ with involution switching the homogeneous coordinates. It might be more consistent with the notations for $\THH$ to denote the (motivic) hermitian $K$-theory spectrum on schemes with involution by $\KO$, and to reserve the notation $\KR$ for a (motivic) spectrum with involution whose fixed points are $\KO$. For further results on $\KR$ and projective spaces we refer to Remarks \ref{krpn} and \ref{thrpn} below. 
\end{rmk}

Throughout this section, we fix morphisms of finite groupoids
\begin{equation}
\label{etale.0.2}
\pt
\xrightarrow{i}
\rB (\Z/2)
\xrightarrow{p}
\pt,
\end{equation}
where $\rB G$ denotes the finite groupoid consisting of a single set $*$ with $\Hom_{\rB G}(*,*):=G$.
According to section \ref{orth.sec}, we often use the alternative notation
\[
N^{\Ztwo}=i_{\otimes},
\;
\Phi^{\Ztwo}=p_{\otimes},
\;
\iota=p^*,
\text{ and }
(-)^G=p_*
\]
instead of the notation in section \ref{equivinfty.sec}.
In particular, all functors in the sequel are $\infty$-functors, which admit lifts to Quillen functors between model categories.

We then have adjoint pairs
\begin{equation}
\label{etale.0.1}
i^*: \Sp_{\Ztwo} \rightleftarrows \Sp : i_*
\text{ and }
\iota : \Sp \rightleftarrows \Sp_{\Ztwo} : (-)^G
\end{equation}
and functors 
\[
N^{\Ztwo} : \Sp \to \Sp_{\Ztwo}
\text{ and }
\Phi^{\Ztwo} : \Sp_{\Ztwo} \to \Sp.
\]
By \cite[Theorem 7.12]{Schwede}, the pair of functors $(i^*,\Phi^{\Ztwo})$ is conservative.
We will use this fact frequently.
We also have adjoint pairs
\begin{gather*}
N^{\Ztwo}:\CAlg \rightleftarrows \NAlg_{\Ztwo} : i^*,
\text{ }
i^*:\NAlg_{\Ztwo}\rightleftarrows \CAlg : i_*,
\\
\text{ and }
\iota : \CAlg \rightleftarrows \CAlg_{\Ztwo} : (-)^{\Ztwo}
\end{gather*}
and a colimit preserving functor
\[
\Phi^{\Ztwo}:\CAlg_{\Ztwo}\to \CAlg. 
\]
We refer sections \ref{equivinfty.sec} and \ref{orth.sec} for further properties of all these functors.

We now define real topological Hochschild homology for commutative ring spectra with involution.

\subsection{Definition of \texorpdfstring{$\THR$}{THR}}

\begin{df}
\label{etale.14}
Suppose $A\in \NAlg_{\Ztwo}$. For abbreviation, we set
\[
A^{\wedge \Ztwo}
:=
N^{\Ztwo}i^* A.
\]
Since $N^{\Ztwo}$ is left adjoint to $i^*$, we have the counit map $A^{\wedge \Ztwo}\to A$.
We use this map to define
\[
\THR(A)
:=
A\wedge_{A^{\wedge \Ztwo}} A\in \NAlg_{\Ztwo},
\]
which is called the  \emph{real topological Hochschild homology of $A$}.
The first $\wedge$ in the formulation of $\THR$ is the pushout in $\NAlg_{\Ztwo}$.
Under a certain flatness condition, this is equivalent to the B\"okstedt model of the real topological Hochschild homology, see \cite[Theorem, p.\ 65]{DMPR21}.
Note also that it is possible to define $\THR$ for $\Ztwo$-spectra with slightly less structure, e.g. for $\mathbb{E}_{\sigma}$-algebras as in \cite{AKGH}.

The map to the second smash factor in $A\wedge_{A^{\wedge \Ztwo}} A$ gives a canonical map
\begin{equation}
\label{etale.14.1}
A\to \THR(A).
\end{equation}

In analogy with hermitian $K$-theory $\KO$ and real algebraic $K$-theory $\KR$, we  define
\[
\THO(A) = (\THR(A))^{\Ztwo} \simeq (A \wedge_{N^{\Ztwo} i^*A} A)^{\Ztwo} \in \CAlg
\]
for $A \in \NAlg_{\Ztwo}$. 
For $B\in \NAlg=\CAlg$, recall that the \emph{topological Hochschild homology of $B$} is defined to be
\[
\THH(B)
:=
B\wedge_{B\wedge B} B \in \CAlg.
\]
which in turn generalizes the classical \cite[Proposition 1.1.13]{Lo} from rings to ring spectra.
\end{df}

\begin{df}
\label{etale.27}
Let $\cC$ be a category.
An object $X$ of $\Fun(\rB (\Ztwo),\cC)$ is called an \emph{object of $\cC$ with involution}.
Explicitly, $X$ is an object of $\cC$ equipped with an automorphism $w\colon X\to X$ such that $w\circ w=\id$.

In particular, we have the notions of commutative rings with involutions, commutative monoids with involutions, etc.
\end{df}

We do not discuss definitions of $\mathrm{HR}(A)$ refining $\mathrm{HH}(A)$ and comparison results between $\THR(A)$ and $\mathrm{HR}(A)$, similarly to e.g.\ \cite[Proposition IV.4.2]{NS}, but see Remark \ref{remHRloday} below.  
The adjoint functors $\ul{\pi}_0$ and $\EM$, which are studied in the appendix, preserve many of the adjunctions we study, see e.g.\ Definition \ref{etale.17}.

\begin{prop}
\label{etale.10}
For $A\in \NAlg_{\Ztwo}$ and $B\in \CAlg$, there exist canonical equivalences
\[
\THH(i^*A)
\simeq
i^*\THR(A)
\text{ and }
\THR(N^{\Ztwo}B)
\simeq
N^{\Ztwo}\THH(B).
\]
Hence there exists a canonical equivalence
\[
\THR(A^{\wedge \Ztwo})
\simeq
\THR(A)^{\wedge \Ztwo}.
\]
\end{prop}

In particular, the first equivalence implies $i^*\THR(\iota B) \simeq \THH(B)$, using (2) of Proposition \ref{orth.5}.  

\begin{proof}
Since $N^{\Ztwo}$ and $i^*$ preserve colimits, we have equivalences
\[
\THH(i^*A)
\simeq
i^*A\wedge_{i^*A\wedge i^*A}i^*A
\simeq
i^*A \wedge_{i^*N^{\Ztwo}i^*A} i^*A
\simeq
i^*\THR(A)
\]
and
\[
\THR(N^{\Ztwo}B)
\simeq
N^{\Ztwo}B \wedge_{N^{\Ztwo}i^*N^{\Ztwo} B} N^{\Ztwo} B
\simeq
N^{\Ztwo}(B\wedge_{B\wedge B} B)
\simeq
N^{\Ztwo}\THH(B)
\]
by Proposition \ref{orth.5}(6).
\end{proof}

\begin{prop}
\label{etale.9}
For $B\in \CAlg$, there exists a canonical equivalence
\begin{equation}
\label{etale.9.1}
\THR(i_*B)\xrightarrow{\simeq} i_*\THH(B).
\end{equation}
\end{prop}
\begin{proof}
The composite
\begin{equation}
\THR(i_*B)
\to
i_*i^*\THR(i_*B)
\xrightarrow{\simeq}
i_*\THH(i^*i_*B)
\to
i_*\THH(B)
\end{equation}
defines \eqref{etale.9.1}, where the first (resp.\ third) map is induced by the unit (resp.\ counit).
Proposition \ref{etale.10} shows that the induced map
\[
i^*\THR(i_*B)
\to
i^*i_*\THH(B)
\]
is an equivalence.
By \cite[Theorem 7.12]{Schwede} (note that $i^*=\Phi^e$), it remains to show that the induced map
\[
\Phi^{\Ztwo}\THR(i_*B)
\to
\Phi^{\Ztwo}i_*\THH(B)
\]
is an equivalence.
The right hand side is equivalent to $0$ by Proposition \ref{orth.5}(3),(5).
On the other hand, we have equivalences
\[
\Phi^{\Ztwo}\THR(i_*B)
\simeq
\Phi^{\Ztwo}(i_*B) \wedge_{i^*i_*B }\Phi^{\Ztwo}(i_*B) 
\simeq
0 \wedge_{i^*i_*B }0,
\]
which is equivalent to $0$ too.
\end{proof}

Applying $i^*$ and  Proposition \ref{orth.5}(2), we obtain $\THO(i_*B) \simeq \THH(B)$,
which compares nicely with the well-known $\KO(X\amalg X)\simeq K(X)$ for schemes $X$, where the involution on $X\amalg X$ switches the components.

\begin{prop}
\label{etale.1}
Let $R\to A,B$ be maps in $\NAlg_{\Ztwo}$.
Then there exists a canonical equivalence
\[
\THR(A)\wedge_{\THR(R)}\THR(B)
\simeq
\THR(A\wedge_R B).
\]
\end{prop}
\begin{proof}
Both $N^{\Ztwo}$ and $i^*$ preserve colimits.
Hence we obtain a canonical equivalence
\[
A^{\wedge \Ztwo}\wedge_{R^{\wedge \Ztwo}}B^{\wedge \Ztwo}
\simeq
(A\wedge_R B)^{\wedge \Ztwo}.
\]
On the other hand, there are canonical equivalences
\begin{align*}
\THR(A)\wedge_{\THR(R)}\THR(B)
&\simeq
(A\wedge_{A^{\wedge \Ztwo}} A)
\wedge_{R\wedge_{R^{\wedge \Ztwo}} R}
(B\wedge_{B^{\wedge \Ztwo}} B)
\\
&\simeq
(A\wedge_R B)
\wedge_{A^{\wedge \Ztwo}\wedge_{R^{\wedge \Ztwo}}B^{\wedge \Ztwo}} (A\wedge_R B).
\end{align*}
Combine the two equations to obtain the desired equivalence.
\end{proof}

\subsection{Mackey functors for \texorpdfstring{$\Ztwo$}{Z/2}}

We refer to the appendix for a general discussion of Mackey and Green functors for finite groups $G$. We now restrict to the case $G=\Ztwo$.

\begin{exm}
\label{etale.29}
According to \cite[Example 4.38]{Schwede}, a Mackey functor $C$ for $G:=\Ztwo$ can be described as a diagram
\[
\begin{tikzcd}
C(G/e)\ar[loop below,"w"]\ar[r,shift left=0.5ex,"\tran"]\ar[r,shift right=0.5ex,leftarrow,"\res"']&
C(G/G),
\end{tikzcd}
\]
where $C(G/e)$ is an abelian group with an involution $w$, $C(G/G)$ is an abelian group with the trivial involution, $\res$ and $\tran$ are homomorphisms of abelian groups with involutions, and the equality
\[
\res\circ \tran = \id+w
\]
is satisfied (i.e.\ the \emph{double coset formula} holds).
A \emph{morphism of Mackey functors $C\to D$} is a diagram of abelian groups with involutions
\[
\begin{tikzcd}
C(G/G)
\ar[r]
\ar[d,"\res"',shift right=0.5ex]
\ar[d,"\tran",shift left=0.5ex,leftarrow]
&
D(G/G)
\ar[d,"\res"',shift right=0.5ex]\ar[d,"\tran",shift left=0.5ex,leftarrow]
\\
C(G/e)
\ar[r]
\ar[loop,out=190, in=170,looseness=6]
&
D(G/e)
\ar[loop,out=350, in=10,looseness=6]
\end{tikzcd}
\]
such that the horizontal homomorphisms commute with $\tran$ and $\res$.
\end{exm}

\begin{exm}
If $M$ is an abelian group with involution, then we can associate the Mackey functor
\[
\begin{tikzcd}
M\ar[r,shift left=0.5ex,"\tran"]\ar[loop below]\ar[r,shift right=0.5ex,leftarrow,"\res"']&
M^{\Ztwo},
\end{tikzcd}
\]
where $\tran$ maps $x\in M$ to $x+w(x)$, and $\res$ is the inclusion.
In this way, we obtain a fully faithful functor from the category of abelian groups with involutions
to the category of Mackey functors $\Mack_{\Ztwo}$.
We often regard an abelian group with involution as a Mackey functor if no confusion seems likely to arise.
\end{exm}

\begin{lem}
\label{etale.28}
Let $A$ be a Green functor for $G:=\Ztwo$, and let $M$ and $L$ are $A$-modules.
If $L$ is associated 
(in the sense of the previous example) 
with an abelian group with involution, then the induced map
\begin{equation}
\label{etale.28.1}
\Hom_{\Mod_A}(M,L)
\to
\Hom_{\Mod_{A(G/e)}}(M(G/e),L(G/e))
\end{equation}
is an isomorphism.
\end{lem}
\begin{proof}
Let $f\colon M(G/e)\to L(G/e)$ be a homomorphism of $A(G/e)$-modules.
Then the image of $f\circ \res$ is in $L(G/e)^{\Ztwo}$.
Hence there exists a unique homomorphism of $A(G/G)$-modules $g\colon M(G/G)\to L(G/G)$ such that in the diagram
\[
\begin{tikzcd}
M(G/G)
\ar[r,"g"]
\ar[d,"\res"',shift right=0.5ex]
\ar[d,"\tran",shift left=0.5ex,leftarrow]
&
L(G/G)
\ar[d,"\res"',shift right=0.5ex]
\ar[d,"\tran",shift left=0.5ex,leftarrow]
\\
M(G/e)
\ar[r,"f"]
\ar[loop,out=190, in=170,looseness=6]
&
L(G/e)
\ar[loop,out=350, in=10,looseness=6]
\end{tikzcd}
\]
the pair $(f,g)$ commutes with $\res$.
We have
\[
\res\circ g\circ \tran
=
f\circ \res\circ \tran
=
f\circ (\id+w)
=
(\id +w)\circ f
=
\res\circ \tran\circ f.
\]
Since $\res$ for $L$ is injective, we deduce that the pair $(f,g)$ commutes with $\tran$.
This constructs an inverse of \eqref{etale.28.1}.
\end{proof}

\subsection{Equivariant Eilenberg-MacLane spectra}
In this subsection, we explain basic properties of equivariant Eilenberg-MacLane spectra.
We also explain how to define $\THR$ of commutative rings.

\begin{df}
\label{etale.32}
Let $\cC$ be a category (not an $\infty$-category).
We have the functors
\[
\cC
\xrightarrow{\iota}
\Fun(\rB(\Ztwo),\cC)
\xrightarrow{i^*}
\cC
\]
induced by \eqref{etale.0.2}.
Let $(-)^{\Ztwo}$ denote the right adjoint of $\iota$ if it exists.

Here, we give some examples.
For a commutative ring $A$, $\iota A$ is the commutative ring $A$ with the trivial involution.
For an $A$-module $M$, $\iota M$ is the $\iota A$-module $M$ with the trivial involution.
By abuse of notation we sometimes denote the constant Mackey functors by $\iota A$ and $\iota M$ as well.

For a commutative ring $B$ with involution, $i^*B$ is the commutative ring obtained by forgetting the involution, and $B^{\Ztwo}$ is the $\Ztwo$-fixed point ring. For a $B$-module $L$, $i^*L$ is the $i^*B$-module obtained by forgetting the involution.
\end{df}

\begin{df}
\label{etale.17}
For an abelian group $M$ with involution, we regard $M$ as a Mackey functor, and take the functor \eqref{t.3.3} to obtain the equivariant Eilenberg-MacLane spectrum $\EM M$.
There are canonical equivalences
\begin{equation}
\label{etale.17.1}
i^*\EM M \simeq \EM i^* M
\text{ and }
(\EM M)^{\Ztwo}
\simeq
\EM (M^\Ztwo).
\end{equation}

For a commutative ring $A$ with involution, we can regard $\EM A$ as an object of $\NAlg_{\Ztwo}$ as explained in \cite[Example 11.12]{Schwede}.
\end{df}

Note that for a given commutative ring $B$ the commutative ring spectra $\EM \iota B$ and $\iota \EM B$ are quite different.
E.g., applying $\ul{\pi}_0$ to the first one yields the constant Mackey functor associated with $B$ whereas for the second one a tensor product over $\Z$ with the Burnside ring Mackey functor of $\Ztwo$ appears.

\begin{prop}
\label{etale.33}
Let $M$ be an abelian group.
Then there is a canonical equivalence
\begin{equation}
\label{etale.33.1}
\EM (M^{\oplus \Ztwo})
\simeq
i_*\EM M,
\end{equation}
where $M^{\oplus \Ztwo}$ denotes the abelian group $M\oplus M$ with the involution given by $(x,y)\mapsto (y,x)$.
The equivalence \eqref{etale.33.1} can be promoted to an equivalence in $\NAlg_{\Ztwo}$ if $M$ is a commutative ring.
\end{prop}
\begin{proof}
There is an equivalence
\begin{equation}
\label{etale.33.3}
i^* \EM(M^{\oplus \Ztwo})\simeq \EM (M\oplus M)
\end{equation}
by \eqref{etale.17.1}.
Compose this with the map $\EM(M\oplus M)\to \EM M$ induced by the summation homomorphism $M\oplus M \to M$, and then we construct \eqref{etale.33.1} by adjunction.
We need to show that the induced map
\[
\Hom_{\Ho(\Sp_{\Ztwo})}(\Sigma^n\Sigma^\infty X_+,\EM(M^{\oplus \Ztwo}))
\to
\Hom_{\Ho(\Sp_{\Ztwo})}(\Sigma^n\Sigma^\infty X_+,i_*\EM M)
\]
is an isomorphism for $X=\Ztwo,e$ and integer $n\in \Z$.
If $n\neq 0$, then both sides are vanishing.
Assume $n=0$.
More concretely, it remains to show that the composite of the induced maps
\begin{equation}
\label{etale.33.2}
\begin{split}
\Hom_{\Ho(\Sp_{\Ztwo})}(\Sigma^\infty X_+,\EM(M^{\oplus \Ztwo}))
\to &
\Hom_{\Ho(\Sp)}(i^*\Sigma^\infty X_+,i^*\EM(M^{\oplus \Ztwo}))
\\
\to &
\Hom_{\Ho(\Sp)}(i^*\Sigma^\infty X_+,\EM M)
\end{split}
\end{equation}
is an isomorphism.

If $X=\Ztwo$, then \eqref{etale.33.2} can be written as the homomorphisms
\[
M\oplus M \to M\oplus M\oplus M\oplus M \to M\oplus M
\]
given by $(x,y)\mapsto (x,0,y,0)$ and $(x,y,z,w)\mapsto (x+y,z+w)$.
The composite is an isomorphism.
If $X=e$, then \eqref{etale.33.2} can be written as the homomorphisms
\[
M\to M\oplus M \to M
\]
given by $x\mapsto (x,0)$ and $(x,y)\mapsto x+y$.
The composite is also an isomorphism.

If $M$ is a commutative ring, then \eqref{etale.33.3} is an equivalence in $\NAlg_{\Ztwo}$.
Hence we obtain \eqref{etale.33.1} as an equivalence in $\NAlg_{\Ztwo}$.
\end{proof}

\begin{df}
\label{etale.12}
For a commutative ring $A$ with involution, we set
\[
\THR(A):=\THR(\EM A).
\]
From the map \eqref{etale.14.1}, we see that $\THR(A)$ is an $\EM A$-module.
\end{df}

Recall from \eqref{Mack.2.3} that $\ul{\pi}_0$ of an equivariant spectrum is a Mackey functor.

\begin{prop}
\label{etale.19}
For every commutative ring $A$, the morphism of Mackey functors
\begin{equation}
\label{etale.19.1}
\iota A
\to
\ul{\pi}_0(\THR(\iota A))
\end{equation}
induced by \eqref{etale.14.1} is an isomorphism.
\end{prop}
\begin{proof}
By \cite[Theorem 5.1]{DMPR21}, the morphism \eqref{etale.19.1} can be described as a diagram
\[
\begin{tikzcd}
A
\ar[r,"\alpha"]
\ar[d,"\res"',shift right=0.5ex]
\ar[d,"\tran",shift left=0.5ex,leftarrow]
&
(A\otimes A)/T,
\ar[d,"\res"',shift right=0.5ex]
\ar[d,"\tran",shift left=0.5ex,leftarrow]
\\
A
\ar[r,"\id"]
\ar[loop left]
&
A
\ar[loop right]
\end{tikzcd}
\]
where $\alpha(x)=1\otimes x$ for $x\in A$, and $T$ is the subgroup generated by $ax\otimes b-a\otimes xb$ for $a,b,x\in A$.
The description of $T$ means that $\alpha$ is an isomorphism.
It follows that \eqref{etale.19.1} is an isomorphism.
\end{proof}

\subsection{Norm functors and flat modules}

The purpose of this subsection is to prove Proposition \ref{etale.25}, which is one ingredient of the proof of Theorem \ref{etale.5}.

\begin{df}\label{defnztwoam}
Suppose $A$ is a commutative ring and $M$ is an $A$-module.
Let $N_A^{\Ztwo} M$ denote the $\iota A$-module $M\otimes_A M$ whose $\Ztwo$-action is given by $a\otimes b\mapsto b\otimes a$.
If $A=\Z$, we simply write $N^{\Ztwo}M$ instead of $N_{\Z}^{\Ztwo}M$.

Observe that there is an isomorphism $N_A^{\Ztwo} A\cong \iota A$ of commutative rings with involutions.
We prefer to use the notation $\iota A$ instead of $N_A^{\Ztwo}A$ for brevity.
\end{df}

\begin{prop}
\label{etale.25}
Let $M$ be an $A$-module, where $A$ is a commutative ring.
Then there exists a canonical map of $\EM \iota A$-modules
\begin{equation}
\label{etale.25.2}
N^{\Ztwo}\EM M\wedge_{N^{\Ztwo}\EM A} \EM \iota A
\to
\EM(N_A^{\Ztwo} M).
\end{equation} If $M$ is flat, then this map is an equivalence.
\end{prop}
\begin{proof}
First, we have a map $N^{\Ztwo} \EM A \to \EM \iota A$ using the adjunction $(N^{\Ztwo},i^*)$, that $i^*$ commutes with $\EM$, and that $i^*\iota\simeq id$.
We have the $\iota A$-module structure on $N_A^{\Ztwo} M$ given by $a(x\otimes y)=ax\otimes y$ for all $a\in A$ and $x\otimes y\in N_A^{\Ztwo} M$.
Hence we can regard $\EM(N_A^{\Ztwo} M)$ as an $\EM \iota A$-module.
By \cite[Proposition 4.6.2.17]{HA}, to construct \eqref{etale.25.2}, it suffices to construct a map of $N^{\Ztwo}\EM A$-modules
\begin{equation}
\label{etale.25.3}
N^{\Ztwo} \EM M
\to
\EM (N_A^{\Ztwo} M).
\end{equation}

By Proposition \ref{t.5}, $N^{\Ztwo}\EM A$ and $N^{\Ztwo} \EM M$ are $(-1)$-connected.
From the adjoint pairs \eqref{Green.4.1} and \eqref{Green.4.3}, we obtain a map in $\CAlg_{\Z/2}$
\[
N^{\Ztwo}\EM A \to \EM \ul{\pi}_0(N^{\Ztwo}\EM A)
\]
and a map of $N^{\Ztwo}\EM A$-modules
\[
N^{\Ztwo}\EM M \to \EM \ul{\pi}_0(N^{\Ztwo}\EM M).
\]
Hence to construct \eqref{etale.25.3}, it suffices to construct a map of $\ul{\pi}_0(N^{\Ztwo} \EM A)$-modules
\[
\ul{\pi}_0(N^{\Ztwo} \EM M)
\to
N_A^{\Ztwo}M.
\]
By Lemma \ref{etale.28} and Definition \ref{defnztwoam}, this is equivalent to constructing a morphism of $\pi_0(i^*N^{\Ztwo}\EM A)$-modules
\begin{equation}
\label{etale.25.4}
\pi_0(i^*N^{\Ztwo}\EM M)
\to
M\otimes_A M,
\end{equation}
i.e., a morphism of $A\otimes A$-modules $M\otimes M\to M\otimes_A M$ since $i^*N^{\Ztwo}\simeq (-)^{\wedge 2}$ by Proposition \ref{orth.5}(6).
The canonical assignment $x\otimes y\to x\otimes y$ finishes the construction.

The class of $A$-modules such that \eqref{etale.25.2} is an equivalence is closed under filtered colimits.
By Lazard's theorem, every flat $A$-module is a filtered colimit of finitely generated free $A$-modules.
Hence to show that \eqref{etale.25.2} is an equivalence if $M$ is flat, we may assume $M=A^n$.
In this case, there is an equivalence
\[
N^{\Ztwo} \EM M
\simeq
V_n \wedge N^{\Ztwo} \EM A,
\]
where $V_n$ is the set $[n]\times [n]$ with the $\Ztwo$-action given by $(a,b)\mapsto (b,a)$.
Hence we obtain equivalences
\[
N^{\Ztwo}\EM M\wedge_{N^{\Ztwo}\EM A} \EM \iota A
\simeq
V_n \wedge \EM \iota A
\simeq
\EM(V_n \times \iota A).
\]
Combining this with \eqref{etale.25.2} we obtain a map
\begin{equation}
f
\colon
\EM(V_n\times \iota A)
\to
\EM(N_A^{\Ztwo} A^n).
\end{equation}
To show that $f$ is an equivalence, it suffices to show that $\ul{\pi}_0(f)$ is an equivalence.
Since $V_n \times \iota A$ and $N_A^{\Ztwo} A^n$ are rings with involutions, by Lemma \ref{etale.28} it suffices to show that $\pi_0(f)$ is an equivalence of modules over rings with involution.
Hence to show that \eqref{etale.25.2} is an equivalence, it suffices to show that it is an equivalence after applying $\pi_0$, i.e., the induced map
\[
g
\colon
(M\otimes M)\otimes_{A\otimes A} A
\to
M\otimes_A M
\]
is an isomorphism.
From the description of \eqref{etale.25.4}, we see that $g$ is given by
\[
g((x\otimes y)\otimes a)=ax\otimes y
\]
for $x,y\in M$ and $a\in A$.
One can readily check that this $g$ is an isomorphism.
\end{proof}

\section{Descent properties of \texorpdfstring{$\THR$}{THR}}
\label{sec3}

\subsection{Some equivariant topologies}

We refer to \cite{HeKO} for the definition of the stable $\Ztwo$-equivariant motivic homotopy category $\SHG$, which is compatible with the later work of Hoyois \cite{Hoy}, but differs from Hermann's \cite{Herrmann}. See his Corollary 2.13 and Example 3.1, as well as \cite[Example 2.16 and section 6.1]{HeKO}, for a comparison. The following definitions are taken from  \cite{HeKO}.
\emph{Throughout this section, we assume that $G$ is an abstract finite group, which we will identify with its associated finite group scheme over a fixed base scheme.} We are mostly interested in the case $G=\Ztwo=C_2$.

\begin{df}
\label{equitop.1}
Let $x$ be a point of a $G$-scheme $X$.
The \emph{set-theoretic stabilizer of $X$ at $x$} is defined to be
\[
S_x
:=
\{g\in G : gx=x\}.
\]
The \emph{scheme-theoretic stabilizer of $X$ at $x$} is defined to be
\[
G_x
:=
\ker(S_x\to \mathrm{Aut}(k(x))).
\]
Let $f\colon Y\to X$ be an equivariant morphism of $G$-schemes.
We say that
\begin{enumerate}
\item[(i)] $f$ is \emph{(equivariant) \'etale} if its underlying morphism of schemes is \'etale,
\item[(ii)] $f$ is an \emph{equivariant \'etale cover} if $f$ is \'etale and surjective,
\item[(iii)] $f$ is \emph{isovariant} if for every point $y\in Y$, the induced homomorphism $G_y\to G_{f(y)}$ is an isomorphism,
\item[(iv)] $f$ is an \emph{isovariant \'etale cover} if $f$ is isovariant and an equivariant \'etale cover,
\item[(v)] $f$ is a \emph{fixed point \'etale cover} if it is an \'etale cover and for every point $x\in X$, there exists a point $y\in f^{-1}(x)$ such that $G_x\simeq G_y$.
\item[(vi)] $f$ is a \emph{equivariant Nisnevich cover} if 
$f$ is an \'etale cover and for every point $x\in X$, there exists a point $y\in f^{-1}(x)$ such that $k(x)\simeq k(y)$ and $S_x\simeq S_y$.

\end{enumerate}
These covers define \emph{equivariant  \'etale, isovariant \'etale, fixed point \'etale, and equivariant Nisnevich topologies} on the category of $G$-schemes.

The isovariant \'etale topology
was first studied by Thomason \cite{Tho88}, see e.g. \cite[section 6.1]{HeKO}. By \cite[Remark 3.1]{Ryd13}, isovariant is the same as "fixed-point reflecting", compare Definition 3.3 of loc.\ cit.. The equivariant Nisnevich topology is due to Voevodsky \cite{Del09}.

The discussion in \cite[p.\ 1223]{HeKO} shows that the equivariant Nisnevich topology is coarser than the fixed point \'etale topology.
As observed in the proof of \cite[Corollary 6.6]{HeKO}. the fixed point \'etale topology is equivalent to the isovariant \'etale topology.
Hence we have the following inclusions of topologies:
\begin{align*}
\text{(equivariant Nisnevich)}
\subset &
\text{(fixed point \'etale)}
\\
\simeq &
\text{(isovariant \'etale)}
\subset
\text{(equivariant \'etale).}
\end{align*}
\end{df}

For a $G$-scheme $X$, let $X/G$ denote the geometric quotient, which is an algebraic space.
For the existence, see e.g.\ \cite[Corollary 5.4]{Ryd13}.
If $S$ is a locally noetherian scheme with the trivial $G$-action and $X\to S$ is a quasi-projective $G$-equivariant morphism, then $X/G$ is representable by an $S$-scheme according to \cite[Th\'eor\`eme V.7.1]{SGA3}.
If $A$ is a commutative ring with $G$-action, then there is a canonical isomorphism $\Spec(A)/G\simeq \Spec(A^G)$, see e.g.\ \cite[Theorem 4.1]{Ryd13}.

\begin{prop}
\label{equitop.2}
Let $f\colon Y\to X$ be a separated isovariant \'etale morphism of $G$-schemes.
Then the quotient morphism $Y/G\to X/G$ of algebraic spaces is \'etale, and the induced square
\[
\begin{tikzcd}
Y\ar[d]\ar[r,"f"]&
X\ar[d]
\\
Y/G\ar[r]&
X/G
\end{tikzcd}
\]
is cartesian.
\end{prop}
\begin{proof}
This appears on \cite[p.\ 1225]{HeKO}.
\end{proof}

\begin{df}
\label{equitop.3}
For a $G$-scheme $X$, let $X_{iso\et}$ denote the small isovariant \'etale site with the isovariant \'etale coverings $Y\to X$ of $G$-schemes. 
\end{df}

\begin{prop}[Thomason]
\label{equitop.4}
Let $X$ be a $G$-scheme.
Then there exists an equivalence of sites
\[
(X/G)_{\et}
\xrightarrow{\simeq}
X_{iso\et}
\]
sending any $X/G$-scheme $Y$ to $Y\times_{X/G} X$.
\end{prop}
\begin{proof}
We refer to \cite[Proposition 6.11]{HeKO}.
\end{proof}

\begin{df}
\label{equitop.5}
Let $X$ be a separated $G$-scheme.
The presheaf $X^G$ on the category of separated schemes $\Sch$ is defined to be
\[
X^G(Y)
:=
\Hom_{\Sch}(Y,X)^G
\]
for $Y\in \Sch$.
By \cite[Proposition 9.2 in Expos\'e XII]{SGA3}, $X^G$ is representable by a closed subscheme of $X$.
Furthermore,
the points of the topological space underlying the scheme $X^G$ are in canonical bijection with the points of the topological space of fixed points (recall $G$ is finite).
\end{df}

\subsection{Isovariant \'etale base change}
\label{sec3.2}

\begin{lem}
\label{etale.26}
Let $A\to B$ be an \'etale homomorphism of commutative rings.
Then the map $\EM(m)\colon \EM(N_A^{\Ztwo} B)\to \EM(\iota B)$ 
induced by the multiplication map
\[
m\colon N_A^{\Ztwo} B\to \iota B, 
\;
x\otimes y \mapsto xy
\]
is flat.
\end{lem}
\begin{proof}
Let $C$ be the kernel of $m$, which is an ideal of $B\otimes_A B$ with the induced involution.
Since $A\to B$ is \'etale, the diagonal morphism of schemes $\Spec(B)\to \Spec(B)\times_{\Spec(A)}\Spec(B)$ is an open and closed immersion by the implication a)$\Rightarrow$b) in \cite[Corollaire IV.17.4.2]{EGA}.
This implies that the ring structure on $B\otimes_A B$ makes a ring structure on $C$.
Hence we have an isomorphism of commutative rings with involutions $N_A^{\Ztwo}B\cong \iota B\times C$.
We set $e:=(1.0)\in \iota B\times C$.
There is an isomorphism $N_A^{\Ztwo}B[1/e]\cong \iota B$, which gives an isomorphism
\[
\colim(N_A^{\Ztwo}B \xrightarrow{\cdot e}N_A^{\Ztwo}B \xrightarrow{\cdot e} \cdots)
\cong
\iota B.
\]
Hence $\iota B$ is a filtered colimit of free $N_A^{\Ztwo}B$-modules.
By Proposition \ref{flat.4}, $m$ is flat.
It follows that $\EM(m)$ is flat too.
\end{proof}

Let $A\to B$ be a map in $\NAlg_{\Ztwo}$.
Then we have the commutative square in $\NAlg_{\Ztwo}$
\[
\begin{tikzcd}
A\ar[d]\ar[r]&
\THR(A)\ar[d]
\\
B\ar[r]&
\THR(B),
\end{tikzcd}
\]
where the horizontal maps are given by \eqref{etale.14.1}.
Hence we obtain an induced map
\begin{equation}
\THR(A)\wedge_A B\to \THR(B)
\end{equation}
in $\NAlg_{\Ztwo}$.
We will now study this map for $\EM \iota A \to \EM \iota B$.

Note that the following result does not directly follow from the \'etale base change result \cite[Proposition 3.2.1]{GH} for $\THH$ as neither $\EM$ nor $\THR$ commute with $\iota$. This result is crucial for the descent results further below.

\begin{prop}
\label{etale.23}
Let $A\to B$ be an \'etale homomorphism of commutative rings.
Then we have an induced equivalence 
\begin{equation}
\label{etale.23.1}
\THR(\iota A)\wedge_{\EM \iota A}\EM \iota B
\stackrel{\simeq}{\to}
\THR(\iota B).
\end{equation}
\end{prop}
\begin{proof}
Consider the commutative square
\begin{equation}
\label{etale.23.2}
\begin{tikzcd}
\EM \iota A\wedge_{N^{\Ztwo} \EM A}\EM \iota B\ar[d,"\simeq"']\ar[r]&
\EM \iota B\wedge_{N^{\Ztwo} \EM B}\EM \iota B\ar[d,"\simeq"]
\\
\EM(N_A^{\Ztwo}B)\wedge_{N^{\Ztwo}\EM B}\EM \iota B\ar[r]&
\EM(N_B^{\Ztwo}B)\wedge_{N^{\Ztwo}\EM B}\EM \iota B,
\end{tikzcd}
\end{equation}
where the horizontal maps are induced by the map $A\to B$, the vertical equivalences are obtained by Proposition \ref{etale.25}, and the right hand one even comes from the algebraic isomorphism mentioned after Definition \ref{defnztwoam}.
We have an equivalence $H\iota B\wedge_{N^{\Ztwo} \EM B} \EM \iota B \simeq \THR(\iota B)$.
(This equivalence involves a computation for the index under the $\wedge$, namely $i^*\EM\iota B \simeq \EM B$, which follows from \eqref{etale.17.1}.
The same equivalence is used when identifying the upper left entry in the square.)
This implies that \eqref{etale.23.1} is equivalent to the upper horizontal map of \eqref{etale.23.2}.
Proposition \ref{etale.21} and Lemma \ref{etale.26} show that the lower horizontal map of \eqref{etale.23.2} is flat.
Putting everything together, we deduce that the map \eqref{etale.23.1} is flat. 

We will show that this map is an equivalence as claimed by applying  Proposition \ref{etale.22}. 
By Proposition \ref{t.5}, $N^{\Ztwo} \EM A$ and  $N^{\Ztwo} \EM B$ are $(-1)$-connected.
Hence we may  use Proposition \ref{Green.4}, and are reduced to showing that the induced morphism
\[
\ul{\pi}_0(\THR(\iota A))
\Box_{\iota A}
\iota B
\to
\ul{\pi}_0(\THR(\iota B))
\]
is an isomorphism.
This follows easily from applying Proposition \ref{etale.19} to $A$ and $B$.
\end{proof}

We now establish
isovariant \'etale base change for commutative rings with possibly non-trivial involution.

\begin{thm}
\label{etale.5}
Let $A\to B$ be an isovariant \'etale homomorphism of commutative rings with involutions.
Then there are canonical equivalences in $\NAlg_{\Ztwo}$
\[
\THR(B)
\simeq
\THR(A)\wedge_{\EM \iota (A^{\Ztwo})}\EM \iota (B^{\Ztwo})
\simeq
\THR(A)\wedge_{\EM A}\EM B.
\]
\end{thm}
\begin{proof}
We have isomorphisms $\Spec(A^{\Ztwo})\simeq \Spec(A)/(\Ztwo)$ and $\Spec(B^{\Ztwo})\simeq \Spec(B)/(\Ztwo)$.
Hence by Proposition \ref{equitop.2}, the induced homomorphism $A^{\Ztwo}\to B^{\Ztwo}$ is \'etale, and there is an isomorphism of commutative rings
with involution 
\begin{equation}
\label{etale.5.1}
B\cong A\otimes_{\iota (A^{\Ztwo})}\iota (B^{\Ztwo}).
\end{equation}
Since $A^{\Ztwo}\to B^{\Ztwo}$ is flat, $B^{\Ztwo}$ is a filtered colimit of finite free $A^{\Ztwo}$-modules.
It follows that $\iota A^{\Ztwo}\to \iota B^{\Ztwo}$ is flat by Proposition \ref{flat.4}.
Hence the isomorphism \eqref{etale.5.1} induces an  equivalence $\EM B \simeq \EM A \wedge_{\EM (\iota A^{\Ztwo})} \EM (\iota B^{\Ztwo})$ by Proposition \ref{Mack.1}. 
Then Proposition \ref{etale.1} yields an equivalence
\[
\THR(B)
\simeq
\THR(A)\wedge_{\THR(\iota (A^{\Ztwo}))}\THR(\iota (B^{\Ztwo})).
\]

Now Proposition \ref{etale.23} implies the left equivalence of the Proposition after canceling out one smash factor $\THR(\iota (A^{\Ztwo}))$. Applying the isovariance condition once more gives the right hand side
equivalence.
\end{proof}

\subsection{Presheaves of equivariant spectra}

Zariski and other sheaves and completely determined by their behaviors on {\em affine} schemes. This is known to be true in some homotopical settings as well. The purpose of this section is to establish a rather general result, namely Proposition \ref{affinedetermine}, which applies to our setting, that is the isovariant \'etale site and $\Ztwo$-equivariant spectra. For this, a result of \cite{Ayo} will be very useful.

\begin{df}
\label{etale.15}
Let $\cC$ be a category with a 
Grothendieck topology $t$, and let $\cV$ be a presentable $\infty$-category.
Let
\[
\infPsh(\cC,\cV) := \Fun(\Nerve(\cC)^{op},\cV)
\]
denote the $\infty$-category of presheaves on $\cC$ with values in $\cV$.
We say that a presheaf $\cF\in \infPsh(\cC,\cV)$ satisfies \emph{t-descent} if the induced map
\begin{equation}
\label{etale.15.1}
\cF(X)\to \lim_{i\in \Delta}\cF(\sX_i)
\end{equation}
is an equivalence for every $t$-hypercover $\sX\to X$.
Let $\infShv_t(\cC,\cV)$ denote the full
subcategory of $\infPsh(\cC,\cV)$ consisting of presheaves satisfying $t$-descent.
We often omit $t$ in the above notation if it is clear from the context.

The above condition is sometimes called \emph{t-hyperdescent}, in order to distinguish it from the weaker descent condition only for covering sieves rather than all hypercovers. We refer to \cite[section 6.5]{HTT} and \cite{DHI04} for a careful comparison.

If $\cV\to \cV'$ is a functor of $\infty$-categories, then this induces a functor $\infPsh(\cC,\cV)\to \infPsh(\cC,\cV')$.
In this way, we obtain functors $i^*$, $(-)^{\Ztwo}$, etc for the presheaf categories.
\end{df}

\begin{prop}
\label{shveq.8}
Suppose $\cF\in \infPsh(\cC,\Sp_{\Ztwo})$, where $\cC$ is a site.
Then $\cF\in \infShv(\cC,\Sp_{\Ztwo})$ if and only if $i^*\cF,\cF^{\Ztwo}\in \infShv(\cC,\Sp)$.
\end{prop}
\begin{proof}
Let $\sX\to X$ be a hypercover.
Since $i^*$ and $(-)^{\Ztwo}$ preserve limits, \eqref{etale.15.1} is an equivalence if and only if
\[
i^*\cF(X)\to \lim_{i\in \Delta}i^*\cF(\sX_i)
\text{ and }
\cF(X)^{\Ztwo}\to \lim_{i\in \Delta}\cF(\sX_i)^{\Ztwo}
\]
are equivalences.
Equivalently, $i^*\cF,\cF^{\Ztwo}\in \infShv(\cC,\Sp)$.
\end{proof}

\begin{df}
\label{shveq.9}
Let $\cC$ be a site.
As a consequence of Proposition \ref{shveq.8}, we see that $\infShv(\cC,\Sp_{\Ztwo})$ is the full subcategory of local objects of \cite[Definition 5.5.4.1]{HTT} with respect to the class of maps consisting of
\[
\iota \Sigma^n\Sigma^\infty \sX_+\to \iota \Sigma^n \Sigma^\infty X_+
\text{ and }
i_\sharp \Sigma^n\Sigma^\infty \sX_+\to i_\sharp \Sigma^n \Sigma^\infty X_+
\]
for all hypercovers $\sX\to X$ and integers $n$.
In particular, there is an adjoint pair
\begin{equation}
\label{shveq.9.1}
L:\infPsh(\cC,\Sp_{\Ztwo})\rightleftarrows \infShv(\cC,\Sp_{\Ztwo}):\eta
\end{equation}
by \cite[Proposition 5.5.4.15(3)]{HTT}, where $\eta$ is the inclusion functor.
A map $\cF\to \cG$ in $\infPsh(\cC,\Sp_{\Ztwo})$ is called a \emph{local equivalence} if $L\cF \to L\cG$ is an equivalence.
\end{df}

\begin{prop}
\label{shveq.11}
Let $\cC$ be a site.
Then the functor
\[
i^*\colon \infPsh(\cC,\Sp_{\Ztwo})\to \infPsh(\cC,\Sp)
\]
preserves local equivalences.
\end{prop}
\begin{proof}
We need to show that the maps
\[
i^*\iota \Sigma^n\Sigma^\infty \sX_+\to i^*\iota \Sigma^n\Sigma^\infty X_+
\text{ and }
i^*i_\sharp \Sigma^n\Sigma^\infty \sX_+\to i^*i_\sharp \Sigma^n\Sigma^\infty X_+
\]
are local equivalences for all hypercovers $\sX\to X$ and integers $n$, which follow from Proposition \ref{orth.5}(2),(7).
\end{proof}

\begin{prop}
\label{shveq.10}
Let $\cC$ be a site.
Then the functor
\[
(-)^{\Ztwo}\colon \infPsh(\cC,\Sp_{\Ztwo})\to \infPsh(\cC,\Sp)
\]
preserves local equivalences.
\end{prop}
\begin{proof}
We need to show that the maps
\begin{equation}
\label{shveq.10.1}
(\iota \Sigma^\infty \sX_+)^{\Ztwo}\to (\iota \Sigma^\infty X_+)^{\Ztwo}
\text{ and }
(i_\sharp \Sigma^\infty \sX_+)^{\Ztwo}\to (i_\sharp \Sigma^\infty X_+)^{\Ztwo}
\end{equation}
are local equivalences.
Since $(-)^{\Ztwo} i_\sharp\simeq (-)^{\Ztwo} i_*\simeq \id$ by Proposition \ref{orth.5}(2),(3), the second map in \eqref{shveq.10.1} is an equivalence.

The formulation \eqref{norm.3.2} means that $\iota$ commutes with $\Sigma^\infty$, and there is an equivalence $(\iota X)^{\Ztwo}\simeq X$.
Together with the tom Dieck splitting \cite[Theorem 3.10]{GM95},
we have an equivalence
\[
(\iota \Sigma^\infty X_+)^{\Ztwo}
\simeq
\Sigma^\infty X_+ \oplus \Sigma^\infty (\rB (\Ztwo)\times X)_+.
\]
We have a similar equivalence for $\sX$ too.
Since
\[
\Sigma^\infty \sX_+\to \Sigma^\infty X_+
\text{ and }
\Sigma^\infty (\rB (\Ztwo) \times \sX)_+\to \Sigma^\infty (\rB (\Ztwo) \times X)_+
\]
are local equivalences in $\infPsh(\cC,\Sp)$, the first map in \eqref{shveq.10.1} is a local equivalence.
\end{proof}

\begin{rmk}
Some of the results in this subsection including Propositions \ref{shveq.8} and \ref{shveq.11} have obvious generalizations to $\Sp_G$ for finite groups $G$.
We also expect that Proposition \ref{shveq.10} can be generalized too, but this should require extra effort since the tom Dieck splitting becomes more complicated. 
\end{rmk}

\begin{prop}
\label{shveq.12}
Let $f\colon \cF\to \cG$ be a morphism in $\infPsh(\cC,\Sp_{\Ztwo})$, where $\cC$ is a site.
If $i^*f$ and $f^{\Ztwo}$ are local equivalences, then $f$ is a local equivalence.
\end{prop}
\begin{proof}
By considering the fiber of $f$, we reduce to the case when $\cG=0$.
This means that $i^*\cF$ and $\cF^{\Ztwo}$ are local equivalent to $0$.
The adjoint pair \eqref{etale.9.1} gives a local equivalence $\cF\to \cF'$ with $\cF'\in \infShv(\cC,\Sp_{\Ztwo})$.
By Propositions \ref{shveq.11} and \ref{shveq.10}, $i^*\cF'$ and $\cF'^{\Ztwo}$ are local equivalent to $0$.
Since $i^*\cF',\cF'^{\Ztwo}\in \infShv(\cC,\Sp)$ by Proposition \ref{shveq.8}, it follows that $i^*\cF'$ and $\cF'^{\Ztwo}$ are equivalent to $0$.
Hence $\cF'$ is equivalent to $0$, i.e., $\cF$ is local equivalent to $0$.
\end{proof}

\begin{df}
\label{shveq.3}
Let $\mathfrak{M}$ be a combinatorial model category.
For a category $\cC$, we set
\[
\Psh(\cC,\mathfrak{M})
:=
\Fun(\cC^{op},\mathfrak{M}).
\]
A morphism $\cF\to \cG$ in $\Psh(\cC,\mathfrak{M})$ is a weak equivalence (resp.\ fibration) if $\cF(X)\to \cG(X)$ is a weak equivalence (resp.\ fibration) for all $X\in \cC$.
One can form a projective model structure based on these, see \cite[Definition 4.4.18]{Ayo}.

If $\cV$ is the underlying $\infty$-category of $\mathfrak{M}$, then the underlying $\infty$-category of $\Psh(\cC,\mathfrak{M})$ with respect to the projective model structure is equivalent to $\infPsh(\cC,\cV)$ by \cite[Proposition 1.3.4.25]{HA}.
\end{df}

\begin{df}[{\cite[Definition 4.4.23]{Ayo}}]
\label{shveq.1}
A \emph{category of coefficients} is a left proper cofibrantly generated stable model category $\mathfrak{M}$ satisfying the following conditions:
\begin{enumerate}
\item[(i)]
Finite coproducts of weak equivalences are weak equivalences.
\item[(ii)]
there exists a set $\cE$ of compact objects of $\mathfrak{M}$ that generates the homotopy category $\Ho(\mathfrak{M})$.
\end{enumerate}
The set $\cE$ is considered as a part of the data.
\end{df}

\begin{exm}
\label{shveq.2}
According to \cite[Proposition B.63]{HHR}, $\SpO_G$ is a cofibrantly generated stable model category. 
By \cite[Example B.10, Remark B.64]{HHR}, $\SpO_G$ is left proper.
A consequence of \cite[Corollary B.43]{HHR} is that finite coproducts of weak equivalences are weak equivalences.
Let $\cE$ be the set of $\Sigma^n\Sigma^\infty (G/H)_+$ for all subgroups $H$ of $G$ and integers $n$, which consists of compact objects and generates $\Ho(\SpO_G)$ by Proposition \ref{norm.1}.
In conclusion, $\SpO_G$ is a category of coefficients.

We also note that $\SpO_G$ is a combinatorial model category.
\end{exm}

\begin{df}\label{deflocalmodelstructure}
\label{shveq.7}
Let $\cC$ be a site.
According to \cite[Definition 4.4.28]{Ayo}, a morphism $\cF\to \cG$ in $\Psh(\cC,\SpO_G)$ is called a \emph{local weak equivalence} if the induced morphism of presheaves
\begin{align*}
&(X\in \cC \mapsto \Hom_{\Ho(\Psh(\cC,\SpO_G))}(\Sigma^n\Sigma^\infty (G/H\times X)_+,\cF)
\\
\to &
(X\in \cC \mapsto \Hom_{\Ho(\Psh(\cC,\SpO_G))}(\Sigma^n\Sigma^\infty (G/H\times X)_+,\cG)
\end{align*}
becomes an isomorphism after sheafification for every subgroup $H$ of $G$ and integer $n$.
There is a local projective model structure, see 
\cite[Definition 4.4.34]{Ayo}.
This is a Bousfield localization of the projective model structure with respect to local weak equivalences.
\end{df}

\begin{prop}
\label{shveq.13}
Let $\cC$ be a site.
The underlying $\infty$-category of $\Psh(\cC,\SpO_{\Ztwo})$ with respect to the local projective model structure is equivalent to $\infShv(\cC,\Sp_{\Ztwo})$.
\end{prop}
\begin{proof}
Let $f\colon \cF\to \cG$ be a morphism of fibrant objects in $\Psh(\cC,\SpO_G)$ with respect to the projective model structure.
We need to show that $f$ is a local weak equivalence if and only if the corresponding map $g$ in $\infPsh(\cC,\Sp_G)$ is a local equivalence.

By adjunction, $f$ is a local weak equivalence if and only if the induced morphism of presheaves
\[
(X\in \cC \mapsto \Hom_{\Ho(\SpO)}(\Sigma^n,\cF(X)^H))
\to
(X\in \cC \mapsto \Hom_{\Ho(\SpO)}(\Sigma^n,\cG(X)^H))
\]
becomes an isomorphism after sheafification for $H=e,\Ztwo$ and integer $n$.
By \cite[Theorem 1.3]{DHI04}, this is equivalent to saying that $i^*g$ and $g^{\Ztwo}$ are local equivalences.
Proposition \ref{shveq.12} finishes the proof.
\end{proof}

\begin{prop}\label{affinedetermine}
\label{shveq.14}
Let $\cC$ and $\cC'$ be sites.
If there is an equivalence of topoi $\Shv(\cC)\simeq \Shv(\cC')$, then there is an equivalence of $\infty$-categories
\[
\infShv(\cC,\Sp_{\Ztwo})
\simeq
\infShv(\cC',\Sp_{\Ztwo}).
\]
\end{prop}
\begin{proof}
Apply \cite[Proposition 4.4.56]{Ayo} to the canonical (compose Yoneda and sheafification) functors $\cC\to \Shv(\cC)$ and $\cC'\to \Shv(\cC')$, where the right hand side categories are equipped with the topology described in \cite[after Th\'eor\`eme 4.4.51]{Ayo}. Hence we obtain left Quillen equivalences
\[
\Psh(\cC,\SpO_{\Ztwo})
\to
\Psh(\Shv(\cC),\SpO_{\Ztwo})
\text{ and }
\Psh(\cC',\SpO_{\Ztwo})
\to
\Psh(\Shv(\cC'),\SpO_{\Ztwo})
\]
with respect to the local model structures of Definition \ref{deflocalmodelstructure}.
Owing to Proposition \ref{shveq.13} and \cite[Lemma 1.3.4.21]{HA}, we obtain equivalences of $\infty$-categories
\[
\infShv(\cC,\Sp_{\Ztwo})
\simeq
\infShv(\Shv(\cC),\Sp_{\Ztwo})
\text{ and }
\infShv(\cC',\Sp_{\Ztwo})
\simeq
\infShv(\Shv(\cC'),\Sp_{\Ztwo}).
\]
We obtain the desired equivalence of $\infty$-categories thanks to the equivalence of topoi (and hence sites)
$\Shv(\cC)\simeq \Shv(\cC')$.
\end{proof}

\subsection{Isovariant \'etale descent}
\label{sec3.4}

For an abelian group $M$, \eqref{etale.17.1} gives a canonical equivalence
\[
\EM M \xrightarrow{\simeq} (\EM \iota M)^{\Ztwo}.
\]
The last adjunction in \eqref{etale.0.1} then yields a maps $\iota \EM M \to \EM \iota M$.
We have a similar map for a commutative ring $A$.

\begin{lem}
\label{etale.16}
Let $A$ be a commutative ring and $M$  a flat $A$-module.
Then there is an equivalence
\[
\EM \iota M
\simeq
\EM \iota A \wedge_{\iota \EM A} \iota \EM M.
\]
\end{lem}
\begin{proof}
By Proposition \ref{t.4}, $\EM$ preserves filtered colimits.
The two functors $\iota$ also preserve filtered colimits since they are left adjoints.
Every flat $A$-module is a filtered colimit of finitely generated free $A$-modules by Lazard's theorem,
so we reduce to the case when $M$ is a finitely generated free $A$-module.
In this case, the claim is clear.
\end{proof}

For a scheme $X$, let $\Et\Aff/X$ denote the category of affine schemes \'etale over $X$. We start by considering the affine case $X=\Spec(A)$.

\begin{lem}
\label{etale.6}
Let $A$ be a  commutative ring, and let $M$ be an $\EM A$-module.
The presheaf $\cF$ of spectra on $\Et\Aff/\Spec(A)$ given by
\[
\cF(\Spec(B))
:=
M\wedge_{\EM A}\EM B
\]
for \'etale homomorphisms $A\to B$ satisfies \'etale descent.
\end{lem}
\begin{proof}
Since $\EM A\to \EM B$ is flat in the sense of \cite[Definition 7.2.2.10]{HA},
we have an isomorphism
\[
\pi_q \cF(\Spec(B)))
\simeq
\pi_q(M)\otimes_A B
\]
for all integers $q$.
In particular, $\pi_q \cF$ is a quasi-coherent sheaf on $\Et\Aff/\Spec(A)$.

According to the paragraph preceding \cite[Proposition 3.1.2]{GH}, there is a conditionally convergent spectral sequence
\[
E_2^{st}
:=
H_{\et}^s(X,a_{\et}\pi_{-t}\cF)
\Rightarrow
\pi_{-s-t}L_{\et}\cF(X),
\]
where $a_{\et}$ denotes the \'etale sheafification functor.
Since $\pi_{-t}\cF$ is a quasi-coherent sheaf for every integer $t$, the cohomology $H^s(X,a_{\et} \pi_{-t}\cF)$ vanishes for every integer $s\neq 0$.
It follows that $\pi_{-t}\cF\to \pi_{-t}L_{\et}\cF$ is an isomorphism for every integer $t$, i.e., $\cF$ satisfies \'etale descent.
\end{proof}

Let $A$ be a commutative ring with involution.
Let $\isoEtAff/\Spec(A)$ denote the category of affine schemes with involutions isovariant \'etale over $\Spec(A)$.

\begin{thm}
\label{etale.7}
The presheaf
\[
\THR\in \infPsh(\Aff_{\Ztwo},\Sp_{\Ztwo})
\]
satisfies isovariant \'etale descent, where $\Aff_{\Ztwo}$ denotes the category of affine schemes with involutions.
\end{thm}
\begin{proof}

We only need to show that the restriction of $\THR$ to $\isoEtAff/\Spec(A)$ satisfies isovariant \'etale descent. 
The functor
\begin{equation}
\label{etale.7.1}
\Et\Aff/\Spec(A^{\Ztwo})
\to
\isoEtAff/\Spec(A)
\end{equation}
sending $\Spec(B)$ to $\Spec(A\otimes_{\iota A^{\Ztwo}}\iota B)$ is an equivalence of sites by Proposition \ref{equitop.4}.
Let
\[
\cF\in \infPsh(\Et\Aff/\Spec(A^{\Ztwo}),\Sp_{\Ztwo})
\]
be the presheaf given by
$
\cF(\Spec(B))
:=
\THR(A)\wedge_{\EM \iota A^{\Ztwo}}\EM \iota B
$
for \'etale homomorphisms $A^{\Ztwo}\to B$.
If we show that $\cF$ satisfies \'etale descent, then the presheaf
\[
\cG\in \infPsh(\isoEtAff/\Spec(A),\Sp_{\Ztwo})
\]
obtained from $\cF$ and \eqref{etale.7.1} satisfies isovariant \'etale descent.
Theorem \ref{etale.5} gives the third equivalence in
\[
\cG(\Spec(\iota B\otimes_{\iota A^{\Ztwo}}A))
=
\cF(\Spec(B))
=
\THR(A)\wedge_{\EM \iota A^{\Ztwo}}\EM \iota B
\simeq
\THR(A\otimes_{\iota A^{\Ztwo}}\iota B),
\]
i.e., $\cG\simeq \THR$.
This means that $\THR$ satisfies isovariant \'etale descent.

Hence it remains to check that $\cF$ satisfies \'etale descent.
There is an equivalence
\[
i^*\cF(\Spec(B))
\simeq
i^*\THR(A)\wedge_{\EM A^{\Ztwo}}\EM B.
\]
By Lemmas \ref{etale.4} and \ref{etale.16}, there are equivalences
\[
\cF(\Spec(B))^{\Ztwo}
\simeq
(\THR(A)\wedge_{\iota \EM A^{\Ztwo}}\iota \EM B)^{\Ztwo}
\simeq
\THR(A)^{\Ztwo}\wedge_{\EM A^{\Ztwo}}\EM B.
\]
Lemma \ref{etale.6} implies that $i^*\cF$ and $\cF^{\Ztwo}$ satisfy \'etale descent, which means that $\cF$ satisfies \'etale descent.
\end{proof}

\begin{prop}
\label{etale.31}
For every separated scheme $X$ with involution, there exists an isovariant \'etale covering $\{U_i\to X\}_{i\in I}$ such that each $U_i$ is an affine scheme with involution.
\end{prop}
\begin{proof}
Let $X^{\Ztwo}$ denote the closed subscheme of $X$ obtained by Definition \ref{equitop.5}.
We set $Y:=X-X^{\Ztwo}$, which is an open subscheme of $X$.

If $x$ is a point of $X^{\Ztwo}$, choose an affine open neighborhood $U_x$ of $x$ in $X$.
Then $V_x:=U_x\cap w(U_x)$ is again an affine open neighborhood of $x$ since $X$ is separated, and $w$ can be restricted to $V_x$.
It follows that
\[
\{V_x\to X\}_{x\in X}
\cup
\{Y\to X\}
\]
is a Zariski covering.

The $\Ztwo$-action on $Y$ is free, so the quotient morphism
$Y\to Y/(\Ztwo)$ is \'etale since the algebraic space $Y/(\Ztwo)$ is formed by the \'etale equivalence relation $Y\times \Ztwo \rightrightarrows Y$, and there is an isomorphism $Y\times_{Y/(\Ztwo)}Y\cong Y\times \Ztwo$.
It follows that the morphism $Y\times \Ztwo\to Y$ obtained by the first projection $Y\times_{Y/(\Ztwo)}Y\to Y$ is isovariant \'etale.
Choose a Zariski covering $\{W_j\to Y\}_{j\in J}$ after forgetting involution such that each $W_j$ is an affine scheme, and then
\begin{equation}
\{V_x\to X\}_{x\in X}
\cup
\{W_j\times \Ztwo\to X\}_{j\in J}
\end{equation}
is an isovariant \'etale covering.
\end{proof}

Let $\Sch_{\Ztwo}$ denote the category of separated schemes with involutions.

\begin{prop}
\label{etale.30}
There is an equivalence of topoi
\begin{equation}
\label{etale.30.1}
\Shv_{iso\et}(\Aff_{\Ztwo})
\simeq
\Shv_{iso\et}(\Sch_{\Ztwo}).
\end{equation}
Hence there is an equivalence of $\infty$-categories
\begin{equation}
\label{etale.30.2}
\infShv_{iso\et}(\Aff_{\Ztwo},\Sp_{\Ztwo})
\simeq
\infShv_{iso\et}(\Sch_{\Ztwo},\Sp_{\Ztwo}).
\end{equation}
\end{prop}
\begin{proof}
Combine \cite[Th\'eor\`eme III.4.1]{SGA4} and Proposition \ref{etale.31} to obtain \eqref{etale.30.1}.
Use Proposition \ref{shveq.14} for \eqref{etale.30.2}.
\end{proof}

\begin{df}
\label{etale.34}
By Theorem \ref{etale.7} and Proposition \ref{etale.30}, we obtain
\[
\THR\in \infShv_{iso\et}(\Sch_{\Ztwo},\Sp_{\Ztwo}).
\]
This definition immediately implies that $\THR(X)$  that satisfies isovariant \'etale descent for all $X\in \Sch_{\Ztwo}$.
\end{df}

Below, we carry out computations of $\THR(X)$ for projective spaces $X=\P^n$, even with involution for $n=1$. Let us first check that extending Proposition \ref{etale.10} to schemes yields a definition of $\THH$ for schemes equivalent to the one of \cite{GH} as follows. See e.g.\ \cite[p.\ 1055 and chapter 3]{BM} for a discussion of other equivalent definitions of $\THH(X)$.

\begin{prop}
\label{etale.35}
For $X\in \Sch_{\Ztwo}$ and $Y\in \Sch$, there are canonical equivalences
\[
\THH(i^*X)
\simeq
i^*\THR(X)
\text{ and }
\THR(Y\amalg Y)
\simeq
i_*\THH(Y),
\]
where the involution on $Y\amalg Y$ switches the components.
\end{prop}
\begin{proof}
The question is isovariant \'etale local on $X$ and \'etale local on $Y$, so we reduce to the case when $X=\Spec(A)$ for some commutative ring $A$ with involution and $Y=\Spec(B)$ for some commutative ring $B$.
We obtain equivalences
\[
\THR(i^*A)
=
\THR(\EM(i^*A))
\simeq
\THR(i^* \EM A)
\simeq
i^*\THR(\EM A)
=
i^* \THR(A)
\]
using Proposition \ref{etale.10} and \eqref{etale.17.1}.
We similarly obtain the remaining desired equivalence using Propositions \ref{etale.9} and \ref{etale.33}.
\end{proof}

\begin{df}
\label{etale.36}
Recall from \cite[section 2.1]{HeKO} that an \emph{equivariant Nisnevich distinguished square} is a cartesian square in $\Sch_{\Ztwo}$ of the form
\begin{equation}
Q
:=
\label{etale.36.1}
\begin{tikzcd}
V\ar[r]\ar[d]&
Y\ar[d,"f"]
\\
U\ar[r,"j"]&
X
\end{tikzcd}
\end{equation}
such that $j$ is an open immersion, $f$ is equivariant \'etale, and the morphism of schemes $(Y-V)_\mathrm{red}\to (X-U)_\mathrm{red}$ is an isomorphism.
The collection of such squares forms a cd-structure, which is bounded, complete, and regular in the sense of \cite{Voe10} by \cite[Theorem 2.3]{HeKO}.
Furthermore, the associated topology is the equivariant Nisnevich topology in Definition \ref{equitop.1} by \cite[Proposition 2.17]{HeKO}.
\end{df}

\begin{cor}
\label{etale.37}
For every equivariant Nisnevich distinguished square $Q$ of the form \eqref{etale.36.1}, the induced square
\[
\THR(Q)
=
\begin{tikzcd}
\THR(X)\ar[d]\ar[r]&
\THR(U)\ar[d]
\\
\THR(Y)\ar[r]&
\THR(V)
\end{tikzcd}
\]
is cartesian.
\end{cor}
\begin{proof}
Let $P$ be the collection of equivariant Nisnevich distinguished squares.
Consider the class $G_P$ of morphisms between simplicial presheaves in \cite[p.\ 1392]{Voe10}.
By \cite[Proposition 3.8(2)]{Voe10}, every $G_P$-local equivalence is an equivariant Nisnevich local equivalence.
This implies that if $\cF$ is a simplicial presheaf satisfying equivariant Nisnevich descent, then $\cF(Q)$ is cartesian for all $Q\in P$.

Since $\THR$ satisfies isovariant \'etale descent, it satisfies equivariant Nisnevich descent.
Hence we can apply the argument in the above paragraph to $\Omega^\infty \Sigma^n i^* \THR$ and $\Omega^\infty \Sigma^n \Phi^{\Ztwo} \THR$ for all integers $n$ to see that $i^*\THR(Q)$ and $\Phi^{\Ztwo}\THR(Q)$ are cartesian, which implies the claim.
\end{proof}

For the computations about $\THR(\P^n)$ we provide in the next sections, we may simply use appropriate equivariant Nisnevich covers by affine schemes and Corollary \ref{etale.37}, which e.g.\ exhibits $\THR(\P^1)$ as part of a homotopy cocartesian square in which all other entries are of the form $\THR(Y)$ for $Y=\Spec(A[M])$ corresponding to some commutative monoid rings for $M=\N$ or $M=\Z$ over the commutative base ring $A$.
We could even go one step further and {\em define}  $\THR(\P^n_{\Sphere},\sigma)$ of THR for projective spaces, possibly with non-trivial involution $\sigma$, over the sphere spectrum $\Sphere$ rather than over $A$ or $\EM A$, using the homotopy pushout of the appropriate diagram of THR of the corresponding spherical monoid rings $\Sphere[M]$, compare the proofs in section \ref{periodicity}.

\section{The dihedral bar construction}
\label{sec4}

\subsection{Crossed simplicial groups}
Hochschild and cyclic homology and their real refinements are closely related to cyclic, real and dihedral nerves. We now present a uniform treatment of these constructions. The original references are \cite{FL} and \cite{Lo},  parts of this are also explained e.g.\ in \cite{DO19} and \cite{DMPR21}.

\begin{df}(see \cite[Definition 1.1]{FL},\cite[chapter 6.3]{Lo})\label{dih.2}
A \emph{crossed simplicial group} $G$ is a sequence of groups $\{G_n\}_{n\geq 0}$ together with a category $\Delta G$ satisfying the following conditions:
\begin{enumerate}
\item[(i)]
$\Delta G$ contains $\Delta$ as a subcategory with the same objects.
\item[(ii)]
$\mathrm{Aut}([n])$ is the opposite group $G_n^{op}$.
\item[(iii)]
A morphism in $\Delta G$ can be uniquely written as the composite $\alpha \circ g$ for some $\alpha \in \Hom_{\Delta}([m],[n])$ and $g\in G_m^{op}$.
\end{enumerate}
\end{df}

Observe that every crossed simplicial group is a (or rather has an underlying) simplicial set, see \cite[Lemma 1.3]{FL}.
A \emph{$G$-set} is a functor $\Delta G^{op}\to \Set$, which by restriction has an underlying simplicial set.

\begin{const}
\label{dih.4}
Let $X$ be a simplicial set.
For a crossed simplicial group $G$, the $G$-set
$F_G(X)$ is defined in \cite[Definition 4.3]{FL}.
In simplicial degree $q$, we have
\[
F_G(X)_q
:=
G_q\times X_q.
\]
The $G_q$ action on $F_G(X)_q$ is the left multiplication on $G_q$.
The faces and degeneracy maps are given by
\begin{equation}
d_i(g,x):=(d_i(g),d_{g^{-1}(i)}(x))
\text{ and }
s_i(g,x):=(s_i(g),s_{g^{-1}(i)}(x)).
\end{equation}

According to \cite[Proposition 5.1]{FL}, we have the projection $p_1\colon F_G(X)\to G$ given by $p_1(g,x):=g$.
We also have the projection $p_2\colon \lvert F_G(X)\rvert \to \lvert X \rvert$ given by
\[
p_2[(g,x),u]
:=
[x,gu]
\]
for $u\in \Delta_{top}^q$.
Furthermore, the two projections define a homeomorphism
\[
\lvert F_G(X) \rvert
\cong
\lvert G \rvert \times \lvert X \rvert.
\]

If $X$ is a $G$-set, then we have the evaluation map $ev\colon  F_G(X) \to X$ given by
\[
ev(g,x)
:=
gx.
\]
According to \cite[Theorem 5.3]{FL}, the composite map
\[
\lvert G \rvert \times \lvert X \rvert
\xrightarrow{\cong}
\lvert F_G(X) \rvert
\xrightarrow{ev}
\lvert X \rvert
\]
defines a $\lvert G \rvert$-action on $\lvert X \rvert$.
The \emph{$G$-geometric realization of $X$} is $\lvert X\rvert$ with this $\lvert G\rvert$-action.
\end{const}

\begin{exm}
\label{dih.3}
We have the following three fundamental examples of crossed simplicial groups.\\[5pt]
(1) \cite[Example 2, section 1.5]{FL} introduces a crossed simplicial group $G$ with $G_n:=\Ztwo$.
For this $G$, a $G$-set is called a \emph{real simplicial set}.
The $G$-geometric realization is called the \emph{real geometric realization}.

Explicitly, a real simplicial set $X$ is a simplicial set equipped with isomorphisms $w_n\colon X_n\to X_n$ with $w_n^2=\id$ for all integers $n\geq 1$ satisfying the relations
\[
d_i w_n=w_{n-1}d_{n-i}
\text{ and }
s_i w_n = w_{n+1}s_{n-i}
\]
for $0\leq i\leq n$.
\\[5pt]
(2) \cite[Example 4, section 1.5]{FL} introduces a crossed simplicial group $G$ with $G_n:=C_{n+1}$.
For this $G$, a $G$-set is called a \emph{cyclic set}, which was defined by Connes.
The $G$-geometric realization is called the \emph{cyclic geometric realization}, and comes with an action of $S^1=SO(2)$ (see also \cite[Theorem 7.1.4]{Lo}).

Explicitly, a cyclic set $X$ is a simplicial set equipped with isomorphisms $t_n\colon X_n\to X_n$ with $(t_n)^{n+1}=\id$ for all integers $n\geq 1$ satisfying the relations
\[
d_it_n=t_{n-1}d_{i-1}
\text{ and }
s_it_n=t_{n+1}s_{i-1}
\]
for $1\leq i\leq n$.\\[5pt]
(3) \cite[Example 5, section 1.5]{FL} introduces a crossed simplicial group $G$ with $G_n:=D_{n+1}$, where $D_n$ denotes the dihedral group of order $2n$.
For this $G$, a $G$-set is called a \emph{dihedral set}.
The $G$-geometric realization is called the \emph{dihedral geometric realization}.

Explicitly, a dihedral set $X$ is a simplicial set equipped with $t_n$ and $w_n$ for all integers $n\geq 1$ satisfying $w_n t_n=t_n^{-1}w_n$ and all the above conditions for $t_n$ and $w_n$.

There are obvious forgetful functors from dihedral to real and to cyclic sets.
\end{exm}

For $\lvert G \rvert$, we obtain $\Ztwo$, $S^1=U(1)=SO(2)$, and $O(2)$ in the three examples above.

\begin{exm}
\label{dih.5}
Let $G$ be the crossed simplicial set in Example \ref{dih.3}(1).
For a simplicial set $X$, the bijection $G \times X_q = X_q \amalg X_q$ induces a canonical isomorphism of $G$-sets
\[
F_{G}(X)
\cong
X\amalg X^{op},
\]
where $w_n\colon X_n\amalg X_n^{op}\to X_n\amalg X_n^{op}$ is the switching map.
If $X$ is a real simplicial set, then the evaluation map 
$ev\colon F_G(X)\to X$ sends $x\in X_n^{op}$ to $w_n(x)$.
It follows that the $\Ztwo$-action on $\lvert X \rvert$ is given by the composite map
\[
\lvert X \rvert
\xrightarrow{\cong}
\lvert X^{op} \rvert
\xrightarrow{w}
\lvert X \rvert,
\]
where the first map is the canonical identity.
\end{exm}

\begin{exm}
\label{dih.6}
Let $G$ be the crossed simplicial set in Example \ref{dih.3}(2).
For every integer $n\geq 0$, we set
\[
\Lambda^n
:=
F_G(\Delta^n).
\]
This is Connes' cyclic $n$-complex.
We warn the reader that $F_G(X)$ for a simplicial set is different from the cyclic bar construction.
\end{exm}

\subsection{Real and dihedral nerves}
We now recall how commutative monoids with involution give rise to real and dihedral simplicial sets, refining nerves and cyclic nerves for commutative monoids without involutions.

\begin{df}
\label{dih.10}
Let $M$ be a commutative monoid with involution $\sigma$.
The \emph{real nerve of $M$}, denoted $\Nsigma M$, is the real simplicial set whose underlying simplicial set is the nerve of $M$ with $(\Nsigma M)_q:=N^{\times q}$ and $w$ is given by
\[
w(x_1,\ldots,x_q)
:=
(\sigma(x_q),\ldots,\sigma(x_1))
\]
in simplicial degree $q$.
The \emph{real bar construction}, denoted $\Bsigma M$, is the real geometric realization of $\Nsigma M$.
We refer to \cite[Example 2.1.1]{DO19} for a similar account.
\end{df}

\begin{df}
\label{dih.9}
Let $M$ be a commutative monoid with involution $\sigma$.
The \emph{dihedral nerve of $M$}, denoted $\Ndi M$, is the dihedral set defined as follows.
In simplicial degree $q$, we have $(\Ndi M)_q:=M^{\times q+1}$.
The face maps are
\[
d_i(x_0,\ldots,x_q)
:=
\left\{
\begin{array}{ll}
(x_0,\ldots,x_{i-1},x_i+x_{i+1},x_{i+2},\ldots,x_q) & \text{if }i=0,\ldots,q-1
\\
(x_q+x_0,x_1,\ldots,x_{q-1}) & \text{if }i=q.
\end{array}
\right.
\]
The degeneracy maps are
\[
s_i(x_0,\ldots,x_q)
:=
(x_0,\ldots,x_{i},0,x_{i+1},\ldots,x_q) \text{ for }i=0,\ldots,q.
\]
The rotation maps are
\[
t_q(x_0,\ldots,x_q)
:=
(x_0,x_i,x_{i+1},\ldots,x_q,x_1,\ldots,x_{i-1})
\text{ for }i=0,\ldots,q.
\]
The involution is
\[
w(x_0,\ldots,x_q)
:=
(\sigma(x_0),\sigma(x_q),\ldots,\sigma(x_1)),
\]
compare also with \cite[Example 2.1.2]{DO19}.

\begin{rmk}\label{remHRloday}
This involution looks similar to the involution on the Hochschild complex as studied in \cite[5.2.1]{Lo}. In \cite[section 5.2]{Lo} Loday investigates several ``real'' versions of Hochschild and cyclic homology for rings. 
The decompositions of  \cite[Proposition 5.2.3, (5.2.7.1)]{Lo}
should be compared to the Bott sequence \cite[section 6]{schlichting17} relating algebraic to hermitian $K$-theory, which also splits after inverting $2$ or rather $1- \epsilon$. 
\end{rmk}

Note that we obtain the \emph{cyclic nerve of $M$} \cite[section 2.3]{Wal79}, denoted $\Ncy M$, if we forget the involution structure.

The \emph{dihedral bar construction}, denoted $\Bdi M$, is the dihedral geometric realization of $\Ndi M$.
By Construction \ref{dih.4}, $\Bdi M$ admits a canonical $O(2)$-action.
The underlying real simplicial set of $\Ndi M$ is different from $\Nsigma M$.

We have the map of dihedral sets
\begin{equation}
\label{dih.9.1}
\Ndi M\to M
\end{equation}
sending $(x_0,\ldots,x_q)$ to $x_0+\cdots +x_q$
in simplicial degree $q$, where we regard $M$ as the constant dihedral set whose rotation maps are the identities and whose involutions are given by $\sigma$.
For every $\sigma$-orbit $I$ in $M$, we set
\begin{equation}
\Ndi(M;I)
:=
\Ndi M \times_M I.
\end{equation}
Let $\Bdi(M;I)$ be its dihedral geometric realization.
We have an isomorphism of dihedral sets
\begin{equation}\label{ndidecomposition}
\Ndi M
\cong
\coprod_{I}\Ndi (M;I),
\end{equation}
where the coproduct runs over the $\sigma$-orbits in $M$.
\end{df}

\begin{prop}
\label{dih.26}
Let $M$ and $L$ be commutative monoids with involutions.
Then there is an isomorphism of dihedral sets
\begin{equation}
\label{dih.26.1}
\Ndi (M\times L)
\cong
\Ndi M \times \Ndi L.
\end{equation}
\end{prop}
\begin{proof}
The two projections $M\times L\rightrightarrows M,L$ induce \eqref{dih.26.1}.
To show that it is an isomorphism, observe that it is given by the shuffle homomorphism
\[
(M\times L)^{\times (q+1)}
\to
M^{\times (q+1)} \times L^{\times (q+1)}
\]
in simplicial degree $q$.
\end{proof}

\begin{prop}
\label{dih.27}
Let $M$ and $L$ be commutative monoids with the involutions.
Then there is an isomorphism of dihedral sets
\[
\Ndi (M\times L;(x,y))
\cong
\Ndi (M;x) \times \Ndi (L;y)
\]
for all $x\in M$ and $y\in L$.
\end{prop}
\begin{proof}
Follows from the commutativity of the square
\[
\begin{tikzcd}
(M\times L)^{\times (q+1)}\ar[d]\ar[r]&
M^{\times (q+1)}\times L^{\times (q+1)}\ar[d]
\\
M\times L\ar[r,"\id"]&
M\times L,
\end{tikzcd}
\]
where the upper horizontal homomorphism is the shuffle homomorphism, and the vertical homomorphisms are the summation homomorphisms.
\end{proof}

\begin{prop}
\label{dih.24}
Let $G$ be an abelian group with an involution $w$, and let $j$ be an element of $G$.
If $j=w(j)$, then there is an isomorphism of real simplicial sets
\begin{equation}
\label{dih.24.1}
\Ndi (G;j)
\xrightarrow{\cong}
\Nsigma G.
\end{equation}
If $j\neq w(j)$, then there is an isomorphism of real simplicial sets
\begin{equation}
\label{dih.24.2}
\Ndi(G;\{j,w(j)\})
\xrightarrow{\cong}
\Ztwo \times \Nsigma G,
\end{equation}
As a consequence, there is an isomorphism of real simplicial sets
\begin{equation}
\label{dih.24.3}
\Ndi G
\simeq
G\times \Nsigma G.
\end{equation}
\end{prop}
\begin{proof}
If $j=w(j)$, the assignment
\[
(x_0,\ldots,x_q)\mapsto (x_1,\ldots,x_q)
\]
in simplicial degree $q$ constructs the isomorphism \eqref{dih.24.1}.
If $j\neq w(j)$, the assignment
\[
(x_0,\ldots,x_q)\mapsto (a,x_1,\ldots,x_q)
\]
in simplicial degree $q$ constructs the isomorphism \eqref{dih.24.2}, where $a:=0$ (resp.\ $a:=1$) if $x_0+\cdots+x_q=j$ (resp.\ $x_0+\cdots+x_q=-j$).
\end{proof}

\begin{df}\label{defSsigma}
Let $S^\sigma$ be the $\Ztwo$-space whose underlying space is $S^1$ and whose $\Ztwo$-action is given by the reflection $(x,y)\in S^1\subset \R^2\mapsto (x,-y)$.
This is the real geometric realization of the real simplicial set whose underlying simplicial set is the simplicial circle $\Delta^1/\partial \Delta^1$ and whose involution on the non-degenerate simplices is the identity. One easily checks that the relations in Example \ref{dih.3}(1) uniquely determines the higher $w_n$.
To understand the involution on the real realization, the reader is advised to look at the explanations in Example \ref{dih.5}.
\end{df}

\begin{prop}
\label{dih.8}
For the additive monoid $\Z$ with the trivial involution, there is a $\Ztwo$-homotopy equivalence
\[
\Bsigma \Z
\simeq
S^\sigma.
\]
\end{prop}
\begin{proof}
Consider the element $1\in (\Nsigma \Z)_1\simeq \Z$ in simplicial degree $1$.
This is fixed by the involution on $\Nsigma \Z$, so the real geometric realization of the real simplicial subset of $\Nsigma \Z$ whose only non-degenerate simplices are $1$ and the base point is $S^\sigma$. 
Hence we obtain a map of $\Ztwo$-spaces $S^{\sigma}\to \Bsigma \Z$.
This is the usual homotopy equivalence after forgetting the involutions.
It remains to check that the induced map
\begin{equation}
\label{dih.8.3}
(S^\sigma)^{\Ztwo}
\to
(\Bsigma \Z)^{\Ztwo}
\end{equation}
is a homotopy equivalence as well.

By \cite[Proposition 2.1.6]{DO19}, $(\Bsigma \Z)^{\Z/2}$ is the classifying space of the category $\mathop{\mathrm{Sym}}\Z$, whose objects are integers, and whose morphisms are given by
\[
\Hom_{\mathop{\mathrm{Sym}}\Z}(a,b)
:=
\left\{
\begin{array}{ll}
* & \text{if }2\mid a-b,
\\
\emptyset & \text{if }2 \nmid a-b.
\end{array}
\right.
\]
The classifying space of $\mathop{\mathrm{Sym}}\Z$ is obviously homotopy equivalent to $S^0$, and hence we deduce a homotopy equivalence
\begin{equation}
\label{dih.8.1}
S^0
\simeq
(\Bsigma \Z)^{\Z/2}.
\end{equation}
This implies that $\pi_n((\Bsigma \Z)^{\Z/2})$ is trivial for every integer $n>0$ and $\pi_0((\Bsigma \Z)^{\Z/2})$ consists of two points.

Hence it suffices to show that the map
\begin{equation}
\label{dih.8.2}
\pi_0((S^\sigma)^{\Ztwo})
\to
\pi_0((\Bsigma \Z)^{\Ztwo})
\end{equation}
induced by \eqref{dih.8.3} is injective.
Let $\sd_\sigma$ denote Segal's edgewise subdivision functor in \cite{Seg73}.
In simplicial degree $1$, the two degeneracy maps in $\sd_\sigma(\Nsigma \Z)$ are given by
\[
(x_1,x_2,x_3)\mapsto x_2,x_1+x_2+x_3.
\]
Then edges $(x_1,x_2,x_1)$ in the $\Ztwo$-fixed point space connects $x_2$ and $2x_1+x_2$.
In simplicial degree $0$, $\sd_{\sigma}(S^\sigma)$ consists of two vertices $0$ and $1$, whose images in $\Bsigma \Z$ are not connected.
Hence \eqref{dih.8.2} is injective as claimed.
\end{proof}

Together with \eqref{dih.24.1}, we obtain an equivalence
\begin{equation}
\label{dih.8.4}
\Bdi (\Z;j)\simeq S^\sigma
\end{equation}
for all $j\in \Z$.

\begin{df}
\label{dih.11}
Let $\Z^\sigma$ denote the monoid $\Z$ with the involution $x\mapsto -x$ for $x\in \Z$.
\end{df}

\begin{prop}
\label{dih.12}
There is a $\Ztwo$-homotopy equivalence
\[
\Bsigma \Z^\sigma
\simeq
S^1.
\]
\end{prop}
\begin{proof}
See \cite[Example 5.13]{DMPR21}.
\end{proof}

Combine this with \eqref{dih.24.3} to obtain a $\Z/2$-homotopy equivalence
\[
\Bdi \Z^\sigma
\simeq
S^1
\amalg
\coprod_{j>0} (\Ztwo \times S^1).
\]
Together with the description of $i_*$ in \eqref{etale.8.1},
we obtain an equivalence in $\Sp_{\Ztwo}$
\begin{equation}
\label{dih.12.1}
\Sphere[\Bdi \Z^\sigma]
\simeq
\Sphere[S^1]\oplus \bigoplus_{j>0} i_*i^*\Sphere[S^1].
\end{equation}

For every integer $j$, let $S^\sigma(j)$ denote the space $S^1$ with the $O(2)$-action, whose $SO(2)$-action is given by $(t,x)\mapsto t^j x$ for $t\in SO(2)$ and $x\in S^1$, and whose $\Ztwo$-action is given by the complex conjugate $x\mapsto \overline{x}$.
Here, we regard $S^1$ as the unit circle in $\C$.
For $j\geq 0$, let $\Delta^j_\sigma$ be the $\Ztwo$-space whose involution is the reflection mapping the vertex $i$ to $j-i$ for all $0\leq i\leq j$.

The following computations refine
\cite[Proposition 3.21]{Rog09}.

\begin{prop}
\label{dih.1}
We have $\Bdi(\N;0)\cong *$.
For every integer $j\geq 1$, there is an $O(2)$-equivariant homeomorphism
\begin{equation}
\label{dih.1.3}
\Bdi (\N;j)
\simeq
(S^\sigma\times \Delta^{j-1}_\sigma)/C_j,
\end{equation}
where the $C_j$-action on $\Delta^{j-1}_{\sigma}$ is the cyclic permutation.
Hence there is an $O(2)$-equivariant deformation retract
\begin{equation}
\label{dih.1.4}
\Bdi \N
\simeq
*
\amalg
\coprod_{j\geq 1} S^\sigma(j).
\end{equation}
\end{prop}
\begin{proof}
The $(j-1)$-simplex $(1,\ldots,1)$ generates $\Ncy(\N;j)$ as a cyclic set.
Use this to construct a surjective map of cyclic sets
\begin{equation}
\label{dih.1.1}
\Lambda^{j-1} \to \Ncy (\N;j).
\end{equation}
In simplicial degree $q$, this can be written as the map
\[
C_{q+1}\times \Hom([q],[j-1])
\to
\{(x_0,\ldots,x_q)\in \N^{q+1} : x_0+\cdots+x_q=j\}.
\]
sending $(t,f)$ to
\[
(f(t(0))-f(t(q)),f(t(1))-f(t(0)),\ldots,f(t(q))-f(t(q-1))).
\]
In this formulation, the values are calculated modulo $j$.
Let $\Lambda_\sigma^{j-1}$ be the dihedral set whose underlying cyclic set is $\Lambda^{j-1}$ and whose involution is given by
\[
(t,f)\mapsto (q+1-t,\rho\circ f),
\]
where $\rho\colon [j-1]\to [j-1]$ is the map sending $x\in [j-1]$ to $j-1-x$.
Then \eqref{dih.1.1} becomes a morphism of dihedral sets
\begin{equation}
\label{dih.1.2}
\Lambda_\sigma^{j-1}\to \Ndi(\N;j).
\end{equation}

The cyclic geometric realization of \eqref{dih.1.1} is $S^1$-homeomorphic to the quotient map
\[
S^1\times \Delta^{j-1}
\to
(S^1\times \Delta^{j-1})/C_j,
\]
see the proof of \cite[Lemma 2.2.3]{Hes96} or \cite[Proposition 3.20]{Rog09}.
Use the observation in Example \ref{dih.5} to show that the real geometric realization of $\lvert \Lambda_\sigma^{j-1}\rvert$ is $\Ztwo$-homeomorphic to $S^\sigma \times \Delta^{j-1}_\sigma$.
Combine these two facts to deduce that the dihedral geometric realization of \eqref{dih.1.2} is $O(2)$-homeomorphic to the quotient map
\[
S^\sigma \times \Delta^{j-1}_\sigma
\to
(S^\sigma \times \Delta^{j-1}_\sigma)/C_j
\]
In particular, we obtain \eqref{dih.1.3}.
Since $\Delta_{\sigma}^{j-1}$ is $D_{j}$-contractible to its barycenter, we obtain \eqref{dih.1.4}.
\end{proof}

Let us review what the proof of \cite[Proposition 3.21]{Rog09} contains.
For every integer $r\geq 1$, let $\sd_r$ denote the $r$-fold edgewise subdivision functor in \cite{BHM93}.
For every integer $j$, the $r$-fold power map $\Ndi(\Z;j) \to \sd_r\Ndi(\Z;j)$ given by
\begin{equation}
\label{dih.7.2}
(m_0,\ldots,m_q) \mapsto (m_0,\ldots,m_q,\ldots,m_0,\ldots,m_q)
\end{equation}
induces a homeomorphism
\begin{equation}
\label{dih.7.4}
\Bdi(\Z;j)
\xrightarrow{\cong}
\Bdi(\Z;rj)^{C_r}.
\end{equation}
If $r \nmid j$, then we have
\begin{equation}
\label{dih.7.7}
\Bdi(\Z;j)^{C_r}=\emptyset.
\end{equation}
Suppose $j\geq 0$.
We similarly have a homeomorphism
\begin{equation}
\label{dih.7.3}
\Bdi(\N;j)
\xrightarrow{\cong}
\Bdi(\N;rj)^{C_r}
\end{equation}
If $r \nmid j$, then we have
\begin{equation}
\label{dih.7.5}
\Bdi(\N;j)^{C_r}=\emptyset.
\end{equation}
If $H$ is a closed subgroup of $SO(2)$, then the induced map
\begin{equation}
\label{dih.7.1}
\Bdi(\N;j)^H
\to
\Bdi(\Z;j)^H
\end{equation}
is a homotopy equivalence.

\begin{prop}
\label{dih.7}
For every integer $j\geq 1$, the induced map
\[
\Bdi(\N;j)
\to
\Bdi(\Z;j)
\]
is an $O(2)$-equivariant homotopy equivalence.
\end{prop}
\begin{proof}
We need to show that \eqref{dih.7.1} is a homotopy equivalence for all closed subgroups $H$ of $O(2)$.

If $H=O(2)$, then $\Bdi(\N;j)^{O(2)}=\Bdi(\Z;j)^{O(2)}=\emptyset$ by \eqref{dih.7.7} and \eqref{dih.7.5}.
Hence the remaining case is $H=D_r$ for every integer $r\geq 1$.
Since \eqref{dih.7.2} commutes with $w$, \eqref{dih.7.3} is $\Ztwo$-equivariant.
Similarly, \eqref{dih.7.4} is $\Ztwo$-equivariant.
Hence we have the induced homeomorphisms
\[
\Bdi(\N;j)^{\Ztwo}
\xrightarrow{\cong}
\Bdi(\N;rj)^{D_r}
\text{ and }
\Bdi(\Z;j)^{\Ztwo}
\xrightarrow{\cong}
\Bdi(\Z;rj)^{D_r}.
\]
Combine with \eqref{dih.7.5} to reduce to the case when $H=\Ztwo$.

Propositions \ref{dih.24} and \ref{dih.8} give a homotopy equivalence
\begin{equation}
\Bdi(\Z;j)^{\Ztwo}
\simeq
S^0.
\end{equation}
By Proposition \ref{dih.1}, we also have a homotopy equivalence $\Bdi(\N;j)^{\Ztwo}\simeq S^0$.
Hence it remains to check that the induced map
\begin{equation}
\label{dih.7.6}
\pi_0(\Bdi(\N;j)^{\Ztwo})
\to
\pi_0(\Bdi(\Z;j)^{\Ztwo})
\end{equation}
is a bijection.
Recall that $\sd_\sigma$ denotes Segal's edgewise subdivision functor.
In simplicial degree $1$, the two degeneracy maps in $\sd_\sigma(\Ndi(\N;j))$ and $\sd_\sigma(\Ndi(\Z;j))$ are given by
\[
(x_0,x_1,x_2,x_3)\mapsto (x_3+x_0+x_1,x_2),(x_0,x_1+x_2+x_3).
\]
The edge $(x_0,x_1,x_2,x_1)$ in the $\Ztwo$-fixed point spaces connects $(x_0,2x_1+x_2)$ and $(x_0+2x_1,x_2)$.
Hence we have
\[
\pi_0(\Bdi(\N;j))
\cong
\{(x_0,x_1)\in \N^2 : x_0+x_1=j\}/\sim
\]
and
\[
\pi_0(\Bdi(\Z;j))
\cong
\{(x_0,x_1)\in \Z^2 : x_0+x_1=j\}/\sim,
\]
where $(x_0,x_1)\sim (x_0',x_1')$ if and only if $x_0-x_0'$ is even.
This shows that \eqref{dih.7.6} is a bijection.
\end{proof}

\begin{prop}
\label{dih.21}
For a commutative monoid $M$ with involution, there is an equivalence in $\Sp_{\Ztwo}$
\[
\THR(\Sphere[M])
\simeq
\Sphere[\Bdi M].
\]
\end{prop}
\begin{proof}
See \cite{Hog}, and  also \cite[Proposition 5.9]{DMPR21}.
\end{proof}

\begin{prop}
\label{thrlog.5}
Let $A$ be a commutative ring with involution, and let $M$ be a commutative monoid with involution.
Then there is an equivalence in $\Sp_{\Ztwo}$
\[
\EM(A[M])
\simeq
\EM A \wedge \Sphere[M].
\]
\end{prop}
\begin{proof}
If we regard $M$ as a $\Ztwo$-set, then $M$ is a coproduct of copies $\Ztwo$ and $e$.
If $M=e$, then the claim is clear.
Hence it suffices to show that there is an equivalence
\begin{equation}
\label{thrlog.5.1}
\EM((i^* A)^{\oplus \Ztwo})
\simeq
\EM A\wedge \Sigma^\infty (\Ztwo)_+,
\end{equation}
where $(i^* A)^{\oplus \Ztwo}$ is the commutative monoid $A\oplus A$ with the involution given by $(x,y)\mapsto (y,x)$.
By \eqref{etale.17.1} and Propositions \ref{norm.2}, \ref{orth.5}(3), and \ref{etale.33}, we obtain equivalences
\[
\EM A \wedge i_\sharp i^* \Sphere
\simeq
i_\sharp(i^*\EM A \wedge i^*\Sphere)
\simeq
i_\sharp \EM i^*A
\simeq
i_* \EM i^*A 
\simeq
\EM((i^* A)^{\oplus \Ztwo}).
\]
Together with the equivalence $i_\sharp i^* \Sphere\simeq \Sigma^\infty (\Ztwo)_+$, we obtain \eqref{thrlog.5.1}.
\end{proof}

\begin{prop}
\label{thrlog.2}
Let $A$ be a commutative ring with involution, and let $M$ be a commutative monoid with involution.
Then there is a canonical equivalence in $\Sp_{\Ztwo}$
\[
\THR(A[M])
\simeq
\THR(A)\wedge\Sphere[\Bdi M].
\]
\end{prop}
\begin{proof}
By Propositions \ref{etale.1} and \ref{dih.21}, we have equivalences
\[
\THR(\EM A \wedge \Sphere[M])
\simeq
\THR(\EM A)\wedge \THR(\Sphere[M])
\simeq
\THR(A)\wedge \Sphere[\Bdi M].
\]
Proposition \ref{thrlog.5} finishes the proof.
\end{proof}

\section{Properties of real topological Hochschild homology}
\label{sec5}

\subsection{THR of the projective line}\label{periodicity}

We now establish the first computations of $\THR$ for non-affine schemes, namely $\P^1$ and $\P^{\sigma}$. 

\begin{rmk}\label{krpn}
Hermitian $K$-theory $\KO$ resp.\ $\KR$ is not an orientable theory, that is the usual projective bundle formula as for Chow groups and algebraic $K$-theory does not hold. The computation of $\P^1$ over regular rings is \cite[Proposition 6.1]{Ho05}, where a $(8,4)$-motivic periodic spectrum $\KO$ is constructed.
This computation is extended by \cite[Theorem 9.10]{schlichting17} to rather general base schemes (still with 2 invertible). For the projective line with involution, a variation of Schlichting's proof leads to the computation of $\KR(\P^{\sigma})$, see \cite[Theorem 5.1]{Ca} and compare \cite[Theorem 7.1]{Xi} and \cite{HuKO} for different proofs.
In particular, this leads to an equivariant motivic spectrum $\KR$ which is $\P^1 \wedge \P^{\sigma}$-periodic. For further periodicities of $\KR$ see \cite[Theorem 10]{HuKO}. 
In their notation, we have $\P^1 \simeq S^1 \wedge S^{\alpha}$ and $\P^{\sigma} \simeq S^{\gamma} \wedge S^{\gamma \alpha} \simeq \P^1_{-}$. For the definition of $S^{\gamma}=S^{\sigma}$ in the motivic setting and a proof of the last equivalence, we refer to \cite[section 2.5]{Ca}. 
Although $\THR$ is not $\A^1$-invariant, it seems reasonable to expect that the formulas for $\THR(\P^n)$ are similar to those for $\KR$. In the cases considered below this is indeed the case: the following Proposition implies that $\Omega^{1 + \alpha}\THR \simeq \Sigma^{\gamma -1}\THR$. Smashing with $S^1$, we obtain the same periodicity as (35) in \cite{HuKO}. Similarly, the next proposition corresponds to (36) of loc.\ cit. 
\end{rmk}

The following computations rely on Proposition \ref{dih.21} and the computations for dihedral nerves in the previous section. 

\begin{thm}
\label{period.3}
For any $X \in \Sch_{\Ztwo}$,
there is an equivalence of $\Ztwo$-spectra
\[
\THR(X\times \P^1)
\simeq
\THR(X)\oplus \Sigma^{\sigma-1}\THR(X).
\]
\end{thm}
\begin{proof}
For notational convenience, we will write the proof as if everything takes place over $\Sphere$ rather than $X$.
Using the description of $\THR$ for spherical groups rings from Proposition \ref{dih.21} and its extension to log schemes with involution from Proposition \ref{thrlog.2}, we are reduced to consider the following homotopy cocartesian square, corresponding to the standard Zariski cover of $\P^1$ by two copies of $\A^1$ and using Corollary \ref{etale.37} and the description of $\THR(\Sphere[M])$ from Proposition \ref{dih.21}:

\[
\begin{tikzcd}
\THR(\P^1_{\Sphere}) \ar[d]\ar[r]&
\Sphere[\Bdi \N]  \ar[d]
\\
\Sphere[\Bdi (-\N)]\ar[r]&
\Sphere[\Bdi \Z]
\end{tikzcd}
\]

Here the notation $\N$ and $-\N$ indicates the two different embeddings of $\A^1$ in $\P^1$.
It is crucial to notice that even as $\G_m$ has trivial involution, the involution given by the dihedral nerve (see Definition \ref{dih.9} above) yields nontrivial involutions.
Using the decomposition
of \eqref{ndidecomposition} and Propositions \ref{dih.24}, \ref{dih.8}, and \ref{dih.1}, we obtain the following $\Ztwo$-equivariant (homotopy) cocartesian square:
 
\[
\begin{tikzcd}
\THR(\P^1_{\Sphere}) \ar[d]\ar[r]&
\Sphere[*
\amalg
\coprod_{j\geq 1} S^\sigma]  \ar[d]
\\
\Sphere[* \amalg \coprod_{j\leq - 1} S^\sigma] \ar[r]&
\Sphere[\coprod_{j \in \Z}S^{\sigma}]
\end{tikzcd}
\]

An obvious cancellation, using Proposition \ref{dih.7} on the relevant maps,
yields the following $\Ztwo$-equivariant (homotopy) cocartesian square

\[
\begin{tikzcd}
\THR(\P^1_{\Sphere}) \ar[d]\ar[r]&
\Sphere[*] \ar[d]
\\
\Sphere[*] \ar[r]&
\Sphere[S^{\sigma}]
\end{tikzcd}
\]
and the result follows.
\end{proof}

We now turn to the slightly more subtle computation of the projective line $\P^{\sigma}$ with involution. Recall that $(i_\sharp,i^*,i_*)$ denotes the free-forgetful-cofree adjunction for the map $i\colon \pt \to \rB(\Ztwo)$.
Let $\Sigma^\sigma\colon \Sp_{\Ztwo}\to \Sp_{\Ztwo}$ be the functor $\Sigma^\infty S^\sigma \wedge (-)$, which has an inverse functor $\Sigma^{-\sigma}$ since the sphere $\Sigma^\infty S^\sigma$ is $\wedge$-invertible in $\Sp_{\Ztwo}$.
For integers $m$ and $n$, we set $\Sigma^{m+n\sigma}:=\Sigma^m (\Sigma^\sigma)^{\wedge n}$,
which is a functor $\Sp_{\Ztwo}\to \Sp_{\Ztwo}$.
We also set $\Sigma^{m+n\sigma}:=\Sigma^{m+n\sigma}\Sphere\in \Sp_{\Ztwo}$ for abbreviation.

For an adjoint pair of $\infty$-categories $F:\cC\rightleftarrows \cD:G$, let $ad\colon \id \to GF$ (resp.\ $ad'\colon FG\to \id$) denote the unit (resp.\ counit).

\begin{lem}
\label{period.2}
There is a natural equivalence of functors
\[
\Sigma^{-\sigma}
\simeq
\fib(\id\xrightarrow{ad} i_*i^*).
\]
\end{lem}
\begin{proof}
Consider the cocartesian square of $\Ztwo$-spaces
\[
\begin{tikzcd}
S^\sigma-\{(1,0),(-1,0)\}\ar[d]\ar[r]&
S^\sigma-\{(1,0)\}\ar[d]
\\
S^\sigma-\{(-1,0)\}\ar[r]&
S^\sigma.
\end{tikzcd}
\]
There are $\Ztwo$-homotopy equivalences $S^\sigma-\{(1,0),(-1,0)\}\simeq \Ztwo$ and $S^\sigma-\{(1,0)\}\simeq S^\sigma-\{(-1,0)\}\simeq \pt$.
Together with the explicit descriptions of $i_\sharp$ and $i^*$ in Construction \ref{etale.8}, we obtain a natural equivalence
\begin{equation}
\label{period.2.1}
\Sigma^\sigma
\simeq
\cofib(i_{\sharp}i^*\xrightarrow{ad'} \id).
\end{equation}
By adjunction, we obtain the desired natural equivalence.
\end{proof}

\begin{thm}
\label{period.1}
For any $X \in \Sch_{\Ztwo}$, there is an equivalence of $\Ztwo$-spectra
\[
\THR(X\times \P^\sigma)
\simeq
\THR(X)\oplus \Sigma^{1-\sigma}\THR(X).
\]
\end{thm}
\begin{proof}
As before, we will write the proof as if everything takes place over $\Sphere$.
Consider the cartesian equivariant Nisnevich square
\begin{equation}
\label{period.1.3}
\begin{tikzcd}
\Ztwo\times \G_{m,\Sphere}^\sigma\ar[d]\ar[r]&
(\P_{\Sphere}^1-\infty) \amalg (\P_{\Sphere}^1-0)\ar[d]
\\
\G_{m,\Sphere}^\sigma\ar[r]&
\P_{\Sphere}^{\sigma}.
\end{tikzcd}
\end{equation}
as in \cite[Lemma 2.23]{Ca}, where the $\Ztwo$-action on the upper right corner is induced by the action on $\P_{\Sphere}^\sigma$.

Propositions \ref{etale.10} and \ref{etale.35} give equivalences
\[
\THR((\P_\Sphere^1-\infty)\amalg (\P_\Sphere^1-0))
\simeq
i_*\THH(i^*\A_\Sphere^1)
\simeq
i_*i^*\THR(\A_\Sphere^1).
\]
We similarly have an equivalence $\THR(\Ztwo\times \G_{m,\Sphere}^\sigma)\simeq i_*i^*\THR(\G_{m,\Sphere}^\sigma)$ since there are isomorphisms $\Ztwo\times \G_{m,\Sphere}^\sigma \cong \Ztwo \times \G_{m,\Sphere}$ and $i^*\G_{m,\Sphere}^\sigma\cong i^*\G_{m,\Sphere}$.
The induced map $\THR(\G_{m,\Sphere}^\sigma)\to \THR(\Z/2\times \G_{m,\Sphere}^\sigma)$ can be identified with the map
\[
\THR(\G_{m,\Sphere}^\sigma)
\to
i_*i^*\THR(\G_{m,\Sphere}^\sigma)
\]
obtained by the unit of the adjunction pair $(i^*,i_*)$.
As in the proof of Theorem \ref{period.3},
use Propositions \ref{dih.24}, \ref{dih.8}, \ref{dih.1}, and \ref{dih.7} to see that the induced map $\THR(\A_{\Sphere}^1)\to \THR(\G_{m,\Sphere})$ can be identified with the map
\[
\Sphere \oplus \bigoplus_{j>0} \Sigma_+^\infty S^\sigma
\to
\Sphere \oplus \bigoplus_{j>0} i_*i^*\Sigma_+^\infty S^\sigma
\]
obtained by the unit of the adjunction pair $(i^*,i_*)$.
By \eqref{dih.12.1}, we obtain an equivalence
\[
\THR(\G_{m,\Sphere}^\sigma)
\simeq
\Sphere \oplus \bigoplus_{j>0} i_*i^*\Sigma_+^\infty S^1.
\]
Applying $\THR$ to \eqref{period.1.3} and combining with the above discussion yield the following homotopy cocartesian square:
\begin{equation}
\label{period.1.2}
\begin{tikzcd}[column sep=small, row sep=small]
\THR(\P_{\Sphere}^\sigma)\ar[dd]\ar[r]&
i_*i^*(\Sphere \oplus \bigoplus_{j> 0} \Sigma^\infty_+ S^\sigma)\ar[d,"\simeq"]
\\
&
i_*i^*(\Sphere \oplus \bigoplus_{j>0} \Sigma^\infty_+ S^1)\ar[d]
\\
\Sigma_+^\infty S^1\oplus \bigoplus_{j>0} i_*i^* \Sigma_+^\infty S^1 \ar[r]&
i_*i^*(\Sigma_+^\infty S^1\oplus \bigoplus_{j>0} i_*i^* \Sigma_+^\infty S^1) 
\end{tikzcd}
\end{equation}

Consider the commutative square
\begin{equation}
\label{period.1.1}
Q
:=
\begin{tikzcd}[column sep=large]
0\ar[d]\ar[rr]&&
\bigoplus_{j>0} i_*i^*\Sigma_+^\infty S^1\ar[d]
\\
\bigoplus_{j>0} i_*i^*\Sigma_+^\infty S^1\ar[rr]&&
\bigoplus_{j>0} i_*i^*i_*i^*\Sigma_+^\infty S^1
\end{tikzcd}
\end{equation}
extracted from \eqref{period.1.2}, where the right vertical map (resp.\ lower horizontal map) is obtained by applying $i_*i^*$ to the left (resp.\ right) of the unit map $\id \to i_*i^*$.
Since $\Phi^{\Ztwo} i_*\simeq 0$ by Proposition \ref{orth.5}(3),(5), $\Phi^{\Ztwo}Q$ is cartesian.
The objects in the square $i^*Q$ are direct sums of $\bigoplus_{j>0}\Sigma_+^\infty S^1$, and the right vertical and lower horizontal maps are the matrix multiplications given by
\[
\left(
\begin{array}{cc}
1 & 0\\
1 & 0\\
0 & 1\\
0 & 1
\end{array}
\right)
\quad
\left(
\begin{array}{cc}
1 & 0\\
0 & 1\\
0 & 1\\
1 & 0
\end{array}
\right)
\]
for certain choices of bases. From this, one can check that $i^*Q$ is cartesian too.
It follows that $Q$ is cartesian.

Hence the other direct summand of the square \eqref{period.1.2}
\[
\begin{tikzcd}
\THR(\P_{\Sphere}^\sigma)\ar[r]\ar[d]&
i_*i^*\Sigma^0\ar[d]
\\
\Sigma^0\oplus \Sigma^1\ar[r]&
i_*i^*\Sigma^0\oplus i_*i^*\Sigma^1
\end{tikzcd}
\]
is also cartesian.
It follows that $\THR(\P_\Sphere^\sigma)$ is equivalent to the direct sum
\[
\lim(\Sigma^0\to i_*i^*\Sigma^0\leftarrow i_*i^*\Sigma^0) \oplus \lim(\Sigma^1\to i_*i^*\Sigma^1\leftarrow 0).
\]
Together with Lemma \ref{period.2}, we obtain the desired equivalence.
\end{proof}

\begin{prop}
There is an equivalence of $\Ztwo$-spectra
\[
\THR(X)
\simeq
\Omega_{\P^1\wedge \P^\sigma}\THR(X).
\]
\end{prop}
\begin{proof}
Combine Theorems \ref{period.3} and \ref{period.1}.
\end{proof}

\subsection{THR of projective spaces}

\begin{df}
As usual, for any integer $n\geq 1$, we consider the $n$-cube $(\Delta^1)^n$ as a partially ordered set, and use the same symbol for the associated category. For an $\infty$-category $\cC$, an \emph{$n$-cube in $\cC$} is a functor
\[
Q\colon \Nerve (\Delta^1)^n \to \cC,
\]
compare \cite[Definition 6.1.1.2]{HA}.
If $\cC$ admits limits, the \emph{total fiber of $Q$} is defined to be
\[
\tfib(Q)
:=
\fib(Q(0,\ldots,0) \to \lim Q|_{(\Delta^1)^n-\{(0,\ldots,0)\}}).
\]
\end{df}

The following $\infty$-categorical result can be shown by dualizing the arguments from \cite[Proposition A.6.5]{BPO2}.

\begin{prop}
\label{proj.1}
Let $Q$ be an $n$-cube in an $\infty$-category $\cC$ with small limits, where $n$ is a nonnegative integer.
Then for every integer $1\leq i\leq n$, there exists a fiber sequence
\[
\tfib(Q)
\to
\tfib(Q|_{(\Delta^1)^{i-1}\times \{0\} \times (\Delta^1)^{n-i-1}})
\to
\tfib(Q|_{(\Delta^1)^{i-1}\times \{1\} \times (\Delta^1)^{n-i-1}}).
\]
\end{prop}

\begin{prop}
\label{proj.5}
Let $\cC$ be a symmetric monoidal stable $\infty$-category with small limits such that the tensor product operation on $\cC$ preserves fiber sequences in each variable.
If $Q$ is an $n$-cube and $f\colon X_0\to X_1$ is a map in $\cC$, then there is a canonical equivalence
\[
\tfib(Q)\otimes \fib(f)
\simeq
\tfib(Q\otimes f),
\]
where $Q\otimes f$ is the associated $(n+1)$-cube sending $(a_1,\ldots,a_{n+1})\in (\Delta^1)^{n+1}$ to $Q(a_1,\ldots,a_n)\otimes X_{a_{n+1}}$.
\end{prop}
\begin{proof}
This is again obtained by dualizing the arguments for the corresponding one \cite[Proposition A.6.7]{BPO2}.
An intermediate step is to show $\tfib(Q)\otimes X\simeq \tfib(Q\otimes X)$ for any $X\in \cC$, where $Q\otimes X$ is the associated $n$-cube sending $(a_1,\ldots,a_n)\in (\Delta^1)^n$ to $Q(a_1,\ldots,a_n)\otimes X$.
Then one can use Proposition \ref{proj.1}.
\end{proof}

\begin{prop}
\label{proj.4}
Let $\cC$ be a symmetric monoidal stable $\infty$-category with small limits such that the tensor product operation on $\cC$ preserves fiber sequences
in each variable.
If $i_1\colon  X_{0,1} \to  X_{1,1}$, $\ldots$, $i_n \colon  X_{0,n} \to X_{1,n}$ are maps in $\cC$, then there is a canonical equivalence
\[
\fib(i_1)\otimes \cdots \otimes \fib(i_n) \simeq \tfib(i_1\otimes \cdots \otimes i_n),
\]
where $i_1\otimes \cdots \otimes i_n$ is the associated $n$-cube sending $(a_1,\ldots,a_n)\in (\Delta^1)^n$ to $X_{a_1,1}\otimes \cdots \otimes X_{a_n,n}$ and arrows given by tensor products of $i_j$s and identities.
\end{prop}
\begin{proof}
Use Proposition \ref{proj.5} repeatedly.
\end{proof}

\begin{prop}
\label{proj.3}
Suppose $X\in \Sch_{\Ztwo}$, and let $\{U_1,\ldots,U_n\}$ be a Zariski covering of $X$ with the induced involutions.
Let $Q$ be the $S$-cube given by
\[
Q(\{i_1,\ldots,i_r\}):=U_{i_1}\cap \cdots \cap U_{i_r}
\]
for nonempty $\{i_1,\ldots,i_r\}\subset S:=\{1,\ldots,n\}$ and $Q(\emptyset):=X$.
Then there is an equivalence
\[
\tfib(\THR(Q))
\simeq
0.
\]
\end{prop}
\begin{proof}
Let us use the equivalence $\Nerve(\mathbf{P}(S))\simeq(\Delta^1)^n$.
We include the description of this equivalence if $S=\{1\}$.
We regard the partially ordered set $\mathbf{P}(\{1\})$ as the category associated with the diagram $\emptyset \to \{1\}$.
The nerve of this category is $\Delta^1$.

We proceed by induction on $n$.
The claim is clear if $n=1$.
Assume $n>1$.
By induction, we fiber sequences
\[
\tfib(\THR(Q|_{\{0\}\times (\Delta^1)^{n-1}}))
\to
\THR(X)
\to
\THR(U_2\cup \cdots \cup U_n)
\]
and
\[
\tfib(\THR(Q|_{\{1\}\times (\Delta^1)^{n-1}}))
\to
\THR(U_1)
\to
\THR(U_1\cap (U_2\cup \cdots \cup U_n)).
\]
Together with Proposition \ref{proj.1}, we reduce to showing that the induced square
\[
\begin{tikzcd}
\THR(X)\ar[r]\ar[d]&
\THR(U_2\cup \cdots \cup U_n)\ar[d]
\\
\THR(U_1)\ar[r]&
\THR(U_1\cap (U_2\cup \cdots \cup U_n))
\end{tikzcd}
\]
is cartesian.
This follows from Corollary \ref{etale.37}.
\end{proof}

\begin{thm}
\label{proj.2}
For any $X\in \Sch_{\Ztwo}$ and integer $n\geq 0$, there is an equivalence of $\Ztwo$-spectra
\[
\THR(X\times \P^{n})
\simeq
\left\{
\begin{array}{ll}
\THR(X)
\oplus
\bigoplus_{j=1}^{\lfloor n/2 \rfloor} i_*\THH(X) &
\text{if $n$ is even,}
\\
\THR(X)
\oplus
\bigoplus_{j=1}^{\lfloor n/2 \rfloor} i_*\THH(X) \oplus \Sigma^{n(\sigma-1)}\THR(X) &
\text{if $n$ is odd.}
\end{array}
\right.
\]
\end{thm}
\begin{proof}
Proposition \ref{orth.5}(3) allows us to replace $i_*$ by $i_\sharp$ in the claim.
We set
\[
M_j:=\{(x_1,\ldots,x_n)\in \Z^n : x_j\geq 0\}
\]
for $j=1,\ldots,n$ 
and
\[
 M_{n+1}:=\{(x_1,\ldots,x_n)\in \Z^n : x_1+\cdots+ x_n \leq 0\}.
\]
Observe that there is an isomorphism of commutative monoids
$M_j\cong \Z^{n-1}\oplus \N$ for all $j=1,\ldots,n+1$.
We set $M_I:=M_{i_1}\cap \cdots \cap M_{i_q}$ for all nonempty subsets $I:=\{i_1,\ldots,i_q\}$ of $\{1,\ldots,n+1\}$, and consequently
$M_\emptyset := \Z^n$.
Together with the obvious maps, we obtain an $(n+1)$-cube $M$ in commutative monoids associated with $M_I$ for any subset $I\subset \{1,\ldots,n+1\}$.
Here we set $M(b_1,...,b_{n+1})=M_I$, where $I$ is the set of indices $i$ such that $b_i=0$.
By \eqref{ndidecomposition}, we have a canonical decomposition
\[
\Bdi M_I
\simeq
\coprod_{v\in \Z^n}
\Bdi(M_I;v),
\]
where $\Bdi(M_I;v):=\emptyset$
if $v\notin M_I$.
Hence for every $v$ the above $\Bdi(M_I;v)$ assemble to an $(n+1)$-cube in $\Ztwo$-spaces, and a decomposition of the $(n+1)$-cube $\Bdi M$ into these smaller cubes as $v$ varies in $\Z^n$.
Combine Propositions \ref{dih.26} and \ref{dih.7} to show that the induced map in this cube
\[
\Bdi(\N \oplus \N^s \oplus \Z^{n-s-1};v)
\to
\Bdi(\Z \oplus \N^s \oplus \Z^{n-s-1};v)
\]
is an equivalence
for every integer $0\leq s\leq n-1$ if the first coordinate of $v$ is greater than $0$.
By a change of coordinates in the target, we see that the induced map
\[
\Bdi(M_{I\cup \{j\}};v)
\to
\Bdi(M_I;v)
\]
is an equivalence for all $j=1,\ldots,n+1$, subsets $I$ of $\{1,\ldots,n+1\}-\{j\}$, and $v\in M_j^+$, where $M_j^+$ denotes the set of non-units of $M_j$.
Use Proposition \ref{proj.1} repeatedly to show
\begin{equation}
\label{proj.2.1}
\tfib(\Sphere[\Bdi(M;v)])
\simeq
0
\end{equation}
for all $v\in M_j^+$.
This means $\tfib(\Sphere[\Bdi(M;v)])\simeq 0$ whenever $v\neq O$ since $M_\emptyset-(M_1^+\cup \cdots \cup M_{n+1}^+)=\{O\}$, where $O$ denotes the origin in $\Z^n$.

Now consider the standard cover of $\P^n_{\Sphere}$ by $(n+1)$ copies of $\A^n_{\Sphere}$,
and recall that $\A^n_{\Sphere}$ is the spherical monoid ring of $\N^n$.
(We continue switching between spherical monoid rings and honest schemes over rings as before. Also, note that by Proposition \ref{thrlog.2} products of affine schemes correspond to products of monoids when computing $\THR$.) Choosing suitable coordinates, the intersections of $s$ elements of the cover with $0<s\leq n+1$ are given by (the spherical monoid rings of) $M_I$ above such that $\lvert I \rvert = n+1-s$. By Proposition \ref{proj.3},
we have an equivalence
\[
\THR(\P_{\Sphere}^n)
\simeq
\lim \Sphere[\Bdi M|_{(\Delta^1)^n-\{(0,\ldots,0)\}}].
\]
Together with \eqref{proj.2.1}, we have an equivalence
\[
\THR(\P_{\Sphere}^n)
\simeq
\lim \Sphere[\Bdi (M;O)|_{(\Delta^1)^n-\{(0,\ldots,0)\}}]
\]
since $\Bdi(M_{\{1,\ldots,n\}};O)\cong \Bdi(M_{\{1,\ldots,n\}})$.
For every subset $I$ of $\{1,\ldots,n\}$, we have the canonical decomposition
\[
\Bdi(M_I;O)
\cong
V_I\amalg *,
\]
where $V_I$ is obtained by removing the base point $*$ of $\Bdi(M_I;O)$ corresponding to the element $0\in M_I$ in simplicial degree $0$.
This yields the canonical decomposition
\[
\Sphere[\Bdi (M;O)]
\simeq
Q\oplus Q',
\]
where every entry of $Q'$ is $\Sigma^\infty_+ * \simeq \Sphere$.
Use Proposition \ref{proj.1} repeatedly to have an equivalence $\tfib(Q') \simeq 0$.
This implies that we have an equivalence
\[
\lim Q'|_{(\Delta^1)^n-\{(0,\ldots,0)\}}
\simeq
\Sphere.
\]
On the other hand, we have $Q(0,\ldots,0)=0$ since $M_{\{1,\ldots,n\}}=0$.
This implies that we have an equivalence
\[
\tfib(Q)
\simeq
\Sigma^{-1} \lim Q|_{(\Delta^1)^n-\{(0,\ldots,0)\}}.
\]
Combine what we have discussed above to have an equivalence
\begin{equation}
\label{proj.2.3}
\THR(\P_{\Sphere}^n)
\simeq
\Sphere
\oplus
\Sigma^1\tfib(Q).
\end{equation}

We claim that
\[
\tfib(Q|_{(\Delta^1)^{d+1}\times \{0\}^{n-d}})
\simeq
\left\{
\begin{array}{ll}
\bigoplus_{j=1}^{\lfloor d/2 \rfloor} i_\sharp\Sigma^{-1} & \text{if $d$ is even},
\\
\bigoplus_{j=1}^{\lfloor d/2 \rfloor} i_\sharp\Sigma^{-1} \oplus \Sigma^{d\sigma-d-1}& \text{if $d$ is odd}.
\end{array}
\right.
\]
Let us proceed by induction on $d$.
The claim is clear if $d=0$ by the definition of the total fiber.
Assume $0<d\leq n$.
Let $\{e_1,\ldots,e_n\}$ be the standard basis in $\Z^n$.
The $d$-cube $M|_{(\Delta^1)^{d} \times \{1\} \times \{0\}^{n-d}}$ is isomorphic to the naturally associated $d$-cube sending $(a_1,\ldots,a_d)\in (\Delta^1)^d$ to \[
P_{a_1}(e_1-e_{d+1})\oplus \cdots \oplus P_{a_d}(e_d-e_{d+1})\oplus \N (-e_{d+1}) \oplus \N(e_{d+2}-e_{d+1})\oplus \cdots \oplus \N (e_{n}-e_{d+1}),
\]
where $P_0:=\N$, $P_1:=\Z$, and $e_{n+1}:=0$.
Combine Propositions \ref{dih.27} and \ref{proj.4} to obtain an equivalence
\begin{align*}
&\tfib(\Sphere[\Bdi(M|_{(\Delta^1)^{d}\times \{1\} \times  \{0\}^{n-d}};O)])
\\
\simeq &
\fib(\Sphere[\Bdi(\N;0)]\to \Sphere[\Bdi(\Z;0)])^{\wedge d}
\wedge
\Sphere[\Bdi(\N;0)]^{\wedge n-d}.
\end{align*}
Use Proposition \ref{proj.1} repeatedly to deduce an equivalence
\[
\tfib(Q'|_{(\Delta^1)^{d}\times \{1\} \times  \{0\}^{n-d}}) \simeq 0.
\]
Together with Proposition \ref{dih.1} and \eqref{dih.8.4}, we obtain equivalences
\begin{align*}
\tfib(Q|_{(\Delta^1)^{d}\times \{1\} \times  \{0\}^{n-d}})
\simeq &
\fib(\Sphere \to \Sphere[S^\sigma])^{\wedge d} \wedge \Sphere^{\wedge n-d}
\\
\simeq &
\Sigma^{-d} (\Sigma^\infty S^\sigma/(1,0))^{\wedge d}
\simeq
\Sigma^{d\sigma-d},
\end{align*}
where $(1,0)$ is the base point of $S^\sigma$.
Proposition \ref{proj.1} gives a fiber sequence
\[
\tfib(Q|_{(\Delta^1)^{d+1}\times   \{0\}^{n-d}})
\to
\tfib(Q|_{(\Delta^1)^{d}\times \{0\} \times  \{0\}^{n-d}})
\to
\tfib(Q|_{(\Delta^1)^{d}\times \{1\} \times  \{0\}^{n-d}}).
\]

If $d$ is odd, then we obtain a fiber sequence
\begin{equation}
\label{proj.2.2}
\tfib(Q|_{(\Delta^1)^{d+1}\times \{0\}^{n-d}})
\to
\bigoplus_{j=1}^{\lfloor d/2 \rfloor} i_\sharp\Sigma^{-1}
\xrightarrow{f}
\Sigma^{d\sigma-d}
\end{equation}
by induction.
Since $\pi_{-1}(\Sigma^0)=0$,
the map $\Sigma^{-1}\to i^*\Sigma^{d\sigma-d}\simeq \Sigma^0$ obtained by adjunction is equivalent to $0$.
It follows that $f$ is equivalent to $0$, i.e., \eqref{proj.2.2} splits.
This completes the induction argument for odd $d$.

If $d$ is even, we obtain a fiber sequence
\[
\tfib(Q|_{(\Delta^1)^{d+1}\times \{0\}^{n-d}})
\to
\bigoplus_{j=1}^{\lfloor d/2 \rfloor-1} i_\sharp\Sigma^{-1} \oplus \Sigma^{(d-1)\sigma-d}
\to
\Sigma^{d\sigma-d}
\]
by induction.
As above, the induced map $\bigoplus_{j=1}^{\lfloor d/2 \rfloor-1} i_\sharp\Sigma^{-1} \to \Sigma^{d\sigma-d}$ is equivalent to $0$.
It follows that we have an equivalence
\[
\tfib(Q|_{(\Delta^1)^{d+1}\times \{0\}^{n-d}})
\simeq
\bigoplus_{j=1}^{\lfloor d/2 \rfloor-1} i_\sharp\Sigma^{-1} 
\oplus
\fib(\Sigma^{(d-1)\sigma-d}\xrightarrow{g}\Sigma^{d\sigma-d}).
\]

We now analyze the non-trivial map $g\colon \Sigma^{(d-1)\sigma-d} \to \Sigma^{d\sigma-d}$. On the level of commutative monoids, this corresponds to the inclusion
\begin{align*}
&\Z (u_1-u_d)\oplus \cdots \oplus \Z (u_{d-1}-u_d) \oplus \N(-e_d) \oplus \N(e_{d+1}-e_d) \oplus \cdots \oplus \N(e_n-e_d)
\\
\to &
\Z u_1\oplus \cdots \oplus \Z u_d \oplus \N u_{d+1} \oplus \cdots \oplus \N u_{n},
\end{align*}
where $u_i:=e_i-e_{d+1}$ for $i\in \{1,\ldots,n\}-\{d+1\}$ and $u_{d+1}:=e_{d+1}$.
As the $0$-entry for $\Bdi \N$ is just a point, we only need to study the homomorphism  $\Z^{d-1}\to \Z^d$ given by
\[
(a_1,\ldots,a_{d-1})
\mapsto
(a_1,\ldots,a_{n-1},-a_1-\cdots-a_{d-1}).
\]
In the degree $(0,\ldots,0)$, this is easily seen to induce via Proposition \ref{dih.8}  the map $(S^{\sigma})^{\times d-1} \to (S^{\sigma})^{\times d}$ given by
\[
(x_1,\ldots,x_{d-1})\mapsto (x_1,\ldots,x_{d-1},x_1^{-1}\cdots x_{d-1}^{-1}).
\]
After a further cancellation of base points, we are left with studying the map $h\colon (S^{\sigma})^{\wedge d-1}\to (S^{\sigma})^{\wedge d}$.
This is an equivariant cofibration, as it is the $(S^{\sigma})^{\wedge d-1}$-suspension of the push-out of the cofibration $G \times S^{0} \to G \times I$ along the projection $G \times S^0 \to (G/G) \times S^0$. Hence unstably the equivariant (homotopy)  cofiber of $f$ is given by $(S^{\sigma})^{\wedge d}/h((S^{\sigma})^{\wedge d-1})\simeq S^d \vee S^d$ where $G=\Ztwo$ acts on the latter by switching the spheres. Thus the stable equivariant (homotopy) fiber $\fib(g)$ is given by $\Sigma^{-d-1}\Sigma^{\infty}(S^{d} \vee S^{d}) \simeq i_\sharp \Sigma^{-1}$ using \eqref{etale.8.1}.
This completes the induction argument for even $d$.

Together with \eqref{proj.2.3}, we obtain an equivalence
\[
\THR(\P_{\Sphere}^{n})
\simeq
\left\{
\begin{array}{ll}
\Sphere
\oplus
\bigoplus_{j=1}^{\lfloor n/2 \rfloor} i_\sharp\Sphere &
\text{if $n$ is even,}
\\
\Sphere
\oplus
\bigoplus_{j=1}^{\lfloor n/2 \rfloor} i_\sharp\Sphere \oplus \Sigma^{n(\sigma-1)} &
\text{if $n$ is odd.}
\end{array}
\right.
\]

Use Propositions \ref{thrlog.2} and \ref{proj.3} for the standard cover of $\P^n$ to obtain an equivalence
\[
\THR(X\times \P^n)
\simeq
\THR(X)\wedge \THR(\P_{\Sphere}^n)
\]
whenever $X$ is an affine scheme with involution.
Use Proposition \ref{proj.3} again to generalize this equivalence to the case when $X$ is a separated scheme with involution.
Hence to obtain the desired equivalence, it suffices to obtain an equivalence
\[
\THR(X)\wedge i_\sharp\Sphere
\simeq
i_\sharp\THH(X).
\]
This follows from Propositions \ref{norm.2} and \ref{etale.10}.
\end{proof}

Note that this result is compatible with the projective bundle theorem for the oriented theory $\THH$, see \cite{BM}, after applying $i^*$ and using Propositions \ref{etale.35} and \ref{orth.5}.

For all $X\in \Sch_{\Ztwo}$ and integers $m$, we set
\begin{equation}
\THO^{[m]}(X)
:=
(\Sigma^{m(1-\sigma)}\THR(X))^{\Ztwo}.
\end{equation}
Recall that $(-)^{\Ztwo}$ commutes with $\Sigma^1$, but not with $\Sigma^\sigma$.

Looking at fixed points, the result becomes 
\begin{equation}
\THO^{[m]}(X\times \P^{n})
\simeq
\left\{
\begin{array}{ll}
\THO^{[m]}(X)
\oplus
\bigoplus_{j=1}^{\lfloor n/2 \rfloor} \THH(X) &
\text{if $n$ is even,}
\\
\THO^{[m]}(X)
\oplus
\bigoplus_{j=1}^{\lfloor n/2 \rfloor} \THH(X) \oplus \THO^{[m-n]}(X) &
\text{if $n$ is odd.}
\end{array}
\right.
\end{equation}

\begin{rmk}\label{thrpn}
As for $n=1$ in the previous subsection, the formula for $\THO(X\times \P^n)$ corresponds to the one for $\KO(X\times \P^n)$.
Indeed, for higher dimensional projective spaces $\P^n$ with trivial involution, $\KO=\KR$ has been recently computed by \cite{Rohrbach22} and \cite{KSW21}. Analyzing the arguments of \cite{Rohrbach22}, one sees that for even $n$ the results for $\KR(\P^n)$ with and without involution are the same.
However, when trying to compute $\THR$ of higher dimensional $\P^n$ with involution, the standard cover of $\P^n$ will not respect the involution. For the one-dimensional $\P^{\sigma}$ we used the square in \eqref{period.1.3} instead, and for $(\P^n,\tau)$ in general one would have to construct more complicated cubes that respect the involution $\tau$.
\end{rmk}

\appendix

\section{Equivariant homotopy theory}
\label{secA}
The purpose of this appendix is to review equivariant homotopy theory.
Throughout this section, $G$ is a finite group.

\subsection{\texorpdfstring{$\infty$-}{Infinity }categories of equivariant spectra}\label{equivinfty.sec}

In this subsection we review the $\infty$-categorical formulation of equivariant homotopy theory following Bachmann and Hoyois \cite{BH21}.
We restrict to finite groupoids, although Bachmann and Hoyois deal more generally with profinite groupoids.
This approach will be compared to more classical references like \cite{HHR} in section \ref{orth.sec} below.

\begin{df}
Let $\FinGpd$ denote the $2$-category of finite groupoids, that is those with only finitely many objects and morphisms.
A morphism in $\FinGpd$ is called a \emph{finite covering} if its fibers (which by our  assumptions have only finitely many objects) automatically are sets, i.e., do not have non-trivial automorphisms.
Recall that the fiber of a $1$-morphism $f\colon Y\to X$ over a point $*$ in $X$ is $Y\times_X *$.
For $X\in \FinGpd$, let $\Fin_X$ denote the category of finite coverings of $X$.
There is an equivalence between $\Fin_{\rB G}$ and the category of finite $G$-sets by \cite[Lemma 9.3]{BH21}.

For a morphism $f\colon Y\to X$ in $\FinGpd$, there is an adjunction
\begin{equation}
\label{norm.0.1}
f^*:\Fin_X \rightleftarrows \Fin_Y : f_*,
\end{equation}
where $f^*$ sends $V\in \Fin_X$ to $V\times_X Y$.
If $f$ is a finite covering, then there is an adjunction
\begin{equation}
\label{norm.0.2}
f_\sharp: \Fin_Y \rightleftarrows \Fin_X : f^*,
\end{equation}
where $f_\sharp$ sends $V\in \Fin_Y$ to $V$, compare the  paragraph preceding \cite[section 9.2]{BH21}.
\end{df}

\begin{exm}\label{setadj}
If $i$ is the obvious morphism of groupoids $\pt\to \rB G$, then $i_{\sharp}E=\coprod_G E$ and $(i_\sharp,i^*)$ is the usual free-forgetful adjunction between sets and $G$-sets.
On the other hand, for any finite set $E$ the $G$-set $i_*E$ is isomorphic to $\prod_G E$ with $G$ acting by permuting the indices.

For $p:\rB G\to \pt$ we have $p_* E = E^G$ for every finite $G$-set $E$. Note that $p_{\sharp} E$ is not defined as $p$ is not a finite covering, although for finite $G$-sets $E$ the left adjoint to $p^*$ exists, and is given by the orbit set $E/G$.
\end{exm}

\begin{df}
For a category $\cC$ with pull-backs, let $\Span(\cC)$ denote the category of spans, whose objects are the same as $\cC$, whose morphisms are given by the diagrams $(X\xleftarrow{f} Y \xrightarrow{p} Z)$ in $\cC$, and whose compositions of morphisms are given by pullbacks.
A morphism $(X\xleftarrow{f} Y \xrightarrow{p} Z)$ is called a \emph{forward morphism} (resp.\ \emph{backward morphism}) if $f=\id$ (resp.\ $p=\id$).
The notion of spans can be generalized to the case when $\cC$ is an $\infty$-category, see \cite[section 5]{Bar17} for the details.
\end{df}

\begin{const}
\label{norm.3}
In \cite[section 9.2]{BH21}, Bachmann and Hoyois construct three functors $\infH,
\infHpt$, and $\infSH$ on $\FinGpd$ by certain presheaves on $\Fin_X$, and then further refine these to functors
\begin{equation}
\label{norm.3.1}
\infH^\otimes,
\infHpt^\otimes,
\infSH^\otimes
\colon
\Span(\FinGpd)\to \CAlg(\Cat_\infty),
\;
(X\xleftarrow{f} Y\xrightarrow{p} Z)\mapsto p_\otimes f^*
\end{equation}
together with natural transformations
\begin{equation}
\label{norm.3.2}
\infH^\otimes
\xrightarrow{(-)_+}
\infHpt^\otimes
\xrightarrow{\Sigma^\infty}
\infSH^\otimes.
\end{equation}
Let us explain parts of their construction.

For an $\infty$-category $\cC$ with finite coproducts, let $\PSigma(\cC)$ denote the $\infty$-category of presheaves of spaces which transform finite coproducts into finite products.
For $X\in \FinGpd$, we set $\infH(X):=\PSigma(\Fin_X)$.
Then we set $\infHpt(X):=\infH(X)_*$, which is the $\infty$-category of pointed objects in $\infH(X)$.
As claimed in \cite[p.\ 81]{BH21}, for $X=\rB G$ these yield the usual $\infty$-categories of $G$-spaces and pointed $G$-spaces.

For a morphism $f\colon Y\to X$, the functor $f^*$ for $\infH$ and $\infHpt$ is induced by \eqref{norm.0.1},
and $f^*$ admits a right adjoint $f_*$.
For $\infHpt$, $f_\otimes$ is a symmetric monoidal functor preserving sifted colimits such that $f_\otimes(V_+)\simeq f_*(V)_+$ for $V\in \Fin_X$.
If $f$ is a finite covering, $f^*$ for $\infH$ and $\infHpt$ admits a left adjoint $f_\sharp$.

We obtain $\infSH(X)$
from $\infHpt(X)$ by $\otimes$-inverting $p_\otimes(S^1)$ for all finite coverings $p\colon Y\to X$.
The functor $f^*$ for $\infSH$ is induced by that for $\infHpt$.
This admits a right adjoint $f_*$, and this admits a left adjoint $f_\sharp$ if $f$ is a finite covering.
The functor $f_\otimes$ for $\infSH$ is the unique symmetric monoidal functor preserving sifted colimits such that the square
\[
\begin{tikzcd}
\infHpt(Y)\ar[d,"\Sigma^\infty"']\ar[r,"f_\otimes"]&
\infHpt(X)\ar[d,"\Sigma^\infty"]
\\
\infSH(Y)\ar[r,"f_\otimes"]&
\infSH(X)
\end{tikzcd}
\]
commutes.
Furthermore, if $f$ has connected fibers, then $f_\otimes$ preserves colimits.

\begin{prop}
\label{norm.5}
Suppose $X_1,\ldots,X_n\in \FinGpd$.
Then there exists a canonical equivalence
\[
\infSH(X_1\amalg \cdots \amalg X_n)
\simeq
\infSH(X_1)\times \cdots \times \infSH(X_n).
\]
\end{prop}
\begin{proof}
This is a consequence of \cite[Lemma 9.6]{BH21}.
\end{proof}

There is an alternative construction only inverting $S^1$: By \cite[Proposition 9.11]{BH21}, we have an equivalence
\begin{equation}
\label{norm.1.1}
\infSH(X)
\simeq
\Sp(\PSigma(\Span(\Fin_X))).
\end{equation}
\end{const}

\begin{prop}
\label{norm.1}
Suppose $X\in \FinGpd$.
Then the family
\[
\{\Sigma^n \Sigma^\infty V_+ : V\in \Fin_X,n\in \Z\}
\]
compactly generates $\infSH(X)$.
In other words, the functor $\Map_{\infSH(X)}(\Sigma^n\Sigma^\infty V_+,-)$ preserves filtered colimits for all $V\in \Fin_X$ and $n\in \Z$, and the family of functors
\begin{equation}
\label{norm.1.2}
\{\Map_{\infSH(X)}(\Sigma^n\Sigma^\infty V_+,-):V\in \Fin_X,n\in \Z\}
\end{equation}
is conservative.
\end{prop}
\begin{proof}
Using \eqref{norm.1.1}, we see that 
a map $\cG\to \cF$ in $\infSH(X)$ is an equivalence if and only if the induced map $\cG(V)\to \cF(V)$ is an equivalence for all $V\in \Fin_X$.
This is further equivalent to saying that the induced morphism $\Omega^\infty \Sigma^{-n} \cG(V)\to \Omega^\infty \Sigma^{-n} \cF(V)$ is an equivalence for all $V\in \Fin_X$ and $n\in \Z$.
This proves that \eqref{norm.1.2} is conservative.
For the other claim, we reduce to the case when $X=\rB G$ using Proposition \ref{norm.5} since every finite groupoid is equivalent to a finite disjoint union of classifying spaces.
Then combine \cite[Lemma I.5.3]{LMS86} and \cite[Proposition 1.4.4.1(3)]{HA} to conclude.
See also \cite[Theorem 9.4.3]{HPS97}.
\end{proof}

Let $\Fold_X$ denote the full subcategory of $\Fin_X$ consisting of the finite fold maps
\[
(\id,\ldots,\id)
\colon
X\amalg \cdots \amalg X \to X,
\]

\begin{prop}
\label{norm.6}
Let $f\colon X^{\amalg n} \to X$ be the $n$-fold map, where $X\in \FinGpd$ and $n\geq 1$ is an integer.
Then the composite
\[
\infSH(X)^{\times n}
\xrightarrow{\simeq}
\infSH(X^{\amalg n})
\xrightarrow{f_\otimes}
\infSH(X)
\]
is the $n$-fold smash product.
\end{prop}
\begin{proof}
One can show an analogous claim for $\infHpt$ as in \cite[Theorem 3.3(6)]{BH21}.
To obtain the claim for $\infSH$, use \cite[Lemma 4.1]{BH21}.
\end{proof}

\begin{df}[{\cite[Definition 9.14]{BH21}}]
Suppose $X\in \FinGpd$.
A \emph{normed $X$-spectrum} is a section of $\infSH^\otimes$
over $\Span(\Fin_X)$ that is cocartesian over the backward morphisms.
Let $\NAlg(\infSH(X))$ denote the $\infty$-category of normed $X$-spectra.
\end{df}

Mapping $X$ to $\pt$ yields an equivalence between $\Fold_X$ and $\Fold_{\pt}$, and forgetting the map to $\pt$ $\Fold_{\pt}$ is obviously equivalent to the category of finite sets. 
By \cite[Corollary C.2]{BH21}, $\CAlg(\infSH(X))$ is equivalent to the $\infty$-category of sections of $\infSH^{\otimes}$ over $\Span(\Fold_X)$ that is cocartesian over the backward morphisms.
Hence there is a forgetful functor
\begin{equation}
\NAlg(\infSH(X))
\to
\CAlg(\infSH(X)),
\end{equation}
which is conservative and preserves colimit and limits as in \cite[Proposition 7.6(3)]{BH21}.
This is an equivalence if $X=\pt$ since $\Fold_\pt\simeq \Fin_\pt$.
The forgetful functor
\begin{equation}
\CAlg(\infSH(X))
\to
\infSH(X)
\end{equation}
is conservative by \cite[Lemma 3.2.2.6]{HA}.
It follows that the composite forgetful functor $\NAlg(\infSH(X))\to \infSH(X)$ is conservative too.

Suppose $X\in \FinGpd$ and $R\in \CAlg(\infSH(X))$.
There is an induced commutative square
\[
\begin{tikzcd}
\CAlg(\Mod_R)\ar[d]\ar[r,"U"]&
\CAlg(\infSH(X))\ar[d]
\\
\Mod_R\ar[r]&
\infSH(X).
\end{tikzcd}
\]
The vertical functors are the forgetful functors, which is symmetric monoidal according to \cite[Example 3.2.4.4, Proposition 3.2.4.10]{HA}.
Note that the symmetric monoidal structure on $\CAlg(-)$ is given by the coproduct.
The monoidal product in $\infSH(X)$ (resp.\ $\Mod_R$) is denoted by $\wedge$ (resp.\ $\wedge_R$).
Then we have the induced monoidal products on $\CAlg(\infSH(X))$ and $\CAlg(\Mod_R)$.
There is an equivalence between $\CAlg(\Mod_R)$ and the $\infty$-category of $R$-algebras $\CAlg(\infSH(X))^{/R}$ by \cite[Corollary 3.4.1.7]{HA}.

The coproduct in $\NAlg(\infSH(X))$ is also denoted by $\wedge$.
Since the forgetful functor $\NAlg(\infSH(X))\to \CAlg(\infSH(X))$ preserves colimits, the notation $\wedge$ on $\CAlg(\infSH(X))$ and $\NAlg(\infSH(X))$ is compatible.

\begin{prop}
\label{norm.9}
Let
\[
\begin{tikzcd}
Y'\ar[d,"f'"']\ar[r,"g'"]&
Y\ar[d,"f"]
\\
X'\ar[r,"g"]&
X
\end{tikzcd}
\]
be a cartesian square in $\FinGpd$ such that $f$ is a finite covering.
For $\infSH$, the natural transformation
\[
f'_\sharp g'^*
\to
g^*f_\sharp
\]
given by the composite
\[
f'_\sharp g'^*
\xrightarrow{ad}
f'_\sharp g'^*f^*f_\sharp
\xrightarrow{\simeq}
f'_\sharp f'^*g^*f_\sharp
\xrightarrow{ad'}
g^*f_\sharp
\]
is an equivalence.
\end{prop}
\begin{proof}
As usual, Proposition \ref{norm.1} allows us to reduce to showing that the induced map $f_\sharp' g'^*\Sigma^n\Sigma^\infty W_+ \to g^*f_\sharp \Sigma^n\Sigma^\infty W_+$ is an equivalence for every $W\in \Fin_Y$ and integer $n$.
This follows from the fact that the composite of the induced morphisms
\[
Y' \times_Y W
\to
Y'\times_Y (Y\times_X W)
\xrightarrow{\cong}
Y'\times_{X'}(X'\times_X W)
\to
X'\times_X W
\]
is an isomorphism.
\end{proof}

\begin{prop}
\label{norm.2}
Let $f\colon Y\to X$ be a finite covering.
Then for $\cF\in \infSH(Y)$ and $\cG\in \infSH(X)$, there exists a canonical equivalence
\[
f_\sharp(\cF \wedge f^* \cG)
\simeq
f_\sharp \cF \wedge \cG.
\]
\end{prop}
\begin{proof}
As usual, Proposition \ref{norm.1} allows us to reduce to the case when $\cF=\Sigma^m\Sigma^\infty V_+$ and $\cG=\Sigma^n\Sigma^\infty W_+$ for some $V\in \Fin_Y$, $W\in \Fin_X$, and $m,n\in \Z$.
In this case, the canonical isomorphism
\[
V\times_Y (W\times_X Y)
\cong
V \times_X W
\]
gives the desired equivalence.
\end{proof}

Let $f\colon Y\to X$ be a morphism in $\FinGpd$.
The formulation \eqref{norm.3.1} tells that the functor $f^*$ is symmetric monoidal.
Hence we obtain an induced adjoint pair
\begin{equation}
f^*
:
\CAlg(\infSH(X))
\rightleftarrows
\CAlg(\infSH(Y))
:
f_*
\end{equation}
by \cite[Remark 7.3.2.13]{HA}.
The formulation of these functors provided in \cite[Proposition 7.3.2.5]{HA} shows that the two squares in
\begin{equation}
\label{norm.4.3}
\begin{tikzcd}
\CAlg(\infSH(X))\ar[r,"f^*",shift left=0.5ex]\ar[r,"f_*"',leftarrow,shift right=0.5ex]\ar[d]&
\CAlg(\infSH(Y))\ar[d]
\\
\infSH(X)\ar[r,"f^*",shift left=0.5ex]\ar[r,"f_*"',leftarrow,shift right=0.5ex]&
\infSH(Y)
\end{tikzcd}
\end{equation}
commute, where the vertical functors are the forgetful functors.

If $f$ has connected fibers, then we noted that $f_\otimes$ preserves colimits.
Hence we similarly obtain a functor
\begin{equation}
\label{norm.4.1}
f_\otimes
:
\CAlg(\infSH(Y))
\to
\CAlg(\infSH(X))
\end{equation}
and a commutative square
\begin{equation}
\begin{tikzcd}
\label{norm.4.2}
\CAlg(\infSH(Y))\ar[d]\ar[r,"f_\otimes"]&
\CAlg(\infSH(X))\ar[d]
\\
\infSH(Y)\ar[r,"f_\otimes"]&
\infSH(X).
\end{tikzcd}
\end{equation}

The following should be compared with \eqref{orth.4.1}.

\begin{prop}
\label{norm.7}
Let $f\colon Y\to X$ be a finite covering in $\FinGpd$.
Then there is an induced adjunction
\[
f_\otimes
:
\NAlg(\infSH(Y))
\rightleftarrows
\NAlg(\infSH(X))
:
f^*.
\]
Furthermore, the two squares in
\begin{equation}
\label{norm.7.1}
\begin{tikzcd}
\NAlg(\infSH(Y))\ar[r,"f_\otimes",shift left=0.5ex]\ar[r,"f^*"',leftarrow,shift right=0.5ex]\ar[d]&
\NAlg(\infSH(X))\ar[d]
\\
\infSH(Y)\ar[r,"f_\otimes",shift left=0.5ex]\ar[r,"f^*"',leftarrow,shift right=0.5ex]&
\infSH(X)
\end{tikzcd}
\end{equation}
commute, where the vertical functors are the forgetful functors.
\end{prop}
\begin{proof}
Apply \cite[Theorem 8.5]{BH21} to the case when $\cC:=\FinGpd$, $\cA:=\infSH^{\otimes}$, and $\mathrm{right}$ (resp.\ $\mathrm{left}$) is the class of all morphisms (finite coverings) in $\FinGpd$.
Then $\mathrm{Sect}(\cA_X)$ in the reference is precisely $\NAlg(\infSH(X))$ for $X\in \FinGpd$.
Hence as observed in \cite[Remark 8.6]{BH21}, we have the desired adjunction such that the two squares in \eqref{norm.7.1} commutes.
\end{proof}

\begin{prop}
\label{norm.8}
Let $f\colon Y\to X$ be a finite covering in $\FinGpd$.
Then there is an induced adjunction
\[
f^*
:
\NAlg(\infSH(X))
\rightleftarrows
\NAlg(\infSH(Y))
:
f_*.
\]
Furthermore, the two squares in
\begin{equation}
\label{norm.8.1}
\begin{tikzcd}
\NAlg(\infSH(X))\ar[r,"f^*",shift left=0.5ex]\ar[r,"f_*"',leftarrow,shift right=0.5ex]\ar[d]&
\NAlg(\infSH(Y))\ar[d]
\\
\infSH(X)\ar[r,"f^*",shift left=0.5ex]\ar[r,"f_*"',leftarrow,shift right=0.5ex]&
\infSH(Y)
\end{tikzcd}
\end{equation}
commute, where the vertical functors are the forgetful functors.
\end{prop}
\begin{proof}
Let us imitate the proof of \cite[Theorem 8.2]{BH21}.
By \cite[Corollary C.21(2)]{BH21}, there is an induced adjunction
\[
f^* : \Span(\Fin_X) \rightleftarrows \Span(\Fin_Y):f_\sharp.
\]
Let $\infSH^{\otimes}\vert \Span(\Fin_X)$ be the restriction of $\infSH^{\otimes}$ to $\Span(\Fin_X)$.
For every $V\in \Span(\Fin_X)$, let $f_V\colon V\times_X Y\to V$ be the projection.
The functor $f_V^*\colon \infSH(V)\to \infSH(V\times_X Y)$ admits the right adjoint $f_{V*}$.
Apply \cite[Proposition 8.16]{BH21} to the cocartesian fibration $\infSH^{\otimes}\vert \Span(\Fin_X)$ to obtain the desired adjunction.

Let us review the descriptions of $f^*$ and $f_*$ for $\NAlg$ in this reference.
The functor $f^*$ for $\NAlg$ used here is the same as the functor $f^*$ for $\NAlg$ in Proposition \ref{norm.7}.
Suppose $B\in \NAlg(\infSH(Y))$.
For $V\in \Span(\Fin_X)$, the section $(f_*B)(V)$ is given by $f_{V*}(B(V\times_X Y))$.
In particular, the section $(f_*B)(X)$ is given by $f_*(B(X))$.
Hence the two squares in \eqref{norm.8.1} commute.
\end{proof}

For abbreviation, we set
\[
\Sp_G
:=
\infSH(\rB G),
\;
\CAlg_G
:=
\CAlg(\infSH(\rB G)),
\text{ and }
\NAlg_G
:=
\NAlg(\infSH(\rB G)).
\]
This notation is further justified by Remark \ref{orth.3}.

\subsection{Equivariant orthogonal spectra}
\label{orth.sec}

The purpose of this section is to review equivariant homotopy theory using model categories.
Our references for that 
are \cite{GM95}, \cite{MM02}, \cite{HHR}, and \cite{Schwede}.
We will also review the comparison between $\infty$-categorical and model categorical constructions of equivariant spectra. Consequently, we may apply certain known constructions and results for equivariant orthogonal spectra to $\infty$-categories as discussed in the previous subsection.

Let $\rB G$ denote the associated finite groupoid.
In this subsection, we are interested in the obvious morphisms
\[
\rB H\xrightarrow{i} \rB G\xrightarrow{p} *,
\]
where $H\to G$ is an inclusion.

\begin{df}
\label{orth.1}
Let $\SpO$ denote the category of orthogonal spectra.
For a finite group $G$, let $\SpO_G$ denote the category of orthogonal $G$-spectra.
Recall that an \emph{orthogonal $G$-spectrum} is an orthogonal spectrum with a $G$-action.
A \emph{morphism of orthogonal $G$-spectra} is a morphism of underlying orthogonal spectra that is compatible with the $G$-actions.

The definition of orthogonal $G$-spectra in \cite{HHR} is \emph{different} from the above one, but the two categories are equivalent.
See \cite[Remark 2.7]{Schwede} for the details.
\end{df}

\begin{df}
\label{orth.2}
The category $\SpO_G$ admits a symmetric monoidal model structure, see \cite[Propositions B.63, B.76]{HHR}.
We denote the (model) category of commutative monoids in $\SpO$ by $\CAlg^O$. According to \cite{HHR}, we sometimes denote it also by $\Comm$, and the (model) category of commutative monoids in $\SpO_G$ by $\mathbf{\CommG}$.
We refer to \cite[section A.1.2]{HHR} for the details.
According to \cite[Proposition B.129]{HHR}, $\CommG$ has a nice model structure.
A morphism in $\CommG$ is a weak equivalence (resp.\ fibration) precisely when its underlying morphism in $\SpO_G$ is a weak equivalence (resp.\ fibration).
The (underived) coproduct in $\CommG$ is denoted by $\wedge$.
\end{df}

\begin{rmk}
\label{orth.3}
As observed in the preceding paragraphs of \cite[Lemma 9.6]{BH21}, $\Sp_G$ is equivalent to the underlying $\infty$-category of the model category of symmetric $G$-spectra.
This is equivalent to the underlying $\infty$-category of $\SpO_G$ by \cite{mandell04}.
See also \cite[Remark 9.12]{BH21} for another $\infty$-description. 
Furthermore, as observed in \cite[after Definition 9.14]{BH21}, $\NAlg_G$ is equivalent to the underlying $\infty$-category of the model category of $G$-$\E_\infty$-rings,
which is equivalent to the underlying $\infty$-category of $\Comm_G$. We refer to \cite{GW} for a comparison of different models, rectification results and further references. 
\end{rmk}

\begin{const}
\label{etale.8}
Let $H$ be a subgroup of $G$, with the inclusion map $H\to G$.
(For $H=\pt$, compare Example \ref{setadj}.)
Let us review several functors from \cite[sections 2.2.3, 2.5.1]{HHR}.
The \emph{norm functor}
\[
N^G_H\colon \SpO_H\to \SpO_G
\]
sends $Y\in \SpO_H$ to $\bigwedge_{i\in G/H} Y$ with a suitable $G$-action.
If $H=\pt$, we often simply write $N^G$. 

The \emph{restriction functor}
\[
i^* \colon \SpO_G\to \SpO_H
\]
sends $X\in \SpO_G$ to $X$, and the action is the restriction of the $G$-action to $H$.
Its left adjoint $i_\sharp$ and right adjoint $i_*$ send $X\in \SpO_H$ to
\begin{equation}
\label{etale.8.1}
\bigvee_{i\in G/H} X_i
\text{ and }
\prod_{i\in G/H} X_i
\end{equation}
respectively, 
where $X_i:=H_i\wedge_H X$ and $H_i\subset G$ is the coset indexed by $G$.
According to \cite[Proposition B.72]{HHR}, $i^*$ is a left and right Quillen functor.
Hence we have Quillen adjunctions
\[
i_\sharp : \SpO_H \rightleftarrows \SpO_G : i^*
\text{ and }
i^* : \SpO_G \rightleftarrows \SpO_H : i_*.
\]

We also have the functor
\[
\iota\colon \SpO\to \SpO_G
\]
imposing the trivial $G$-action.
The \emph{fixed point functor}
\[
(-)^G\colon \SpO_G\to \SpO
\]
sends an orthogonal $G$-spectrum $(X_0,X_1,\ldots)$ to $(X_0^G,X_1^G,\ldots)$.
There is a Quillen adjunction
\[
\iota :\SpO\rightleftarrows \SpO_G:(-)^G.
\]
\end{const}

\begin{const}
\label{desc.1}
For the definition of the \emph{geometric fixed point functor}
\[
\Phi^G
\colon
\SpO_G\to \SpO,
\]
we refer to \cite[section B.10.1]{HHR}.
There is yet another functor
\[
\Phi_M^{G}
\colon
\SpO_G\to \SpO,
\]
in \cite[Definition B.190]{HHR}, which is called the \emph{monoidal geometric fixed point functor}.
This is lax monoidal and preserves cofibrations and acyclic cofibrations, see \cite[sections B.10.3, B.10.4]{HHR}.

According to \cite[Proposition B.201]{HHR}, there is a zig-zag of weak equivalences between $\Phi^G(X)$ and $\Phi_M^G(X)$ whenever $X\in \SpO_G$ is cofibrant.
\end{const}

\begin{rmk}
\label{orth.4}
The functor $i^*\colon \SpO_H \to \SpO_G$ is a model for the functor of $\infty$-categories $i^*\colon \Sp_H \to \Sp_G$ since their values on $\Sigma^n\Sigma^\infty X_+$ are equivalent for all $X\in \Fin_H$ and integers $n$.
Likewise, the functor $\iota \colon \SpO\to \SpO_G$ is a model for the functor $p^*\colon \Sp\to \Sp_G$, where $p\colon \rB G \to \pt$.
It follows by the uniqueness of $\infty$-adjoints that the fixed point functor $(-)^G\colon \SpO_G \to \SpO$ is a model for the functor $p_*\colon \SpO_G\to \SpO$.
We have similar comparison results for $i_\sharp$ and $i_*$.

According to \cite[Remark 9.10]{BH21}, the norm functor $N^G_H\colon \SpO_H\to \SpO_G$ is a model for the functor $i_{\otimes}\colon \Sp_H \to \Sp_G$, and the geometric fixed point functor $\Phi^G\colon \SpO_G\to \SpO$ is a model for the functor $p_{\otimes}\colon \Sp_G\to \Sp$.
It follows that $\Phi_M^G$ is a model for $p_\otimes$ too.

The functors $N_H^G\colon \SpO_H \to \SpO_G$ and $i^*\colon \SpO_H\to \SpO_G$ are symmetric monoidal.
By \cite[Proposition 2.27]{HHR}, they induce a Quillen adjunction
\begin{equation}
\label{orth.4.1}
N_H^G: \Comm_H \rightleftarrows \CommG : i^*.
\end{equation}
In short, comparing the adjunctions between $\Comm_G$  with those for the underlying spectra, $N_H^G$ already exists for spectra, but only becomes a left adjoint to $i^*$ in $\Comm_G$, replacing $i_\sharp$.

Compare the diagram in \cite[Proposition A.56]{HHR} with \eqref{norm.4.2} and use the conservativity of the forgetful functor $\NAlg_G\to \Sp_G$ to show that $N_H^G\colon \Comm_H \to \Comm_G$ is a model for the functor $i_\otimes \colon \NAlg_H\to \NAlg_G$.
Then $i^*\colon \Comm_G\to \Comm_H$ is a model for the functor $i^*\colon \NAlg_G \to \NAlg_H$ by adjunction.
\end{rmk}

\begin{prop}
\label{orth.5}
We have the following equivalences of functors between $\infty$-categories $\Sp_G$ for appropriate $G$:
\begin{enumerate}
\item[\textup{(1)}] $p_\otimes i_\otimes \simeq \id$ if $H=e$,
\item[\textup{(2)}] $i^*p^*\simeq \id$ if $H=e$,
\item[\textup{(3)}] $i_\sharp\xrightarrow{\simeq} i_*$,
\item[\textup{(4)}] $\id \xrightarrow{\simeq} p_\otimes p^*$,
\item[\textup{(5)}] $p_\otimes i_\sharp \simeq 0$ if $G\neq e$,
\item[\textup{(6)}] $i^*i_\otimes\simeq (-)^{\wedge [G:H]}$,
\item[\textup{(7)}] $i^*i_\sharp\simeq (-)^{\oplus [G:H]}$.
\end{enumerate}
\end{prop}
\begin{proof}
The first two follow from $pi=\id$.
The next three follow from \cite[Propositions B.56, B.182, B.192]{HHR}.
For (6), consider the cartesian square
\[
\begin{tikzcd}
G/H \times \rB H\ar[d,"q"']\ar[r,"q"]&
\rB H\ar[d,"i"]
\\
\rB H\ar[r,"i"]&
\rB G,
\end{tikzcd}
\]
where $q$ is the $\lvert G/H \rvert$-fold map.
If $\bar{i}$ and $\bar{q}$ denote the forward morphisms in $\Span(\FinGpd)$ associated with $i$ and $q$, then $\bar{i}i\simeq q\bar{q}$.
Hence we have an equivalence $i^*i_\otimes \simeq q_\otimes q^*$.
Together with Proposition \ref{norm.6}, we obtain the desired equivalence.
For (7), Proposition \ref{norm.9} gives an equivalence $i^*i_\sharp\simeq q_\sharp q^*$.
The functor $q^*$ can be identified with the diagonal functor $\Sp_H \to (\Sp_H)^{\times G/H}$, and the functor $q_\sharp$ can be identified with the $[G:H]$-fold direct sum $(\Sp_H)^{\times G/H}\to \Sp_H$.
Use these facts to conclude.
\end{proof}

\begin{rmk}
Recall that the forgetful functors $\CAlg_G\to \Sp_G$ and $\NAlg_G\to \Sp_G$ are conservative.
Together with \eqref{norm.4.3} and \eqref{norm.4.2}, we see that Proposition \ref{orth.5}(2),(4) holds for $\CAlg_G$ for appropriate $G$.
Similarly, together with \eqref{norm.7.1}, we see that Proposition \ref{orth.5}(2),(6) holds for $\NAlg_G$ for appropriate $G$.
\end{rmk}

\begin{lem}
\label{etale.4}
Suppose $R\in \CAlg$, $M\in \Mod_R$, and $L\in \Mod_{p^* R}$.
If $M$ and $R$ are connective, then there exists a canonical equivalence of $R$-modules
\[
M\wedge_R p_*L \simeq p_*(p^* M\wedge_{p^* R}L).
\]
\end{lem}
\begin{proof}
We have maps
\[
p^*(M\wedge_R p_*L)\xrightarrow{\simeq} p^* M\wedge_{p^* R} p^* p_*L \to p^* M\wedge_{p^* R}L,
\]
where the second map is induced by the counit map $p^*p_*L\to L$.
By adjunction, we obtain $M\wedge_R p_*L\to  p_*(p^* M\wedge_{p^*R}L)$.
We only need to show that this is an equivalence in $\Sp$ after forgetting the module structures.

Let $\cF$ be the class of $R$-modules $M$ such that this map is an equivalence.
The functors $p^*$, $\wedge_R p_*L$, and $\wedge_{p^*R} L$ preserve colimits.
As explained after \cite[Remark 6.8]{MNN}, the functor $p_*$ preserves colimits too.
It follows that $\cF$ is closed under colimits.
Furthermore, $\cF$ is closed under shifts.
Since $\cF$ contains $R$, $\cF$ contains all connective $R$-modules by \cite[Proposition 7.1.1.13]{HA}.
\end{proof}

\subsection{Mackey functors}

\begin{df}
Recall from \cite[Definition 8.2.5]{HHR21} that a \emph{Mackey functor for $G$} (or simply \emph{Mackey functor}) is a presheaf $M$ on $\Span(\Fin_{\rB G})$ of abelian groups that transforms finite coproducts into finite products. (This is easily seen to be equivalent to more classical definitions as e.g.\ recalled in \cite[Definition 3.1]{HHR}.)
For a forward (resp.\ backward) morphism $f$ in $\Span(\Fin_{\rB G})$, $M(f)$ is called a \emph{restriction map} (resp.\ \emph{transfer map}).
Let $\Mack_G$ denote the category of Mackey functors for $G$.
We include the explicit description of Mackey functors for $\Ztwo$ in Example \ref{etale.29}.
\end{df}

For all $X\in \Sp_G$, $M\in \Fin_{\rB G}$, and integers $n$, we set
\begin{equation}
\label{Mack.2.2}
\ul{\pi}_n(X)(M)
:=
\Hom_{\Ho(\Sp_G)}(\Sigma^n \Sigma^\infty M_+,X)
\cong
\Hom_{\Ho(\Sp_G)}(\Sigma^n,\Sigma^\infty M_+\wedge X),
\end{equation}
where the isomorphism comes from \cite[Example 2.6]{HHR}.
The first (resp.\ second) formulation is contravariant (resp.\ covariant) in $M$, and these two can be combined to produce the \emph{equivariant homotopy group functor}
\begin{equation}
\label{Mack.2.3}
\ul{\pi}_n
\colon
\Sp_{G}
\to
\Mack_{G}.
\end{equation}
We refer to \cite[section 3.1]{HHR} for the details.
For $X\in \Sp_G$ and an integer $n$, we say that $X$ is \emph{$n$-connected} if $\ul{\pi}_k(X)=0$ for all integers $k\leq n$.

\begin{prop}
\label{Mack.2}
For every integer $n$, $\ul{\pi}_n$ preserves products and filtered colimits.
\end{prop}
\begin{proof}
The claim for products follows from the first formulation in \eqref{Mack.2.2}.
By Proposition \ref{norm.1},
$\Sigma^n \Sigma^\infty M_+$ is compact, i.e., $\Map_{\Sp_G}(\Sigma^n\Sigma^\infty M_+,-)$ preserves filtered colimits.
This immediately implies the claim.
\end{proof}

Now suppose $M\in \Mack_{G}$.
According to \cite[Theorem 5.3]{GM95}, one can associate an \emph{equivariant Eilenberg-MacLane spectrum} $\EM M\in \SpO_{G}$, which satisfies
\[
\ul{\pi}_n(\EM M)
\cong
\left\{
\begin{array}{ll}
M & \text{if }n=0,
\\
0 & \text{if }n\neq 0.
\end{array}
\right.
\]
Furthermore, $\EM M$ is unique up to an isomorphism in the homotopy category $\Ho(\SpO_{G})\simeq \Ho(\Sp_G)$, and there is a canonical isomorphism
\begin{equation}
\label{Mack.3.1}
\Hom_{\Ho(\SpO_{G})}(\EM M,\EM L)
\cong
\Hom_{\Mack_{G}}(M,L)
\end{equation}
for all $M,L\in \Mack_{G}$.
It follows that we have a functor
\begin{equation}
\label{Mack.2.1}
\EM
\colon
\Mack_G
\to
\Ho(\Sp_G).
\end{equation}

\begin{df}
\label{t.1}
Let $n$ be an integer.
Let $(\Sp_G)_{\geq n}$ (resp.\ $(\Sp_G)_{\leq n}$) denote the full subcategory of $\Sp_G$ spanned by $X\in \Sp_G$ such that $\ul{\pi}_k(X)=0$ for all integers $k<n$ (resp.\ $k> n$).
Observe that there are equivalences
\begin{equation}
(\Sp_G)_{\geq n}
\simeq
\Sigma^n(\Sp_G)_{\geq 0}
\text{ and }
(\Sp_G)_{\leq n}
\simeq
\Sigma^n(\Sp_G)_{\leq 0}.
\end{equation}
\end{df}

Suppose $X,Y\in \Sp_G$.
By \cite[Proposition 4.11]{HHR}, $X\in (\Sp_G)_{\geq 0}$ (resp.\ $Y\in (\Sp_G)_{\leq -1}$) if and only if $X$ is slice $(-1)$-positive (resp.\ $Y$ is slice $0$-null) in the sense of \cite[Definition 4.8]{HHR}.
This immediately implies the vanishing
\begin{equation}
\Map_{\Sp_G}(X,Y)\simeq 0
\end{equation}
for $X\in  (\Sp_G)_{\geq 0}$ and $Y\in (\Sp_G)_{\leq -1}$.
According to \cite[Remark 4.12]{HHR}, there is an example of $X\in \Sp_G$ such that $X$ is slice $0$-positive but $X\notin (\Sp_G)_{\geq 1}$.

Suppose $X\in \Sp_G$.
As explained in \cite[section 4.2]{HHR}, there exists a cofiber sequence in $\Sp_G$
\[
X'\to X\to X''
\]
such that $X'$ is slice $(-1)$-positive and $X''$ is slice $0$-null, i.e., $X'\in (\Sp_G)_{\geq 0}$ and $X''\in (\Sp_G)_{\leq -1}$.

We combine what we have discussed above and recall the notion of $t$-structures in $\infty$-categories from \cite[Definitions 1.2.1.1, 1.2.1.4]{HA} to deduce the following result, which is probably known to the experts: 

\begin{prop}
\label{t.2}
The pair of $(\Sp_G)_{\geq 0}$ and $(\Sp_G)_{\leq 0}$ forms a $t$-structure on $\Sp_G$.
\end{prop}

This $t$-structure is called the \emph{equivariant homotopy $t$-structure on $\Sp_G$}.
By \cite[Remark 1.2.1.12]{HA}, the heart $\Sp_G^{\heartsuit}:=(\Sp_G)_{\geq 0}\cap (\Sp_G)_{\leq 0}$ is the nerve of an abelian category, and there is an equivalence
\begin{equation}
\label{t.3.1}
\Sp_G^{\heartsuit}\simeq \Nerve(\Ho(\Sp_G^{\heartsuit})).
\end{equation}
For every integer $n$, let $\tau_{\geq n}$, $\tau_{\leq n}$, and $h_n$ denote the truncation and homology functors.

\begin{prop}
\label{t.3}
The functor of $\infty$-categories
\begin{equation}
\label{t.3.2}
\Sp_G^{\heartsuit}
\to
\Nerve(\Mack_G)
\end{equation}
sending $X\in \Sp_G^\heartsuit$ to $\ul{\pi}_0(X)$ is an equivalence.
\end{prop}
\begin{proof}
The functor \eqref{Mack.2.1} gives an equivalence $\Mack_G\simeq \Ho(\Sp_G^\heartsuit)$.
Combine with \eqref{t.3.1} to obtain a quasi-inverse of \eqref{t.3.2}.
\end{proof}

Compose a quasi-inverse of \eqref{t.3.2} with the inclusion $\Sp_G^\heartsuit \to \Sp_G$ to obtain the functor of $\infty$-categories
\begin{equation}
\label{t.3.3}
\EM \colon \Nerve(\Mack_G) \to \Sp_G,
\end{equation}
which is an upgrade of \eqref{Mack.2.1}.

\begin{prop}
\label{t.4}
The functor of $\infty$-categories $\EM \colon \Nerve(\Mack_G) \to \Sp_G$ preserves products and filtered colimits.
\end{prop}
\begin{proof}
Owing to Proposition \ref{t.3}, it remains to check that the inclusion functor $\Sp_G^\heartsuit\to \Sp_G$ preserves products and filtered colimits.
This follows from Proposition \ref{Mack.2}.
\end{proof}

\begin{prop}
\label{t.5}
Let $H$ be a subgroup of $G$.
Then the norm functor $i_{H \otimes}\colon \Sp_H \to \Sp_G$ sends $(\Sp_H)_{\geq 0}$ into $(\Sp_G)_{\geq 0}$.
\end{prop}
\begin{proof}
We refer to \cite[Proposition 4.33]{HHR}.
\end{proof}

\subsection{Green functors}
\label{seca.4}

Our references for Green functors are \cite{lewis88} and \cite{Shulman}.

\begin{df}
\label{Green.2}
For Mackey functors $M$ and $L$, the \emph{box product of $M$ and $L$} is defined to be
\[
M \Box L
:=
\ul{\pi}_0(\EM M \wedge \EM L).
\]
\end{df}

There is a purely algebraic definition of the box product of Mackey functors, which is rather explicit for $G=\Z/p$ and some prime $p$, see \cite[p.\ 61]{lewis88} and \cite[sections 2.2 and 2.4.3]{Shulman}. This is expected to coincide with the above Definition, but we won't need this.

\begin{df}
\label{Green.3}
A \emph{Green functor} $A$ is a commutative monoid in the category $\Mack_G$, i.e., $A$ is equipped with morphisms $A\Box A\to A$ and $\ul{\pi}_0(\Sphere)\to A$ satisfying the unital, associative, and commutative axioms, where $\Sphere$ denotes equivariant sphere spectrum.
Let $\Green_G$ denote the category of Green functors.

An \emph{$A$-module} $M$ is an object of $\Mack_G$ equipped with an action morphism $A\Box M\to M$ satisfying the module axioms.

For $A$-modules $M$ and $L$, $M\Box_A L$ is defined to be the coequalizer of the two action morphisms
\[
M\Box A\Box L
\rightrightarrows
M \Box L.
\]
Let $\ul{\Tor}_*^A(M,L)$ be the derived functor of $M\Box_A L$.
\end{df}

\begin{prop}
\label{Green.1}
Suppose $A\in \CAlg_G$,
and let $M$ and $L$ be $A$-modules.
Then there exists a convergent spectral sequence
\begin{equation}
\label{Green.1.1}
E_{p,q}^2
:=
\ul{\Tor}_{p}^{\ul{\pi}_*(A)}(\ul{\pi}_*(M),\ul{\pi}_*(L))_q
\Rightarrow
\ul{\pi}_{p+q}(M\wedge_A L).
\end{equation}
\end{prop}
\begin{proof}
We refer to \cite[section 6]{BGS20}.
With the stronger assumption $A\in \NAlg_G$, this result is due to Lewis and Mandell \cite[Theorem 6.6]{LM06}.
\end{proof}

\begin{prop}
\label{Green.4}
Suppose $A\in \CAlg_G$, and let $M$ and $L$ be $A$-modules.
If $A$, $M$, and $L$ are $(-1)$-connected, then $M\wedge_A L$ is $(-1)$-connected too, and there is an isomorphism
\[
\ul{\pi}_0(M)
\Box_{\ul{\pi}_0(A)}
\ul{\pi}_0(L)
\simeq
\ul{\pi}_0(M\wedge_A L).
\]
\end{prop}
\begin{proof}
We refer to \cite[Corollary 6.8.1]{BGS20}.
\end{proof}

Apply Proposition \ref{Green.4} to the case when $A$ is the equivariant sphere spectrum to obtain the symmetric monoidal structure on $(\Sp_G)_{\geq 0}$ that is the restriction of the symmetric monoidal structure on $\Sp_G$.
Furthermore, the functor
\[
\ul{\pi}_0
\colon
(\Sp_G)_{\geq 0}
\to
\Mack_G
\]
is symmetric monoidal.
Its right adjoint is the functor $\EM \colon \Mack_G\to (\Sp_G)_{\geq 0}$ by Proposition \ref{t.3}.
Together with \cite[Remark 7.3.2.13]{HA}, the induced functors
\begin{equation}
\label{Green.4.1}
\ul{\pi}_0
:
(\CAlg_G)_{\geq 0}
\rightleftarrows
\Green_G
:
\EM
\end{equation}
form an adjoint pair, where $(\CAlg_G)_{\geq 0}:=\CAlg((\Sp_G)_{\geq 0})$.
The formulation of these functors provided in \cite[Proposition 7.3.2.5]{HA} shows that the two squares in the diagram
\begin{equation}
\label{Green.4.2}
\begin{tikzcd}
(\CAlg_G)_{\geq 0}\ar[r,shift left=0.5ex,"\ul{\pi}_0"]\ar[d,"U"']\ar[r,shift right=0.5ex,"H"',leftarrow]&
\Green_G\ar[d,"U"]
\\
(\Sp_G)_{\geq 0}\ar[r,shift left=0.5ex,"\ul{\pi}_0"]\ar[r,shift right=0.5ex,"H"',leftarrow]&
\Mack_G
\end{tikzcd}
\end{equation}
commute, where the vertical functors are the forgetful functors.

Suppose $A\in (\CAlg_G)_{\geq 0}$.
Let $(\Mod_A)_{\geq 0}$ denote the $\infty$-category of $A$-modules in $(\Sp_G)_{\geq 0}$.
By \cite[Remark 3.8]{PSW22}, we also have adjoint functors
\begin{equation}
\label{Green.4.3}
\ul{\pi}_0
:
(\Mod_A)_{\geq 0}
\rightleftarrows
\Mod_{\ul{\pi}_0(A)}
:
\EM
\end{equation}
such that the two squares in the diagram
\begin{equation}
\label{Green.4.4}
\begin{tikzcd}
(\Mod_A)_{\geq 0}\ar[r,shift left=0.5ex,"\ul{\pi}_0"]\ar[d,"U"']\ar[r,shift right=0.5ex,"H"',leftarrow]&
\Mod_{\ul{\pi}_0(A)}\ar[d,"U"]
\\
(\Sp_G)_{\geq 0}\ar[r,shift left=0.5ex,"\ul{\pi}_0"]\ar[r,shift right=0.5ex,"H"',leftarrow]&
\Mack_G
\end{tikzcd}
\end{equation}
commute, where the vertical functors are the forgetful functors.

\begin{rmk}
\label{Green.5}
The equivariant Eilenberg-MacLane spectrum of a Green functor does \emph{not} produce an object of $\NAlg_G$ in general, see \cite[Theorem 5.3, Proposition 6.1]{Ull}.
We need the stronger notion of Tambara functors to construct an object of $\NAlg_G$ as the equivariant Eilenberg-MacLane spectrum.
We refer to \cite{Ull} for the details.
\end{rmk}

\subsection{Flat modules}
\label{seca.5}

\begin{df}
\label{etale.24}
Let $A$ be a Green functor.
An $A$-module $M$ is called \emph{flat} if the functor $M\Box_A (-)$ from the category of $A$-modules to the category of Mackey functors is exact.
Equivalently, $\ul{\Tor}_s^{A}(-,M)=0$ for every integer $s\geq 1$.
This definition was considered in \cite[section 4]{LM06}.
\end{df}

If $A\to B$ is a morphism of Green functors and $M$ is a flat $A$-module, then $M\Box_A B$ is a flat $B$-module.

\begin{df}
\label{flat.5}
Recall from \cite[section 2]{LM06} that the \emph{Burnside category} $\mathfrak{B}_G$ is defined to be the additive category whose objects are the finite $G$-sets and whose hom groups are given by
\[
\Hom_{\mathfrak{B}_G}(X,Y)
:=
\Hom_{\Ho(\Sp_G)}(\Sigma^\infty X_+,\Sigma^\infty Y_+)
\]
for all finite $G$-sets $X$ and $Y$.

For a finite $G$-set $X$, let $B^X$ denote the Mackey functor $\Hom_{\mathfrak{B}_G}(-,X)$.
As explained in \cite[p.\ 519]{LM06}, there is an isomorphism
\begin{equation}
\label{flat.5.1}
M\Box B^X(Y)
\cong
M(Y\times X)
\end{equation}
for all finite $G$-sets $X$ and $Y$.
\end{df}

\begin{prop}
\label{flat.6}
Let $A$ be a Green functor.
Then an $A$-module $M$ is projective if and only if $M$ is a direct summand of a direct sum of $A$-modules of the form $A\Box B^{G/H}$, where $H$ is a subgroup of $G$.
Furthermore, every projective $A$-module is flat.
\end{prop}
\begin{proof}
These are non-graded versions of \cite[Proposition 4.4, Theorem 4.5(c)]{LM06}.
See also \cite[Corollary 1.5]{Greenlees92} for the first claim.
\end{proof}

\begin{df}
\label{etale.20}
Suppose $A\in \CAlg_G$.
An $A$-module $M$ is called \emph{flat} if the following two conditions are satisfied:
\begin{enumerate}
\item[(i)] $\ul{\pi}_0(M)$ is a flat $\ul{\pi}_0(A)$-module,
\item[(ii)] the induced map
\[
\ul{\pi}_n(A)\Box_{\ul{\pi}_0(A)}\ul{\pi}_0(M)
\to
\ul{\pi}_n(M)
\]
is an isomorphism for every integer $n$.
\end{enumerate}
\end{df}

\begin{prop}
\label{etale.21}
Let $A\to B$ be a map in $\CAlg_G$, and let $M$ be a flat $A$-module.
Then $B\wedge_A M$ is a flat $B$-module. Consequently, if we have maps $A \to B$, $A \to C$ and $B \to L$ in $\CAlg_G$ such that $L$ is a flat $B$-module, then the induced map $B \wedge_A C \to L \wedge_A C$ is flat.
\end{prop}
\begin{proof}
The conditions (i) and (ii) in Definition \ref{etale.20} for $M$ imply that $\ul{\pi}_*(M)$ is a flat $\ul{\pi}_*(A)$-module.
This means $\ul{\Tor}_{p}^{\ul{\pi}_*(A)}(-,\ul{\pi}_*(M))_q=0$ for all integers $p\geq 1$ and $q$.
Together with the convergent spectral sequence
\[
E_{p,q}^2:=
\ul{\Tor}_{p}^{\ul{\pi}_*(A)}(\ul{\pi}_*(B),\ul{\pi}_*(M))_q
\Rightarrow
\ul{\pi}_{p+q}(B\wedge_A M)
\]
obtained from Proposition \ref{Green.1}, we obtain isomorphisms of Mackey functors
\begin{align*}
\ul{\pi}_n(B)\Box_{\ul{\pi}_0(A)}\ul{\pi}_0(M)
\cong &
(\ul{\pi}_*(B)\Box_{\ul{\pi}_0(A)}\ul{\pi}_0(M))_n
\\
\cong &
(\ul{\pi}_*(B)\Box_{\ul{\pi}_*(A)}\ul{\pi}_*(M))_n
\cong
\ul{\pi}_n(B\wedge_A M)
\end{align*}
for all integers $n$.
This implies the conditions (i) and (ii) in Definition \ref{etale.20}
for $B\wedge_A M$. The second statement follows from the first applied to $B \to B \wedge_A C$.
\end{proof}

\begin{prop}
\label{etale.22}
Let $f\colon A\to B$ be a flat map in $\CAlg_G$.
If the induced morphism $\ul{\pi}_0(A)\to \ul{\pi}_0(B)$ is an isomorphism, then $f$ is an equivalence.
\end{prop}
\begin{proof}
Immediate from the condition (ii) in Definition \ref{etale.20}.
\end{proof}

\begin{prop}
\label{Mack.1}
Let $A$ be a Green functor, and let $M$ and $L$ be $A$-modules.
If $M$ is flat, then there is an equivalence
\[
\EM(M\Box_A L)
\simeq
\EM M\wedge_{\EM A} \EM L
\]
\end{prop}
\begin{proof}
From the convergent spectral sequence
\[
E_{p,q}^2:=
\ul{\Tor}_{p}^{\ul{\pi}_*(\EM A)}(\ul{\pi}_*(\EM M),\ul{\pi}_*(\EM L))_q
\Rightarrow
\ul{\pi}_{p+q}(\EM M\wedge_{\EM A} \EM L)
\]
obtained from Proposition \ref{Green.1}, we have \[
M\Box_A L\xrightarrow{\simeq} \ul{\pi}_0(\EM M \wedge_{\EM A}\EM L)
\text{ and }
\ul{\pi}_k(\EM M\wedge_{\EM A}\EM L)\cong 0
\]
for every nonzero integer $k$.
\end{proof}

\begin{prop}
\label{flat.3}
Let $A$ be a Green functor.
If $\colim_{i\in I}M_i$ is a filtered colimit of $A$-modules and $L$ is an $A$-module, then there is a canonical equivalence
\[
\colim_{i\in I}(M_i \Box_A L)
\cong
(\colim_{i\in I} M_i) \Box_A L.
\]
\end{prop}
\begin{proof}
Since $\wedge$ commutes with colimits in each variable, we have a canonical equivalence
\[
\colim_{i\in I} (\EM M_i \wedge_{\EM A} \EM L)
\simeq
(\colim_{i\in I} \EM M_i) \wedge_{\EM A} \EM L.
\]
Apply $\ul{\pi}_0$ to this, and use Propositions \ref{Mack.2}, \ref{t.4}, and  \ref{Green.4}  to obtain the desired equivalence.
\end{proof}

\begin{prop}
\label{flat.1}
Let $A$ be a Green functor, and let
\[
\{0\to M_i'\to M_i\to M_i''\to 0\}_{i\in I}
\]
be a system of exact sequence of $A$-modules over a filtered category $I$.
Then the induced sequence
\[
0\to \colim_{i\in I} M_i'\to \colim_{i\in I} M_i \to \colim_{i\in I} M_i'' \to 0
\]
is exact.
\end{prop}
\begin{proof}
We have a system of cofiber sequences in $\Sp_G$
\[
\{\EM M_i' \to \EM M_i \to \EM M_i''\}_{i\in I}.
\]
Take colimits and use Proposition \ref{t.4} to obtain a cofiber sequence
\[
\EM \colim_{i\in I} M_i'\to \EM \colim_{i\in I} M_i \to \EM \colim_{i\in I} M_i''.
\]
Together with the fact that cofiber sequences and exact sequences coincide in the heart of a $t$-structure, we deduce the claim.
\end{proof}

\begin{prop}
\label{flat.2}
Let $A$ be a Green functor.
Then every filtered colimit $\colim_{i\in I} M_i$ of flat $A$-modules is flat.
\end{prop}
\begin{proof}
Let $0\to L'\to L\to L''\to 0$ be an exact sequence of $A$-modules.
Then the induced sequence $0\to M_i\Box_A L'\to M_i\Box_A L\to M_i\Box_A L''\to 0$ is exact, so the induced sequence
\[
0 \to \colim_{i\in I}  (M_i\Box_A L') \to \colim_{i\in I}  (M_i\Box_A L) \to \colim_{i\in I}  (M_i\Box_A L'') \to 0
\]
is exact too by Proposition \ref{flat.1}.
Combine with Proposition \ref{flat.3} to conclude.
\end{proof}

\begin{prop}
\label{flat.4}
Suppose $A\in \CAlg_G$.
Then every filtered colimit of flat $A$-modules is flat.
In particular, every filtered colimit of free $A$-modules is flat.
\end{prop}
\begin{proof}
Combine Propositions \ref{Mack.2} and \ref{flat.2} to show the first claim.
For the second claim, use Proposition \ref{Mack.2} to show that every free $A$-module is flat.
\end{proof}

\section*{Addendum: Real topological Hochschild homology of schemes}

For commutative rings $A$ in which 2 is invertible, as well as for $A=\Z$,  \cite[Proposition 2.3.5]{HP23add} is true. However, for some rings $A$ in which 2 is not invertible the result is not correct as stated. The reason is the description of the ideal $T$ in the proof of loc.\ cit.\ is not correct in this case, and this might imply that the map $\alpha$ in loc.\ cit.\ is not an isomorphism.

In more detail, let $A$ be a commutative ring.
Due to \cite[Corollary 5.2]{DMPR21add},
$\ul{\pi}_0\THR(A)$ is isomorphic to the Mackey functor
\[
\begin{tikzcd}
A\ar[loop below,"\id"]\ar[r,shift left=0.5ex,"\tran"]\ar[r,shift right=0.5ex,leftarrow,"\res"']&
(A\otimes A)/T_A,
\end{tikzcd}
\]
where $\res(x\otimes y)=xy$ for $x,y\in A$,
$\tran(a)=2a\otimes 1$ for $a\in A$,
and $T_A$ is the subgroup generated by $x\otimes a^2 y-a^2 x\otimes y$ and $x\otimes 2ay-2ax\otimes y$ for $a,x,y\in A$.
Let $2A$ denote the ideal $(2)$ in $A$.
We have the monomorphism 
$2A\to (A\otimes A)/T_A$ given by $2a\mapsto 2a\otimes 1$ for $2a\in 2A$.
Its cokernel is isomorphic to $(A/2\otimes A/2)/T_{A/2}$,
where $A/2:=A/2A$.
Observe that $T_{A/2}$ is the subgroup generated by $x\otimes a^2 y-a^2 x\otimes y$ for $a,x,y\in A/2$.
Hence we have the short exact sequence
\[
0
\to
2A
\to
(A\otimes A)/T_A
\to
A/2\otimes_{\varphi,A/2,\varphi} A/2
\to
0,
\]
where $\varphi\colon A/2\to A/2$ denotes the Frobenius (i.e.\ the squaring map).
In particular, \cite[Proposition 2.3.5]{HP23add} holds if and only if $\varphi$ is surjective.

The only statement in \cite{HP23add} where this Proposition is used is the following one in the proof of \cite[Proposition 3.2.2]{HP23add}, for which we now provide an alternative proof.

\begin{prop1}
Let $A\to B$ be an \'etale homomorphism of commutative rings.
Then the induced morphism of Mackey functors
\[
\ul{\pi}_0\THR(\iota A)\square_{\iota A}\iota B
\to
\ul{\pi}_0\THR(\iota B)
\]
is an isomorphism.
\end{prop1}
\begin{proof}
By \cite[Lemma 5.1]{HPadd},
the Mackey functor $\ul{\pi}_0\THR(\iota A)\square_{\iota A}\iota B$ is isomorphic to
\[
\begin{tikzcd}
B\ar[loop below,"w"]\ar[r,shift left=0.5ex,"\tran"]\ar[r,shift right=0.5ex,leftarrow,"\res"']&
(A\otimes A)/T_A\otimes_A B.
\end{tikzcd}
\]
Using the above computations, we have the following commutative diagram where the vertical maps are induced by multiplication:

\[
\begin{tikzcd}[column sep=tiny]
0\ar[r]&
2A\otimes_A B
\ar[d,"\alpha"']\ar[r]&
(A\otimes A)/T_A \otimes_A B
\ar[d,"\beta"]\ar[r]&
A/2\otimes_{\varphi,A/2,\varphi}A/2\otimes_A B\ar[d,"\gamma"]\ar[r]&
0
\\
0\ar[r]&
2B\ar[r]&
(B\otimes B)/T_B\ar[r]&
B/2\otimes_{\varphi,B/2,\varphi}B/2\ar[r]&
0.
\end{tikzcd}
\]
We only need to show that $\beta$ is an isomorphism.
Since $B$ is flat over $A$,
the rows are short exact sequences, and $\alpha$ is an isomorphism.
The induced square of commutative rings
\[
\begin{tikzcd}
A/2\ar[d]\ar[r,"\varphi"]&
A/2\ar[d]
\\
B/2\ar[r,"\varphi"]&
B/2
\end{tikzcd}
\]
is coCartesian by \cite[Tag 0EBS]{stacksadd} since $A/2\to B/2$ is \'etale.
It follows that $\gamma$ is an isomorphism.
Hence $\beta$ is an isomorphism by the five lemma.
\end{proof}


\begin{thebibliography}{10}

\bibitem{AKGH}
{\sc G.~Angelini-Knoll, T.~Gerhardt, and M.~Hill}, {\em Real topological
  {H}ochschild homology via the norm and {R}eal {W}itt vectors}.
\newblock ArXiv Preprint 2111.06970.

\bibitem{SGA4}
{\sc M.~Artin, A.~Grothendieck, and J.~L. Verdier}, {\em Th\'eorie des topos et
  cohomologie \'etale des sch\'emas}, vol.~269, 270, 305 of Lecture Notes in
  Mathematics, Springer-Verlag, 1972--1973.
\newblock S\'eminaire de G\'eom\'etrie Alg\'ebrique du Bois--Marie 1963---64.

\bibitem{Ayo}
{\sc J.~Ayoub}, {\em Les six op\'{e}rations de {G}rothendieck et le formalisme
  des cycles \'{e}vanescents dans le monde motivique}, Ast\'{e}risque, 314, 315
  (2007).

\bibitem{BH21}
{\sc T.~Bachmann and M.~Hoyois}, {\em Norms in motivic homotopy theory},
  Ast\'{e}risque,  (2021), pp.~ix+207.

\bibitem{Bar17}
{\sc C.~Barwick}, {\em Spectral {M}ackey functors and equivariant algebraic
  {$K$}-theory ({I})}, Adv. Math., 304 (2017), pp.~646--727.

\bibitem{BGS20}
{\sc C.~Barwick, S.~Glasman, and J.~Shah}, {\em Spectral {M}ackey functors and
  equivariant algebraic {$K$}-theory, {II}}, Tunis. J. Math., 2 (2020),
  pp.~97--146.

\bibitem{BMS19}
{\sc B.~Bhatt, M.~Morrow, and P.~Scholze}, {\em Topological {H}ochschild
  homology and integral {$p$}-adic {H}odge theory}, Publ. Math. Inst. Hautes
  \'{E}tudes Sci., 129 (2019), pp.~199--310.

\bibitem{BPO2}
{\sc F.~Binda, D.~Park, and P.~A. {\O}stv{\ae}r}, {\em Logarithmic motivic
  homotopy theory}.
\newblock ArXiv Preprint 2303.02729.

\bibitem{BPOCras}
\leavevmode\vrule height 2pt depth -1.6pt width 23pt, {\em Motives and homotopy
  theory in logarithmic geometry}, C. R. Math. Acad. Sci. Paris, 360 (2022),
  pp.~717--727.

\bibitem{BM}
{\sc A.~J. Blumberg and M.~A. Mandell}, {\em Localization theorems in
  topological {H}ochschild homology and topological cyclic homology}, Geom.
  Topol., 16 (2012), pp.~1053--1120.

\bibitem{BHM93}
{\sc M.~B\"{o}kstedt, W.~C. Hsiang, and I.~Madsen}, {\em The cyclotomic trace
  and algebraic {$K$}-theory of spaces}, Invent. Math., 111 (1993),
  pp.~465--539.

\bibitem{Ca}
{\sc D.~Carmody}, {\em Cdh descent for homotopy {H}ermitian {$K$}-theory of
  rings with involution}, Doc. Math., 26 (2021), pp.~1275--1327.

\bibitem{Del09}
{\sc P.~Deligne}, {\em Voevodsky's lectures on motivic cohomology 2000/2001},
  in Algebraic topology, vol.~4 of Abel Symp., Springer, Berlin, 2009,
  pp.~355--409.

\bibitem{SGA3}
{\sc M.~Demazure and A.~Grothendieck}, {\em Sch\'{e}mas en groupes. {I}:
  {P}ropri\'{e}t\'{e}s g\'{e}n\'{e}rales des sch\'{e}mas en groupes}, Lecture
  Notes in Mathematics, Vol. 151, Springer-Verlag, Berlin-New York, 1970.
\newblock S\'{e}minaire de G\'{e}om\'{e}trie Alg\'{e}brique du Bois Marie
  1962/64 (SGA 3).

\bibitem{EGA}
{\sc J.~Dieudonn{\'e} and A.~Grothendieck}, {\em \'{E}l\'ements de
  g\'eom\'etrie alg\'ebrique}, Inst. Hautes \'Etudes Sci. Publ. Math., 4, 8,
  11, 17, 20, 24, 28, 32 (1961--1967).

\bibitem{DMPR21}
{\sc E.~Dotto, K.~Moi, I.~Patchkoria, and S.~P. Reeh}, {\em Real topological
  {H}ochschild homology}, J. Eur. Math. Soc. (JEMS), 23 (2021), pp.~63--152.

\bibitem{DO19}
{\sc E.~Dotto and C.~Ogle}, {\em {$K$}-theory of {H}ermitian {M}ackey functors,
  real traces, and assembly}, Ann. K-Theory, 4 (2019), pp.~243--316.

\bibitem{DHI04}
{\sc D.~Dugger, S.~Hollander, and D.~C. Isaksen}, {\em Hypercovers and
  simplicial presheaves}, Math. Proc. Cambridge Philos. Soc., 136 (2004),
  pp.~9--51.

\bibitem{DGM13}
{\sc B.~I. Dundas, T.~G. Goodwillie, and R.~McCarthy}, {\em The local structure
  of algebraic {K}-theory}, vol.~18 of Algebra and Applications,
  Springer-Verlag London, Ltd., London, 2013.

\bibitem{FL}
{\sc Z.~Fiedorowicz and J.-L. Loday}, {\em Crossed simplicial groups and their
  associated homology}, Trans. Amer. Math. Soc., 326 (1991), pp.~57--87.

\bibitem{GH}
{\sc T.~Geisser and L.~Hesselholt}, {\em Topological cyclic homology of
  schemes}, in Algebraic {$K$}-theory ({S}eattle, {WA}, 1997), vol.~67 of Proc.
  Sympos. Pure Math., Amer. Math. Soc., Providence, RI, 1999, pp.~41--87.

\bibitem{Greenlees92}
{\sc J.~P.~C. Greenlees}, {\em Some remarks on projective {M}ackey functors},
  J. Pure Appl. Algebra, 81 (1992), pp.~17--38.

\bibitem{GM95}
{\sc J.~P.~C. Greenlees and J.~P. May}, {\em Equivariant stable homotopy
  theory}, in Handbook of algebraic topology, North-Holland, Amsterdam, 1995,
  pp.~277--323.

\bibitem{GW}
{\sc J.~J. Guti\'{e}rrez and D.~White}, {\em Encoding equivariant commutativity
  via operads}, Algebr. Geom. Topol., 18 (2018), pp.~2919--2962.

\bibitem{HRW}
{\sc J.~Hahn, A.~Raksit, and D.~Wilson}, {\em A motivic filtration on the
  topological cyclic homology of commutative ring spectra}.
\newblock ArXiv Preprint 2206.11208.

\bibitem{HeKO}
{\sc J.~Heller, A.~Krishna, and P.~A. {\O}stv{\ae}r}, {\em Motivic homotopy
  theory of group scheme actions}, J. Topol., 8 (2015), pp.~1202--1236.

\bibitem{Herrmann}
{\sc P.~Herrmann}, {\em Equivariant {M}otivic {H}omotopy {T}heory}.
\newblock ArXiv Preprint 1312.0241.

\bibitem{Hes96}
{\sc L.~Hesselholt}, {\em On the {$p$}-typical curves in {Q}uillen's
  {$K$}-theory}, Acta Math., 177 (1996), pp.~1--53.

\bibitem{HHR}
{\sc M.~A. Hill, M.~J. Hopkins, and D.~C. Ravenel}, {\em On the nonexistence of
  elements of {K}ervaire invariant one}, Ann. of Math. (2), 184 (2016),
  pp.~1--262.

\bibitem{HHR21}
\leavevmode\vrule height 2pt depth -1.6pt width 23pt, {\em Equivariant stable
  homotopy theory and the {K}ervaire invariant problem}, vol.~40 of New
  Mathematical Monographs, Cambridge University Press, Cambridge, 2021.

\bibitem{Ho05}
{\sc J.~Hornbostel}, {\em {$A^1$}-representability of {H}ermitian {$K$}-theory
  and {W}itt groups}, Topology, 44 (2005), pp.~661--687.

\bibitem{HPS97}
{\sc M.~Hovey, J.~H. Palmieri, and N.~P. Strickland}, {\em Axiomatic stable
  homotopy theory}, Mem. Amer. Math. Soc., 128 (1997), pp.~x+114.

\bibitem{Hoy}
{\sc M.~Hoyois}, {\em The six operations in equivariant motivic homotopy
  theory}, Adv. Math., 305 (2017), pp.~197--279.

\bibitem{HuKO}
{\sc P.~Hu, I.~Kriz, and K.~Ormsby}, {\em The homotopy limit problem for
  {H}ermitian {K}-theory, equivariant motivic homotopy theory and motivic
  {R}eal cobordism}, Adv. Math., 228 (2011), pp.~434--480.

\bibitem{Hog}
{\sc A.~Høgenhaven}, {\em Real topological cyclic homology of spherical group
  rings}.
\newblock ArXiv Preprint 1611.01204.

\bibitem{KSW21}
{\sc M.~Karoubi, M.~Schlichting, and C.~Weibel}, {\em Grothendieck-{W}itt
  groups of some singular schemes}, Proc. Lond. Math. Soc. (3), 122 (2021),
  pp.~521--536.

\bibitem{lewis88}
{\sc L.~G. Lewis, Jr.}, {\em The {$R{\rm O}(G)$}-graded equivariant ordinary
  cohomology of complex projective spaces with linear {${\bf Z}/p$} actions},
  in Algebraic topology and transformation groups ({G}\"{o}ttingen, 1987),
  vol.~1361 of Lecture Notes in Math., Springer, Berlin, 1988, pp.~53--122.

\bibitem{LM06}
{\sc L.~G. Lewis, Jr. and M.~A. Mandell}, {\em Equivariant universal
  coefficient and {K}\"{u}nneth spectral sequences}, Proc. London Math. Soc.
  (3), 92 (2006), pp.~505--544.

\bibitem{LMS86}
{\sc L.~G. Lewis, Jr., J.~P. May, M.~Steinberger, and J.~E. McClure}, {\em
  Equivariant stable homotopy theory}, vol.~1213 of Lecture Notes in
  Mathematics, Springer-Verlag, Berlin, 1986.
\newblock With contributions by J. E. McClure.

\bibitem{Lo}
{\sc J.-L. Loday}, {\em Cyclic homology}, vol.~301 of Grundlehren der
  mathematischen Wissenschaften [Fundamental Principles of Mathematical
  Sciences], Springer-Verlag, Berlin, second~ed., 1998.
\newblock Appendix E by Mar\'{\i}a O. Ronco, Chapter 13 by the author in
  collaboration with Teimuraz Pirashvili.

\bibitem{HTT}
{\sc J.~{Lurie}}, {\em {Higher topos theory}}, vol.~170, Princeton, NJ:
  Princeton University Press, 2009.

\bibitem{HA}
\leavevmode\vrule height 2pt depth -1.6pt width 23pt, {\em {Higher algebra}}.
\newblock Unpublished book, {A}vailable at http://www.math.harvard.edu/
  $\sim$lurie/papers/HA.pdf., 2017.

\bibitem{mandell04}
{\sc M.~A. Mandell}, {\em Equivariant symmetric spectra}, in Homotopy theory:
  relations with algebraic geometry, group cohomology, and algebraic
  {$K$}-theory, vol.~346 of Contemp. Math., Amer. Math. Soc., Providence, RI,
  2004, pp.~399--452.

\bibitem{MM02}
{\sc M.~A. Mandell and J.~P. May}, {\em Equivariant orthogonal spectra and
  {$S$}-modules}, Mem. Amer. Math. Soc., 159 (2002), pp.~x+108.

\bibitem{MNN}
{\sc A.~Mathew, N.~Naumann, and J.~Noel}, {\em Nilpotence and descent in
  equivariant stable homotopy theory}, Adv. Math., 305 (2017), pp.~994--1084.

\bibitem{NS}
{\sc T.~Nikolaus and P.~Scholze}, {\em On topological cyclic homology}, Acta
  Math., 221 (2018), pp.~203--409.

\bibitem{PSW22}
{\sc I.~Patchkoria, B.~Sanders, and C.~Wimmer}, {\em The spectrum of derived
  {M}ackey functors}, Trans. Amer. Math. Soc., 375 (2022), pp.~4057--4105.

\bibitem{QS}
{\sc J.~D. Quigley and J.~Shah}, {\em On the equivalence of two theories of
  real cyclotomic spectra}.
\newblock ArXiv Preprint 2112.07462.

\bibitem{Rog09}
{\sc J.~Rognes}, {\em Topological logarithmic structures}, in New topological
  contexts for {G}alois theory and algebraic geometry ({BIRS} 2008), vol.~16 of
  Geom. Topol. Monogr., Geom. Topol. Publ., Coventry, 2009, pp.~401--544.

\bibitem{Rohrbach22}
{\sc H.~Rohrbach}, {\em The projective bundle formula for {Grothendieck}-{Witt}
  spectra}, J. Pure Appl. Algebra, 226 (2022), p.~106917.

\bibitem{Ryd13}
{\sc D.~Rydh}, {\em Existence and properties of geometric quotients}, J.
  Algebraic Geom., 22 (2013), pp.~629--669.

\bibitem{schlichting17}
{\sc M.~Schlichting}, {\em Hermitian {$K$}-theory, derived equivalences and
  {K}aroubi's fundamental theorem}, J. Pure Appl. Algebra, 221 (2017),
  pp.~1729--1844.

\bibitem{Schwede}
{\sc S.~Schwede}, {\em Lecture notes on equivariant stable homotopy theory}.
\newblock Unpublished notes. Available at
  https://www.math.uni-bonn.de/people/schwede/equivariant.pdf, 2016.

\bibitem{Seg73}
{\sc G.~Segal}, {\em Configuration-spaces and iterated loop-spaces}, Invent.
  Math., 21 (1973), pp.~213--221.

\bibitem{Shulman}
{\sc M.~G. Shulman}, {\em Equivariant local coefficients and the {RO(G)}-graded
  cohomology of classifying spaces}.
\newblock ArXiv Preprint 1405.1770.

\bibitem{Tho88}
{\sc R.~W. Thomason}, {\em Equivariant algebraic vs. topological {$K$}-homology
  {A}tiyah-{S}egal-style}, Duke Math. J., 56 (1988), pp.~589--636.

\bibitem{Ull}
{\sc J.~Ullman}, {\em Tambara functors and commutative ring spectra}.
\newblock ArXiv Preprint 1304.4912.

\bibitem{Voe10}
{\sc V.~Voevodsky}, {\em Homotopy theory of simplicial sheaves in completely
  decomposable topologies}, J. Pure Appl. Algebra, 214 (2010), pp.~1384--1398.

\bibitem{Wal79}
{\sc F.~Waldhausen}, {\em Algebraic {$K$}-theory of topological spaces. {II}},
  in Algebraic topology, {A}arhus 1978 ({P}roc. {S}ympos., {U}niv. {A}arhus,
  {A}arhus, 1978), vol.~763 of Lecture Notes in Math., Springer, Berlin, 1979,
  pp.~356--394.

\bibitem{GW91}
{\sc C.~A. Weibel and S.~C. Geller}, {\em \'{E}tale descent for {H}ochschild
  and cyclic homology}, Comment. Math. Helv., 66 (1991), pp.~368--388.

\bibitem{Xi}
{\sc H.~Xie}, {\em A transfer morphism for {Hermitian} {{\(K\)}}-theory of
  schemes with involution}, J. Pure Appl. Algebra, 224 (2020), p.~106215.

\end{thebibliography}

\begin{thebibliography}{9}

\bibitem{DMPR21add}
{\sc E.~Dotto, K.~Moi, I.~Patchkoria, and S.~P. Reeh}, {\em Real topological
  {H}ochschild homology}, J. Eur. Math. Soc. (JEMS), 23 (2021), pp.~63--152.

\bibitem{HPadd}
{\sc J.~Hornbostel and D.~Park}, {\em Real topological hochschild homology of
  perfectoid rings}.
\newblock ArXiv Preprint 2310.11183.

\bibitem{HP23add}
\leavevmode\vrule height 2pt depth -1.6pt width 23pt, {\em Real topological
  {H}oschschild homology of schemes}, JIMJ,  (2023).
\newblock doi:10.1017/S1474748023000178.

\bibitem{stacksadd}
{\sc {Stacks project authors}}, {\em The stacks project}.
\newblock \url{https://stacks.math.columbia.edu}, 2022.

\end{thebibliography}
\end{document}